\documentclass[a4paper]{amsart}
\usepackage[all]{xy}

\usepackage[colorlinks=true]{hyperref}
\usepackage{enumerate}
\usepackage{amssymb}
\usepackage{comment}
\newtheorem{thm}{Theorem}[section]
\newtheorem{prop}[thm]{Proposition}
\newtheorem{lem}[thm]{Lemma}
\newtheorem{cor}[thm]{Corollary}

\theoremstyle{definition}
\newtheorem{defn}[thm]{Definition}
\newtheorem{ex}[thm]{Example}

\theoremstyle{remark}
\newtheorem{rem}[thm]{Remark}

\def\diff{\mathsf{d}}

\begin{document}

\title{Nijenhuis forms on $L_\infty$-algebras and Poisson geometry}
\author{M. J. Azimi}
\address{CMUC, Department of Mathematics, University of Coimbra, 3001-501 Coimbra, Portugal}
\email{mjazimi2004@gmail.com}
\author{C. Laurent-Gengoux}
\address{UMR 7122, Universit\'e de Lorraine, 57045 Metz, France}
\email{camille.laurent-gengoux@univ-lorraine.fr}
\author{J. M. Nunes da Costa}
\address{CMUC, Department of Mathematics, University of Coimbra, 3001-501 Coimbra, Portugal}
\email{jmcosta@mat.uc.pt}

\begin{abstract}
We investigate Nijenhuis deformations of $L_\infty$-algebras, a notion that unifies several Nijenhuis deformations, namely those of Lie algebras, Lie algebroids, Poisson structures and Courant structures. Additional examples, linked to Lie $n$-algebras and $n$-plectic manifolds, are included.
\end{abstract}

%%% ----------------------------------------------------------------------
\maketitle
%%% ----------------------------------------------------------------------


%%%%%%%%%%%%%%%%%%%%%%%%%%%%%%%%%%%%
%%%%%%%%%%%%%%%%%%%%%%%%%%%%%%%%%%%%
%%%%%%%%%%%%%%%%%%%%%%%%%%%%%%%%%%%%
\section*{Introduction}             %
\label{sec:introduction}           %
%%%%%%%%%%%%%%%%%%%%%%%%%%%%%%%%%%%%
%%%%%%%%%%%%%%%%%%%%%%%%%%%%%%%%%%%%
%%%%%%%%%%%%%%%%%%%%%%%%%%%%%%%%%%%%

$L_\infty$-algebras, introduced by Lada and Stasheff \cite{Lada-Sta}, who called them
 strongly homotopy Lie algebras, are collections of $n$-ary operations, assumed to satisfy some homogeneous relations that reduce to the Jacobi identity when only the binary operation is not trivial. These structures gained notoriety when Kontsevitch
used $L_\infty$-morphisms to prove the existence of star-products on Poisson manifolds \cite{Kont}.
derived an $L_\infty$-algebra
from a Poisson element and an abelian subalgebra of a differential graded Lie algebra.
For instance, an $L_\infty$-algebra encodes a Poisson structure in a
neighborhood of a coisotropic submanifold, provided that a linear transversal is given, see \cite{Cattaneo-Schaaetz} and \cite{Cattaneo-Felder}.
This makes $L_\infty$-algebras a central tool for studying Poisson brackets, but there are more occurences.
 Roytenberg and  Weinstein \cite{RoytenbergWeinstein} gave a description of the so-called Courant algebroids
 in terms of Lie $2$-algebras.
 In  the same vein, Rogers \cite{CRoger} encodes $n$-plectic manifolds by Lie $n$-algebras
and Fr\'{e}gier, Rogers and Zambon \cite{FRZ} used this formalism to construct moment maps of those.

In this paper we develop a theory of Nijenhuis forms on $L_\infty$-algebras. Here, by Nijenhuis forms, we mean a generalization of
the notion of Nijenhuis $(1,1)$-tensors on manifolds, i.e., $(1,1)$-tensors whose Nijenhuis torsion vanishes.
On manifolds, Nijenhuis tensors are unary operations on  the Lie algebra of vector fields.
Since, when dealing with $L_\infty$-algebras, one has to replace Lie algebra brackets by collections of $n$-ary brackets for all integers $n \geq1$,
we also want to define Nijenhuis forms that are collections of $n$-ary operations for all integers $n \geq1$.
Our main idea is based on the fact that,
given a Lie algebra $({\mathfrak g},[.,.])$
and a linear endomorphism $N$ of ${\mathfrak g}$, $N$ is Nijenhuis if deforming twice by $N$ the original bracket yields the original bracket deformed by $N^2$.
 We translate this idea to $L_\infty$-algebras, where the brackets to be deformed are their $n$-ary brackets.

We present several examples of Nijenhuis forms on $L_\infty$-algebras. The first example is universal, in the sense that every $L_\infty$-structure admits
 it: the Euler map $S$, that multiplies an element by its degree.
 Nijenhuis operators on ordinary graded Lie algebras are among the most trivial examples.
 Poisson elements, and more generally, Maurer-Cartan elements of differential graded Lie algebras
 are also examples, which are not purely made of vector valued $1$-forms, but which are the sum of a vector valued $1$-form with a vector valued $0$-form.  Less trivial examples are given on Lie $n$-algebras.
 On those, we have Nijenhuis forms which are the sum of a family of vector valued $k$-forms. An interesting case is when the Lie $n$-algebra is associated to an $n$-plectic manifold \cite{CRoger}. The case of Lie $2$-algebras is treated separately, and we have Nijenhuis forms which are the sum of a vector valued $1$-form with a vector valued $2$-form.

 We discuss how  Nijenhuis tensors on Courant algebroids \cite{CGM, YKSBrazil, CJP, AntunesCosta} fit in our defintion of Nijenhuis forms on some $L_\infty$-algebras. In order to include Lie algebroids in our examples, we recall the concept of multiplicative $L_\infty$-algebras (related to $P_\infty$-algebras in \cite{Cattaneo-Felder}). In the last part of the paper, our examples come from well-known structures on Lie algebroids, defined by pairs of compatible tensors \cite{KoRou, antunes, AntunesCosta}, such as $\Omega N$-, Poisson-Nijenhuis \cite{YKS} and $P \Omega$-structures.

Very recently, while we were about to finish this paper, a notion of Nijenhuis operator on Lie $2$-algebras was introduced in \cite{LiuSheng}, using a different perspective. That definition is a particular case of ours, as we explain in Remark \ref{sheng}.

The paper is organized in seven sections. In Section 1 we introduce a bracket of graded symmetric vector valued forms on a graded vector space that we call Richardson-Nijenhuis bracket, because it reduces to the usual Richardson-Nijenhuis bracket of vector valued forms on a (non-graded) vector space. With this graded bracket, we characterize $L_\infty$-structures as Poisson elements on the graded Lie algebra of graded symmetric vector valued forms. In Section 2, we present our main definition of Nijenhuis vector valued form with respect to an $L_\infty$-algebra, or more generally, with respect to a vector valued form of degree $1$. Relaxing a bit the definition of Nijenhuis vector valued form, yields the notions of weak Nijenhuis and co-boundary Nijenhuis forms, which provide interesting examples to be discussed in the next sections. Section 2 also contains the first examples of Nijenhuis forms on symmetric graded Lie algebras and symmetric differential graded Lie algebras: the Euler map, Poisson and Maurer-Cartan elements. Section 3 is devoted to Nijenhuis forms on Lie $n$-algebras. We construct examples of Nijenhuis forms on general Lie $n$-algebras, in particular on those defined by $n$-plectic manifolds. The case $n=2$ is treated separately, in Section 4. There, we find necessary and sufficient conditions to have a Nijenhuis form which is the sum of a vector valued $1$-form with a vector valued $2$-form which is, in fact, the most general case. The importance of Lie $2$-algebras appears in Section 5, where we focus on Courant algebroids. Using a construction established in \cite{RoytenbergWeinstein}, we associate a Lie $2$-algebra to each Courant algebroid and we relate $(1,1)$-tensors with vanishing Nijenhuis torsion on a Courant algebroid, with Nijenhuis forms on the corresponding associated Lie $2$-algebra. In Section 6, we study multiplicative $L_\infty$-algebras and its relation with pre-Lie and Lie algebroids. We introduce the notions of extension by derivation of $(1,1)$-tensors and of $k$-forms on a Lie algebroid,  needed to construct examples of Nijenhuis forms on Lie algebroids in the last section. In Section 7, the last one, we obtain, out of $\Omega N$-, Poisson-Nijenhuis and $P \Omega$-structures on a Lie algebroid, examples of weak Nijenhuis and co-boundary Nijenhuis vector valued forms.

\section{Richardson-Nijenhuis bracket and $L_\infty$-algebras}

In this section we extend the usual Richardson-Nijenhuis bracket of vector valued forms on vector spaces \cite{KMS} to graded symmetric vector valued forms on graded vector spaces. Then, we use it to characterize $L_\infty$-structures on graded vector spaces. We start by fixing some notations on graded vector spaces.

Let $E$ be a graded vector space over a field $\mathbb{K}=\mathbb{R}$ or $\mathbb{C}$, that is, a vector space of the form
$ \oplus_{i \in {\mathbb Z}} E_i. $
For a given $i\in \mathbb{Z}$, the vector space $E_i$ is called the component of degree $i$,
elements of $E_i$ are called \emph{homogeneous elements of degree $i$}, and elements in the union $\cup_{i\in \mathbb{Z}} E_i$
are called the \emph{homogeneous elements}.
We denote by $|X|$ the degree of a non-zero homogeneous element $X$. Given a graded vector space $E=\oplus_{i\in \mathbb{Z}}E_i$ and an integer $p$, one may shift all the degrees
by $p$ to get a new grading on the vector space $E$. We use the notation $E[p]$ for the graded vector space $E$ after shifting the degrees
 by $p$, that is, the graded vector space whose  component of degree $i$ is $E_{i+p}$.

  We denote by $S(E)$ the symmetric space of $E$ which is, by definition, the quotient space of the tensor algebra $\otimes E$ by the two-sided ideal $I \subset \otimes E$  generated by elements of the type $X\otimes Y-(-1)^{|X||Y|}Y\otimes X$, with $X$ and $Y$ arbitrary homogeneous elements in $E$.
For a given $k\geq 0$, $S^k(E)$ is the image of $\otimes^{k} E$ through the quotient map $\otimes E \mapsto \frac{\otimes E}{I} = S(E)$ and one has the following
 decomposition
   \begin{equation*} S(E) = \oplus_{k \geq 0} S^k (E), \end{equation*}
where $S^0(E)$ is simply the field $\mathbb{K}$.
  Moreover, when all the components in the graded space $E$ are of finite dimension, the dual of $S^k(E)$ is isomorphic to $S^k(E^*)$, for all $k\geq 0$. In this case,  there is a one to one correspondence between
\begin{enumerate}
\item[(i)] graded symmetric $k$-linear maps on the graded vector space $E$,
\item[(ii)] linear maps from the space $S^k(E)$ to $E$,
 \item[(iii)] $S^k(E^*)\otimes E$.
\end{enumerate}
Elements of the space $S^k(E^*)\otimes E$ are called \emph{symmetric vector valued  $k$-forms}. Notice that $S^0(E^*)\otimes E$, the space of vector valued zero-forms, is isomorphic to the space $E$.

 Having the decomposition $S(E)=\oplus_{k\geq 0}S^k(E)$, every element in $S(E)$ is the sum of finitely many elements in $S^k(E)$, $k \geq 0$.
 We absolutely need to consider also infinite sums, which is often referred in the literature as
 taking the completion of $S(E)$. By a \emph{formal sum}, we mean a sequence  $\phi: \mathbb{N}\bigcup \{0\} \to S(E) $ mapping an integer $k$
 to an element $a_k \in S^k(E)$: we shall, by a slight abuse of notation, denote by $\sum_{k=0}^\infty a_k$
 such an element. We denote the set of all formal sums by $\tilde{S}(E)$.  The algebra structure on $S(E)$ extends in an unique manner to $\tilde{S}(E)$.
 For two formal sums $a=\sum_{k=0}^\infty a_k$ and $b=\sum_{k=0}^\infty b_k$ we define $a+b$ to be $\sum_{k=0}^\infty (a_k+b_k)$,
 while the product of $a$ and $b$ is the infinite sum $\sum_{k=0}^\infty c_k$ with
   $ c_k = \sum_{i=0}^k a_i \cdot b_{k-i}  $ (with $\cdot$ being the product of $S(E)$).

When all the components in the graded space $E$ are of finite dimension, there is a one to one correspondence between
\begin{enumerate}
\item[(i)] collections indexed by ${k\geq 0}$ of graded symmetric $k$-linear maps on the graded vector space $E$,
\item[(ii)] collections  indexed by ${k\geq 0}$ of linear maps from $S^k(E)$ to $E$,
 \item[(iii)] $\tilde{S}(E^*)\otimes E$.
\end{enumerate}
Elements of the space $\tilde{S}(E^*)\otimes E$ are called \emph{ symmetric vector valued forms} and shall be written as infinite sums $\sum K_i$ with $K_i \in S^i(E^*)\otimes E$.

Let $E$ be a graded vector space, $E= \oplus_{i\in \mathbb{Z}}E_i$. The insertion operator of a symmetric vector valued $k$-form $K$ is an operator
$$ \iota_K : S(E^*)\otimes E \to S(E^*)\otimes E$$
defined by
\begin{equation}\label{equ:insertion}
\iota_{K}L(X_1,...,X_{k+l-1})=\sum_{\sigma \in Sh(k,l-1)}\epsilon(\sigma)L(K(X_{\sigma(1)},...,X_{\sigma(k)}),...,X_{\sigma(k+l-1)}),
\end{equation}
for all $L \in S^{l}(E^*)\otimes E $, $l\geq 0$ and $X_1,\cdots, X_{k+l-1} \in E$, where $Sh(i,j-1)$ stands for the set of $(i,j-1)$-unshuffles and
$\epsilon(\sigma)$ is the Koszul sign which
is defined as follows
   \begin{equation*}
   X_{\sigma(1)}\otimes\cdots \otimes X_{\sigma(n)}=\epsilon(\sigma)X_{1}\otimes\cdots \otimes X_{n},
   \end{equation*}
for all $X_{1}, \cdots, X_{n} \in E.$
If $L$ is an element in $S^0(E^*)\otimes E \simeq E$, then (\ref{equ:insertion}) should be understood as meaning that
$\iota_{K}L=0$, for all vector valued forms $K$ and
\begin{equation*}\label{eq:insertionof0}
\iota_{L}K (X_1,...,X_{k-1})=K(L,X_1,...,X_{k-1}),
\end{equation*}
for all vector valued $k$-form $K$.

Allowing $L$ and $K$ to be  symmetric vector valued forms, that is, $L=\sum_{i\geq 0}L_i$ and $K=\sum_{i\geq 0}K_i$, with $L_i$ and $K_i$ vector valued $i$-forms, the previous definition of insertion operator extends in the obvious way.

If $K$ is an element in $S^i(E^*)$, i.e. a linear form on $S^i(E) $, $i \geq 0$, one may define $\iota_K$ by a formula similar to (\ref{equ:insertion}). Moreover,
 $ \iota_K: \tilde{S}(E^*) \to \tilde{S}(E^*) $, with $K \in \tilde{S}(E^*)\otimes E$,
is the zero map if and only if $K=0.$

Now, we define a bracket on the space $\tilde{S}(E^*) \otimes E$ as follows.
Given a symmetric vector valued $k$-form $K\in S^{k}(E^*)\otimes E$ and a symmetric vector valued $l$-form $L\in S^{l}(E^*)\otimes E$, the {\em Richardson-Nijenhuis bracket} of $K$ and $L$ is the symmetric vector valued $(k+l-1)$-form $[K,L]_{_{RN}}$, given by
\begin{equation} \label{RNbracket}
[K,L]_{_{RN}}=\iota_{K}L-(-1)^{\bar{K}\bar{L}}\iota_{L}K,
\end{equation}
where $\bar{K}$ is the degree of $K$ as a graded map, that is $K(X_1,\cdots, X_k)\in E_{1+\cdots+k+\bar{K}},$ for all $X_i \in E_i$.
For an element $X \in E$, $\bar{X}=|X|$, that is, the degree of a vector valued $0$-form, as a graded map, is just its degree as an element of $E$.

\begin{prop}
The space $\tilde{S}(E^*) \otimes E$, equipped with the Richardson-Nijenhuis bracket, is a graded (skew-symmetric) Lie algebra.
\end{prop}
If $K\in S^k(E^*)\otimes E$ is a vector valued $k$-form, an easy computation gives
\begin{equation} \label{kform}
K(X_1,\cdots,X_k)=[X_k,\cdots,[X_2,[X_1,K]_{_{RN}}]_{_{RN}}\cdots]_{_{RN}},
\end{equation}
for all $X_1,\cdots,X_k\in E$.

\

In \cite{LMS}, the authors defined a multi-graded
Richardson-Nijenhuis bracket, in a graded vector space, but their approach is different from
ours.

\

Next, we recall the notion of $L_\infty$-algebra, following \cite{Getzler}.
\begin{defn}\label{def:SymLinfty}
An {\em $L_{\infty}$-algebra} is a graded vector space $E= \oplus_{i\in \mathbb{Z}}E_i$ together with a family of symmetric vector valued forms $(l_i)_{i\geq 1}$ of degree $1$, with $l_i: \otimes^i E\to E$ satisfying the following relation:
\begin{equation}\label{L_infty-def}
\sum_{i+j=n+1}\sum_{\sigma \in Sh(i,j-1)}\epsilon(\sigma)l_j(l_i(X_{\sigma(1)},\cdots,X_{\sigma(i)}),\cdots,X_{\sigma(n)})=0,
\end{equation}
for all $n\geq 1$ and all homogeneous  $X_1,\cdots, X_n \in E$, where $\epsilon(\sigma)$ is the Koszul sign. The family of symmetric vector valued forms $(l_{i})_{i\geq 1}$ is called an {\em $L_{\infty}$-structure} on the graded vector space $E$. Usually, we denote this $L_{\infty}$-structure by $\mu:=\sum_{i\geq 1} l_i$ and we say, by an abuse of language,  that $\mu$ has degree $1$.
\end{defn}

A slight generalization of an $L_{\infty}$-algebra is the so-called {\em curved} $L_{\infty}$-algebra. In this case, the family of symmetric vector valued forms is $(l_{i})_{i\geq 0}$ that is, there is an extra symmetric vector valued $0$-form $l_0 \in E_1$, called the {\em curvature}, such that  $l_1(l_0)=0$ and Equation (\ref{L_infty-def}) is replaced by
\begin{equation*}
l_{n+1}(l_0,X_1,\cdots, X_{n})+\sum_{i+j=n+1}\sum_{\sigma \in Sh(i,j-1)} \epsilon(\sigma) l_j(l_i(X_{\sigma(1)},\cdots,X_{\sigma(i)}),\cdots,X_{\sigma(n)})
=0.
\end{equation*}

There is an equivalent definition of $L_{\infty}$-algebra in terms of graded skew-symmetric vector valued forms $l'_i$  of degree $i-2$. This was, in fact, the original definition introduced in \cite{Lada-Sta}. The equivalence of both definitions is established by the so-called {\em d\'ecalage isomorphism} \begin{equation*}\label{def:L_inftysymmetric}
   l_i(X_1,\cdots,X_i)\mapsto (-1)^{(i-1)|X_1|+(i-2)|X_2|+\cdots+|X_{i-1}|}l'_i(X_1,\cdots,X_i),
   \end{equation*}
$  X_1,\cdots,X_i \in E$. The family of  graded skew-symmetric brackets $(l'_{i})_{i\geq 1}$ defines an $L_{\infty}$-structure on the graded vector space $E[-1]$ if each $l'_{i}$ has degree $i-2$ and
\begin{equation*}\label{L_infty-def-skew}
\sum_{i+j=n+1}\sum_{\sigma \in Sh(i,j-1)}(-1)^{i(j-1)}\epsilon(\sigma)\, sign(\sigma)l_j(l_i(X_{\sigma(1)},\cdots,X_{\sigma(i)}),\cdots,X_{\sigma(n)})=0,
\end{equation*}
for all $n\geq 1$ and all $X_1,\cdots,X_n \in E$, with $sign(\sigma)$ being the sign of the permutation $\sigma$.

 Next, we see that some well-known structures on (graded) vector  spaces are examples of $L_{\infty}$-algebras.

   We start with a \emph{symmetric graded  Lie algebra}, which is a graded vector space $E=\oplus_{i\in \mathbb{Z}}E_i$ endowed with a binary graded symmetric bracket $[.,.]= \mu$ of degree $1$, satisfying the graded Jacobi identity i.e.
\begin{equation}\label{JacobiLie}
[X,[Y,Z]]=(-1)^{|X|+1}[[X,Y],Z]+(-1)^{(|X|+1)(|Y|+1)}[Y,[X,Z]],
\end{equation}
for all homogeneous elements $X,Y,Z \in E$.
Note that when the graded vector space is concentrated on degree $-1$, that is, all the vector spaces $E_i$ are zero, except $E_{-1}$, then (\ref{JacobiLie}) is the usual Jacobi identity and we get a Lie algebra with symmetric bracket. We would like to remark that (\ref{JacobiLie}) can be written as
\begin{equation}  \label{gla}
\mu(\mu(X,Y),Z)+(-1)^{|Y||Z|}\mu(\mu(X,Z),Y)+(-1)^{|X|(|Y|+|Z|)}\mu(\mu(Y,Z),X)=0,
\end{equation}
for all homogeneous elements $X,Y,Z \in E$.
This means that a symmetric graded Lie algebra is simply an $L_{\infty}$-algebra such that all the multi-brackets are zero except the binary one. From this, we also conclude that a {\em Lie algebra} is an $L_{\infty}$-algebra on a graded vector space concentrated on degree $-1$, for which all the brackets are zero except the binary bracket.

Another special case of an $L_\infty$-algebra is a \emph{symmetric differential graded Lie algebra}. It is an $L_\infty$-structure on $E=\oplus_{i\in \mathbb{Z}} E_i$, with all the brackets, except $l_1$ and $l_2$, being zero. In other words, a symmetric differential graded Lie algebra is a symmetric graded Lie algebra $(\oplus_{i\in \mathbb{Z}} E_i,[.,.]=l_2)$ endowed with a differential $\diff=l_1$, that is, a linear map $\diff :\oplus_{i\in \mathbb{Z}} E_i \to \oplus_{i\in \mathbb{Z}} E_i $ of degree $1$ and squaring to zero, satisfying the compatibility condition
\begin{equation*}\label{def:DGLA}
\diff[X,Y]+[\diff(X),Y]+(-1)^{|X|}[X,\diff(Y)]=0,
\end{equation*}
for all homogeneous elements $X,Y \in E.$ We shall denote a symmetric differential graded Lie algebra by $(E, \diff, [.,.])$ or by $(E, l_1+l_2)$.

We may also consider  two particular cases of a curved $L_\infty$-algebra, that is to say, a curved symmetric graded Lie algebra and a curved symmetric differential graded Lie algebra.  More precisely,
a {\em curved} symmetric differential graded Lie algebra on a graded vector space $E=\oplus_{i\in \mathbb{Z}} E_i$ is a symmetric differential graded Lie algebra $(E, \diff, [.,.])$ together with an element ${\mathfrak C} \in E_{1}$ such that:
\begin{equation*}
\diff({\mathfrak C})=0 \quad \mbox{ \rm and} \quad [{\mathfrak C},X]+\diff^2X=0, \,\,\, \mbox{ \rm for all} \,\,\, X \in E.
\end{equation*}
 We shall denote the curved symmetric differential graded Lie algebra by $(E, {\mathfrak C}, \diff, [.,.])$ or by $(E, {\mathfrak C}+l_1+l_2)$. When $\diff=0$, the curved symmetric differential graded Lie algebra  is simply a curved symmetric graded Lie algebra.

The Richardson-Nijenhuis bracket on graded vector spaces, introduced previously, is intimately related to $L_{\infty}$-algebras. In the next theorem, that appears in an implicit form in \cite{Roytenberg}, we use the Richardson-Nijenhuis bracket to characterize a (curved) $L_{\infty}$-structure on a graded vector space $E$ as a homological vector field on $E$.

\begin{thm}\label{th:linftyRN}
Let $E=\oplus_{i\in \mathbb{Z}}E_i$ be a graded vector space,
 $(l_i)_{i \geq 1}: \otimes^{i}E \to E$ be a family of symmetric vector valued forms on $E$ of degree $1$ and $l_0 \in E_1$ be a symmetric vector valued $0$-form.
 Set $\mu= \sum_{i\geq 1} l_i$ and $\mu'= \sum_{i\geq 0} l_i$. Then,
 \begin{itemize}
 \item[i)]
 $\mu$ is an $L_{\infty}$-structure on $E$ if and only if $[\mu,\mu]_{_{RN}}=0$;
 \item[ii)]
 $\mu'$ is a curved $L_{\infty}$-structure on $E$ if and only if $[\mu',\mu']_{_{RN}}=0$.
 \end{itemize}
\end{thm}
\begin{proof} (i) It is a direct consequence of the following equalities that can be obtained from (\ref{equ:insertion}) and (\ref{RNbracket}):
\begin{equation*}\label{mumu}
[\mu,\mu]_{_{RN}}=\sum_{n\geq 1}(\sum_{i+j=n+1}[l_i,l_j]_{_{RN}})=2\sum_{n\geq 1}(\sum_{i+j=n+1}\iota_{l_i}l_j).
\end{equation*}
The proof of (ii) is easy.
\end{proof}

Notice that for the case of symmetric graded Lie algebras, the statement of Theorem~\ref{th:linftyRN} appears in a natural way, since equation (\ref{gla}) is equivalent to
\begin{equation*}
(\iota_{\mu}\mu+\iota_{\mu}\mu)(X,Y,Z)=[\mu,\mu]_{_{RN}}(X,Y,Z)=0.
\end{equation*}

\section{Nijenhuis forms on $L_\infty$-algebras: definition and first examples}

In this section we define a Nijenhuis vector valued form with respect to a given vector valued form $\mu$ and deformation of $\mu$ by a Nijenhuis vector valued form. We show that deforming an $L_{\infty}$-structure by a Nijenhuis vector valued form, one gets an $L_{\infty}$-structure. Then, we present the first examples of Nijenhuis vector valued forms on some $L_{\infty}$-algebras.
\begin{defn} \label{def:Nijenhuis}
Let $E$ be a graded vector space and $\mu$ be a symmetric vector valued form on $E$ of degree $1$. A vector valued form ${\mathcal N}$ of degree zero is called
\begin{itemize}
  \item {\em weak Nijenhuis} with respect to $\mu$ if
       \begin{equation*}\label{weakN}
       \left[\mu,\left[{\mathcal N},\left[{\mathcal N},\mu\right]_{_{RN}}\right]_{_{RN}}\right]_{_{RN}}=0,
       \end{equation*}
  \item {\em co-boundary Nijenhuis} with respect to $\mu$ if there exists a vector valued form ${\mathcal K}$ of degree zero, such that
   \begin{equation*}\label{boundN}
       \left[{\mathcal N},\left[{\mathcal N},\mu\right]_{_{RN}}\right]_{_{RN}}=\left[{\mathcal K},\mu\right]_{_{RN}},
       \end{equation*}
  \item {\em Nijenhuis} with respect to $\mu$  if there exists a vector valued form ${\mathcal K}$ of degree zero, such that
       \begin{equation*}\label{strongN}
       \left[{\mathcal N},\left[{\mathcal N},\mu\right]_{_{RN}}\right]_{_{RN}}=\left[{\mathcal K},\mu\right]_{_{RN}}\,\,\,\hbox{and }\,\,\,\left[{\mathcal N},{\mathcal K}\right]_{_{RN}}=0.
       \end{equation*}
  Such a ${\mathcal K} $ is called a {\em square} of ${\mathcal N} $. If ${\mathcal N}$ contains an element of the underlying graded vector space, that is, ${\mathcal N}$ has a component which is a vector valued zero form, then ${\mathcal N}$ is called Nijenhuis (respectively, co-boundary Nijenhuis) vector valued form with {\em curvature}.
\end{itemize}
\end{defn}

It is obvious that the following implications hold:

\begin{center}
${\mathcal N}$ Nijenhuis $\Rightarrow$ ${\mathcal N}$ co-boundary Nijenhuis $\Rightarrow$ ${\mathcal N}$ weak Nijenhuis
\end{center}

\begin{rem}
It would be of course tempting to choose ${\mathcal K} = \iota_{\mathcal N}{\mathcal N} $ in Defintion~\ref{def:Nijenhuis}, having in mind what happens for manifolds, and the fact
that  $\iota_{ {\mathcal N}}{{\mathcal N}} = {\mathcal N}^2 $ for vector valued $1$-forms.
However, it is not what examples show to be a reasonable definition. Also, for ${\mathcal N}$  a vector valued $2$-form we do not have, in general,
  $ [\iota_{\mathcal N} {\mathcal N} , {\mathcal N}]_{_{RN}} =0,$
 which says $\iota_{\mathcal N} {\mathcal N}  $ is not a good candidate for the square, except maybe for vector valued $1$-forms.
\end{rem}

\begin{prop}\label{prop:deformedisLinfty}
Let $(E,\mu)$ be a (curved)  $L_\infty$-algebra and ${\mathcal N}$ be a symmetric vector valued form on $E$ of degree zero. Then ${\mathcal N}$ is weak Nijenhuis with respect to $\mu$ if and only if $[{\mathcal N},\mu]_{_{RN}}$ is a (curved) $L_{\infty}$-algebra.
\end{prop}
\begin{proof}
Using the Jacobi identity,
 we get
\begin{eqnarray*}\label{k1}
[\mu,[{\mathcal N},[{\mathcal N},\mu]_{_{RN}}]_{_{RN}}]_{_{RN}}& = &[[\mu,{\mathcal N}]_{_{RN}},[\mu,{\mathcal N}]_{_{RN}}]_{_{RN}}+[{\mathcal N},[\mu,[{\mathcal N},\mu]_{_{RN}}]_{_{RN}}]_{_{RN}}\\
& = & [[\mu,{\mathcal N}]_{_{RN}},[\mu,{\mathcal N}]_{_{RN}}]_{_{RN}}\\
& = & [[{\mathcal N}, \mu]_{_{RN}},[{\mathcal N}, \mu]_{_{RN}}]_{_{RN}},
\end{eqnarray*}
which concludes the proof.
\end{proof}

Given an $L_\infty$-structure $\mu$ and a symmetric vector valued form of degree zero ${\mathcal N}$ on a graded vector space, we call $[{\mathcal N}, \mu]_{_{RN}}$ the {\em deformation of $\mu$ by ${\mathcal N}$} and denote the deformed structure by $\mu^{{\mathcal N}}$. When $\mu$ is deformed $k$ times by ${\mathcal N}$, the deformed structure is denoted by $\mu^{{\mathcal N}, {\stackrel{k}{\dots}}, \, {\mathcal N}}$ or simply $\mu_k$ if there is no danger of confusion.

\

Weak Nijenhuis forms do not, in general, give hierarchies in any sense.
However, Nijenhuis forms do.
\begin{thm}\label{theo:Hierarchy}
Let ${\mathcal N}$ be a Nijenhuis vector valued form with respect to a  (curved) $L_\infty$-structure ${\mu}$ with square ${\mathcal K}$, on a graded vector space $E$. Then,
\begin{itemize}
\item[(i)]
for all integers $k \geq 1$, $\mu_k$  is a (curved) $L_\infty$-structure on $E$ and ${\mathcal N}$ is Nijenhuis with square ${\mathcal K}$, with respect to $\mu_k$;
\item[(ii)]
 for all integers $k, l \geq 1$, $[\mu_k, \mu_l]_{_{RN}}=0$.
 \end{itemize}
\end{thm}

\begin{proof}
\begin{itemize}
\item[(i)]
For the case $k=1$, Proposition~\ref{prop:deformedisLinfty} together with the observation that if ${\mathcal N}$ is Nijenhuis then it is also weak Nijenhuis with respect to $\mu$, yields that $\mu_1= [{\mathcal N},\mu]_{_{RN}}$ is a (curved) $L_\infty$-algebra. Since ${\mathcal N}$ is Nijenhuis with respect to $\mu$ with square $\mathcal K$, we have
\begin{equation}\label{J1}
[{\mathcal N},[{\mathcal N},[{\mathcal N},\mu]_{_{RN}}]_{_{RN}}]_{_{RN}}=[{\mathcal N},[{\mathcal K},\mu]_{_{RN}}]_{_{RN}}.
\end{equation}
Applying the Jacobi identity on the right hand side of (\ref{J1}) and using the assumption that  ${\mathcal N}$ and  ${\mathcal K}$ commute with respect to the Richardson-Nijenhuis bracket, we get
\begin{equation*}
[{\mathcal N},[{\mathcal N},\mu_{1}]_{_{RN}}]_{_{RN}}=[{\mathcal K},\mu_{1}]_{_{RN}}.
\end{equation*}
Thus, ${\mathcal N}$ is Nijenhuis with respect to $\mu_{1}$, with square ${\mathcal K}$.

Assume, by induction, that ${\mathcal N}$ is Nijenhuis with respect to $\mu_k$ with square ${\mathcal K}$. Then, we have
\begin{equation*}
[{\mathcal N},[{\mathcal N},[{\mathcal N},\mu_k]_{_{RN}}]_{_{RN}}]_{_{RN}}=[{\mathcal N},[{\mathcal K},\mu_k]_{_{RN}}]_{_{RN}},
\end{equation*}
or, equivalently,
\begin{equation*}
[{\mathcal N},[{\mathcal N},\mu_{k+1}]_{_{RN}}]_{_{RN}}=[{\mathcal K},\mu_{k+1}]_{_{RN}}
\end{equation*}
which shows that ${\mathcal N}$ is Nijenhuis with respect to $\mu_{k+1}$, with square ${\mathcal K}$.

Assuming, by induction, that $\mu_k$ is a (curved) $L_\infty$-structure, i.e. $[\mu_k,\mu_k]_{_{RN}}=0$, we get  $$[[{\mathcal N},\mu_k]_{_{RN}}, [{\mathcal N},\mu_k]_{_{RN}}]_{_{RN}}=0$$ by using  the Jacobi identity, which means that $\mu_{k+1}$ is a (curved) $L_\infty$-structure. Thus, $\mu_k$ is a (curved) $L_\infty$-structure, for all $k\geq 1$.
\item[(ii)]
Let $k$ and $l$ be two positive integers with $k \geq l$. The case $k=l$ was proved in i). For the case $k > l$, assume by induction that $[\mu_k, \mu_n]_{_{RN}}=0$ for all integers $k \geq n \geq l$.
 Then,  the Jacobi identity gives
 \begin{eqnarray}
 [\mu_{k+1}, \mu_l]_{_{RN}}&=&[[{\mathcal N},\mu_k]_{_{RN}}, \mu_l]_{_{RN}} \nonumber \\
 &=&[{\mathcal N},[\mu_k, \mu_l]_{_{RN}}]_{_{RN}}-[\mu_k, [{\mathcal N},\mu_l]_{_{RN}}]_{_{RN}} \label{mu_kmu_l}
  \end{eqnarray}
  and by the induction assumption, both terms in (\ref{mu_kmu_l}) vanish. So, we get $[\mu_{k+1}, \mu_{l}]_{_{RN}}=0$.
  This completes the induction and shows that $[\mu_k, \mu_l]_{_{RN}}=0$, for all $k,l \geq 1$.
\end{itemize}
\end{proof}

\begin{rem}
Theorem \ref{theo:Hierarchy} implies that:
 \begin{equation}\label{def:mu_de_t} \mu(t) = \sum_{i=0}^\infty \frac{t^i}{i!} \mu_i ,\end{equation}
with $t$ a formal parameter, satisfies the relation:
 \begin{equation}\label{eq:all_in_one} {\displaystyle{\left[ \frac{d^k}{dt^k} \mu(t), \frac{d^l}{dt^l}  \mu(t)\right]_{_{RN}}=0}}, \end{equation}
for all pair of positive integers $k,l$.
%Equation (\ref{eq:all_in_one}) could also be stated as:
%$$ [\mu(s), \mu(t)]_{_{RN}}=0, $$
%with $s,t$ formal parameters and $\mu(s),\mu(t)$ as in (\ref{def:mu_de_t}).
Let us prove this point:  (\ref{def:mu_de_t}) implies that $$\frac{d^k}{dt^k} \mu(t)=\sum_{i=0}^\infty \frac{t^i}{i!} \mu_{k+i}.$$
 Hence, $${\displaystyle{\left[ \frac{d^k}{dt^k} \mu(t), \frac{d^l}{dt^l}  \mu(t)\right]_{_{RN}}}}=\sum_{i,j=0}^\infty \frac{t^{i+j}}{i!j!} \left[ \mu_{k+i}, \mu_{l+j}\right]_{_{RN}}.$$
Thus, (\ref{eq:all_in_one}) holds for all pair of positive integers $k, l$ if item (ii) in Theorem \ref{theo:Hierarchy} holds.
 %as follows from the relation:
%$$ [ \mu_k,\mu_l ]_{_{RN}} = \left[ \left. \frac{d^k}{dt^k}\right|_{t=0} \mu(t), \left. \frac{d^l}{dt^l}\right|_{t=0} \mu(t)\right]_{_{RN}},$$
%together with the obvious fact that the terms of a given degree $n$ in left hand side of (\ref{eq:all_in_one}) is a linear
%combination of the quantities
%$[ \mu_i,\mu_j ]_{_{RN}} $ with $ i+j=n+k+l$.

Now, recall from \cite{Markl} that a symmetric vector valued form $ \nu $ extends in a unique manner to a coderivation $\tilde{\nu} $ of the symmetric algebra $S(E)$. The Richardson-Nijenhuis bracket can be seen as being the graded commutator of coderivations of $S(E)$: said otherwise, the relation $ \widetilde{[\mu,\nu]_{_{RN}}} = [ \tilde{\mu}, \tilde{\nu} ] $ holds. In particular, for $\mu$ a $L_\infty$-structure, the coderivation $\tilde{\nu} $ squares to zero.

Recall also that given two $L_\infty$-structures $\mu$ and $ \nu$, a coalgebra endomorphism $\Phi $ of $S(E)$ that satisfies $ \Phi \circ \tilde{\mu} = \tilde{\nu} \circ \Phi $ is called an $L_{\infty}$-morphisms from $ \mu$ to $\nu$.
When $\Phi$ is invertible, we speak of $L_\infty $-isomorphism.
These facts extend to the algebra $S(E) \otimes {\mathbb K}[[t]]$, with $ {\mathbb K} $ being the base field
and $t$ a formal parameter.

By definition of $ \mu_k$, $\mu(t)$ can be expressed as
$$ \mu(t) = e^{ad_{t {\mathcal N}}}  (\mu),$$
where $e$ stands for the exponential (which makes sense since $t$ is a formal parameter) and $ad$ refers to the adjoint action with respect to the Richardson-Nijenhuis bracket. By usual abstract non-sense, the previous relation implies:
$$ \widetilde{\mu(t)} = e^{ t\widetilde{{\mathcal N}}} \circ \widetilde{\mu} \circ e^{-t\widetilde{{\mathcal N}}} $$
where $\widetilde{\mu}$ is the coderivation of the symmetric algebra $ S(E) $ associated to $ \mu $.
Since $e^{ t\widetilde{{\mathcal N}}} $ is an invertible coalgebra morphism, the $L_{\infty} $-structures $ \mu(t) $
and $\mu$ are $L_\infty$-isomorphic. This point should be interpreted as meaning that
$ \mu_1 =  \left.\frac{d}{dt}\right|_{t=0} \mu(t)$ is a trivial deformation
of the $L_\infty $-structure $ \mu$. As usual in deformation theory, trivial deformations are defined as being those  which are the first derivatives at $ t=0$ of a parameter depending deformation $\mu(t) $ of a  $ L_\infty$-structure $ \mu$ such that $ \mu(t)$ is $L_\infty $-isomorphic to $\mu$. This is parallel to what happens for Nijenhuis deformations of Poisson structures. Notice, however, that $ \mu_2$ is not, a priori, a trivial deformation of $\mu$.
\end{rem}

 Recall from \cite{YKS} that a \emph{Nijenhuis} operator on a graded Lie algebra $(E,\mu=\left[.,.\right])$ is a linear map $N:E\rightarrow E$ such that its  Nijenhuis torsion with respect to $\mu$, defined by
\begin{equation}\label{Torsion}
T_{\mu}N(X,Y):=\mu(NX,NY)-N(\mu(NX,Y)+\mu(X,NY)-N(\mu(X,Y))),
\end{equation}
for all $X$, $Y\in E$, is identically zero.
For a binary bracket $\mu=\left[.,.\right]$, the deformed bracket by $N$ is denoted by $\left[.,.\right]_N$ and is given by $[X,Y]_N=[NX,Y]+[X,NY]-N[X,Y]$. It has been shown in \cite{YKS} that if $N$ is Nijenhuis on a Lie algebra $(E,\left[.,.\right])$, then $(E,\left[.,.\right]_N)$ is also a Lie algebra and $N$ is a morphism of Lie algebras.
Also, it has been shown that $N$ is Nijenhuis if and only if deforming the original bracket of the Lie algebra twice by $N$ is equivalent to deform it once by $N^2$, that is $([X,Y]_N)_N=[X,Y]_{N^2}$. This can be stated using the notion of Richardson-Nijenhuis bracket on the space of vector valued forms on a graded vector space $E$, as follows:
\begin{equation*}
[N,[N,\mu]_{_{RN}}]_{_{RN}}=[N^2,\mu]_{_{RN}}.
\end{equation*}
So, we conclude that \emph{Nijenhuis operators in the usual and traditional sense are, of course, Nijenhuis in our sense also.}

Next, we present the first examples of Nijenhuis vector valued forms on $L_\infty$-algebras. We start by introducing the
Euler map $S$, the map that simply counts the degree of homogeneous elements in a graded vector space. More precisely, given a graded vector space
$E=\oplus_{i\in \mathbb{Z}}E_i$,  $S:E\rightarrow E$ is defined by $S(X)=-|X|X$, for all homogeneous elements $X\in E$ of degree $|X|$.

Notice that $S$, as a graded map, has degree zero, $\bar{S}=0$. By a simple computation, using the definition of $S$, we get the following result.

\begin{lem}\label{Euler}
Let $E=\oplus_{i\in \mathbb{Z}}E_i$ be a graded vector space. Then,
\begin{equation*}
[S,\alpha]_{_{RN}}=\bar{\alpha}\,\alpha,
\end{equation*}
for every symmetric vector valued form $\alpha$ on $E$ of degree $\bar{\alpha}$.
\end{lem}

\begin{prop}
 Let $\mu$ be a vector valued form of degree $1$ on a graded vector space $E=\oplus_{i\in \mathbb{Z}}E_i$. The Euler map $S$ is a Nijenhuis vector valued form with respect to $\mu$ with square $S$.
\end{prop}
\begin{proof}
Let $\mu=\sum_{i=1}^{\infty}l_i$. Applying Lemma \ref{Euler} to each $l_i, 1\leq i\leq \infty$, and taking the sum we get:
\begin{equation*}
[S,\mu]_{_{RN}}=\sum_{i=1}^{\infty}[S,l_i]_{_{RN}}=\sum_{i=1}^{\infty}l_i=\mu.
\end{equation*}
Therefore
\begin{equation*}
[S,[S,\mu]_{_{RN}}]_{_{RN}}=[S,\mu]_{_{RN}}.
\end{equation*}
Since $\bar{S}=0$, Lemma \ref{Euler} implies that $[S,S]_{_{RN}}=0$ and this completes the proof.
\end{proof}

Of course, the result can be enlarged for every $\mu$-cocycle, that is, a vector valued form $\alpha$ such that $[\mu,\alpha]_{_{RN}}=0$.
\begin{prop}
 Let $\mu$ be a vector valued form of degree $1$ on a graded vector space $E$. Then, for every element $\alpha$ of degree $0$ in $ {\tilde S}(E^*) \otimes E $ with $[\mu,\alpha]_{_{RN}} =0$, $S+\alpha $ is Nijenhuis with respect to $\mu$, with square $S$.
 \end{prop}

Next, we give some examples of Nijenhuis forms on symmetric graded and symmetric differential graded Lie algebras. For that, we need to introduce the notions of Maurer-Cartan and Poisson elements.

A \emph{Maurer-Cartan element} in a symmetric differential graded Lie algebra $(E,\diff,[.,.])$ is an element $e \in E_0$ such that
\begin{equation*}
\diff(e)-\frac{1}{2}[e,e]=0.
\end{equation*}
A \emph{Maurer-Cartan element} in a symmetric curved differential graded Lie algebra $(E,\, {\mathfrak C},\diff, [.,.])$ is an element $e \in E_0$ such that
\begin{equation*}
(\diff(e)-{\mathfrak C})-\frac{1}{2}[e,e]=0.
\end{equation*}

A \emph{Poisson element} in a curved $L_{\infty}$-algebra $(E,\mu=\sum_{i\geq 0}l_i)$ is an element $\pi \in E_0$, such that $l_2(\pi,\pi)=0$.

The next proposition provides an example of a Nijenhuis vector valued form on a symmetric differential graded Lie algebra.

\begin{prop}\label{deformationofDGLA}
Let $\mu={\mathfrak C}+l_1+l_2$ be a curved symmetric differential graded Lie algebra  structure on a graded vector space $E=\oplus_{i\in \mathbb{Z}}E_i$ and $\pi\in E_0$. Then, ${\mathcal N}=\pi+S$ is a Nijenhuis vector valued form  (with curvature $\pi$) with respect to $\mu$ and with square $2\pi+S$ if, and only if, $\pi$ is a Poisson element.

In this case, the deformed structure is the curved symmetric differential graded Lie algebra $\left(E,({\mathfrak C}+l_1(\pi))+(l_1+l_2(\pi,.))+l_2\right)$.
\end{prop}
\begin{proof}
The proof of the equivalence follows from:
\begin{equation}  \label{deformed DGLA}
[\pi+S,{\mathfrak C}+l_1+l_2]_{_{RN}}={\mathfrak C}+l_1(\pi)+(l_2(\pi,.)+l_1)+l_2,
\end{equation}

\begin{equation*}
\begin{array}{rcl}
[\pi+S,[\pi+S,{\mathfrak C}+l_1+l_2]_{_{RN}}]_{_{RN}}&=&[\pi+S,{\mathfrak C}+l_1+l_2+l_1(\pi)+l_2(\pi,.)]_{_{RN}}\\
                                 &=&{\mathfrak C}+l_1+l_2+2l_1(\pi)+2l_2(\pi,.)+l_2(\pi,\pi)\\
                                 &=&[2\pi+S,{\mathfrak C}+l_1+l_2]_{_{RN}}+l_2(\pi,\pi)
\end{array}
\end{equation*}
and
\begin{equation*}
[\pi+S,2\pi+S]_{_{RN}}=2[\pi,\pi]_{_{RN}}+[\pi,S]_{_{RN}}+2[S,\pi]_{_{RN}}+[S,S]_{_{RN}}=0.
\end{equation*}
The last statement follows directly from  (\ref{deformed DGLA}) and Theorem \ref{theo:Hierarchy}.
\end{proof}

Notice that, in Proposition \ref{deformationofDGLA}, if we start with a symmetric differential graded Lie algebra without curvature, that is, if ${\mathfrak C}=0$, then, the deformed structure is a curved symmetric differential graded Lie algebra with curvature $l_1(\pi)$. Moreover, if $l_1=0$, Proposition~\ref{deformationofDGLA} provides an example of Nijenhuis vector valued form on a symmetric graded Lie algebra.

\begin{prop}\label{prop:MaurerCartanDGLA}
Let $\mu={\mathfrak C}+l_1+l_2$ be a curved symmetric differential graded Lie algebra structure on a graded vector space $E=\oplus_{i\in \mathbb{Z}}E_i$ and $\pi\in E_0$. Then, ${\mathcal N}=Id_E+\pi$ is a Nijenhuis vector valued form (with curvature $\pi$) with respect to $\mu$ and with square $Id_E+\pi$ if, and only if, $\pi$ is a Maurer-Cartan element.

In this case, the deformed structure is the curved symmetric differential graded Lie algebra $(E,(l_1(\pi)-{\mathfrak C})+l_2(\pi,.)+l_2)$.
\end{prop}
\begin{proof} First notice that
\begin{equation} \label{deformedMC}
[\pi+Id_E,{\mathfrak C}+l_1+l_2]_{_{RN}}=(l_1(\pi)-{\mathfrak C})+l_2(\pi,.)+l_2
\end{equation}
and
\begin{eqnarray*}
& &[\pi+Id_{E},[\pi+Id_E,{\mathfrak C}+l_1+l_2]_{_{RN}}]_{_{RN}}=
                             l_2(\pi,\pi)+l_2(\pi,.)-l_1(\pi)+{\mathfrak C}+l_2\\
                                   &&\hspace{1cm} =-{\mathfrak C}-2((l_1(\pi)-{\mathfrak C})-\frac{1}{2}l_2(\pi,\pi))+l_1(\pi)+l_2(\pi,.)+l_2\\
                                   &&\hspace{1cm} =-2((l_1(\pi)-{\mathfrak C})-\frac{1}{2}l_2(\pi,\pi))+[\pi+Id_E,{\mathfrak C}+l_1+l_2]_{_{RN}}.
                                   \end{eqnarray*}
This, together with the fact that $[\pi+Id_E,\pi+Id_E]_{_{RN}}=0$, imply that $Id_E+\pi$ is a Nijenhuis vector valued form with respect to $\mu$ if, and only if,  $\pi$ is a Maurer-Cartan element of the curved symmetric differential graded Lie algebra $(E,\mu)$. The last statement follows from (\ref{deformedMC}) and Theorem~\ref{theo:Hierarchy}.
\end{proof}

%%%%%%%%%%%%%%%%%%%%%%%%%%%%%%%%%%%%%%%%%%%%%%%%%%%%%%%%%%%%%%%%%%%%%%%%%%%%%%%%%%%%%%%%%%%%%%%%%%%%%%%%%%%%%%%%%%%%%%%%%%%%%%%%%%%%%%%%%%%%%%%%%%

\section{Nijenhuis forms on Lie $n$-algebras}

Lie $n$-algebras are particular cases of $L_{\infty}$-algebras for which only $n+1$ brackets may be non-zero. We define Nijenhuis forms for this special case and we analyze, in particular, the Lie $n$-algebra defined by an $n$-plectic manifold.

A graded vector space $E=\oplus_{i\in \mathbb{Z}}E_i$ is said to be \emph{concentrated in degrees $p_1,\cdots p_k$}, with $p_1,\cdots, p_k \in \mathbb{Z}$, if $E_{p_1},\cdots,E_ {p_k}$ are the only non-zero components of $E$.
\begin{defn}
A symmetric Lie $n$-algebra is a symmetric $L_{\infty}$-algebra whose underlying graded vector space is concentrated on degrees $-n,\cdots,-1$.
\end{defn}

\begin{rem}\label{degreereason}
Note that, by degree reasons, any vector valued $k$-form of degree $1$ has to vanish for $k \geq n+2$. So,
the only non-zero symmetric vector valued forms (multi-brackets) on a symmetric Lie $n$-algebra are $l_1,\cdots,l_{n+1}$.
\end{rem}

\begin{prop}\label{corlast}
Let $(E=E_{-n}\oplus \cdots\oplus E_{-1}, \mu=l_1+\cdots +l_{n+1})$ be a Lie $n$-algebra. Let $N_1,\cdots,N_l$ be a family of symmetric vector valued $k_1,\cdots,k_l$-forms, respectively, of degree zero on $E$, with $\frac{n+3}{2}\leq k_1\leq \cdots\leq k_l \leq n+1$. Then, ${\mathcal N}=S+\sum_{i=1}^l N_i$ is a Nijenhuis vector valued form with respect to $\mu$, with square $S+2\sum_{i=1}^l N_i$. The deformed Lie $n$-algebra structure is
\begin{equation*}
\begin{array}{rcl}
\left[S+\sum_{i=1}^l N_i,\mu\right]_{_{RN}}&=&\mu+\left[\sum_{i=1}^l N_i,l_1\right]_{_{RN}}+\cdots+\left[\sum_{i=1}^l N_i,l_{n-k_l+2}\right]_{_{RN}}\\
       && + \left[\sum_{i\not=l} N_i,l_{n-k_l+3}\right]_{_{RN}}+\cdots+\left[\sum_{i\not=l} N_i,l_{n-k_{l-1}+2}\right]_{_{RN}}\\
       && + \left[\sum_{i\not=l,l-1} N_i,l_{n-k_l+3}\right]_{_{RN}}+\cdots+\left[\sum_{i\not=l,l-1} N_i,l_{n-k_{l-1}+2}\right]_{_{RN}}\\
       && + \cdots+\\
       && +\left[ N_1,l_{n-k_2+3}\right]_{_{RN}}+\cdots+\left[ N_1,l_{n-k_1+2}\right]_{_{RN}}.
\end{array}
\end{equation*}
\end{prop}
\begin{proof}
Let $1\leq i,j\leq l$. By Remark \ref{degreereason}, any vector valued $(m+k_i-1)$-form of degree $1$, with  $m\geq n-k_i+3$,  is identically zero; hence,
 \begin{equation}  \label{Nilm}
 \left[N_i,l_m\right]_{_{RN}}=0,
 \end{equation}
  for all $m\geq n-k_i+3$. Also,
 any vector valued $(k_i+k_j+m-2)$-form, with $m\geq 1$ is identically zero, because out of the conditions $\frac{n+3}{2}\leq k_1\leq \cdots\leq k_l \leq n+1$ and $m\geq 1$ we get $k_i+k_j+m-2\geq n+2.$ Thus,
\begin{equation}  \label{Nilm1}
\left[N_i,\left[N_j,l_m\right]_{_{RN}}\right]_{_{RN}}=0,
\end{equation} for all  $m\geq 1$.
From Equations (\ref{Nilm}) and (\ref{Nilm1}), we get
\begin{equation*}
\begin{array}{rcl}
\left[S+\sum_{i=1}^l N_i,\mu\right]_{_{RN}}&=&\mu+\left[\sum_{i=1}^l N_i,l_1\right]_{_{RN}}+\cdots+\left[\sum_{i=1}^l N_i,l_{n-k_l+2}\right]_{_{RN}}\\
       && +\left[\sum_{i\not=l} N_i,l_{n-k_l+3}\right]_{_{RN}}+\cdots+\left[\sum_{i\not=l} N_i,l_{n-k_{l-1}+2}\right]_{_{RN}}\\
       && + \left[\sum_{i\not=l,l-1} N_i,l_{n-k_l+3}\right]_{_{RN}}+\cdots+\left[\sum_{i\not=l,l-1} N_i,l_{n-k_{l-1}+2}\right]_{_{RN}}\\
       & & +\cdots +\\
       &&+ \left[ N_1,l_{n-k_2+3}\right]_{_{RN}}+\cdots+\left[ N_1,l_{n-k_1+2}\right]_{_{RN}}
\end{array}
\end{equation*}
 and
 \begin{equation*}
 \left[S+\sum_{i=1}^l N_i,\left[S+ \sum_{i=1}^l N_i,\mu\right]_{_{RN}}\right]_{_{RN}}=\mu+2\left[\sum_{i=1}^l N_i,\mu\right]_{_{RN}}=\left[S+2\sum_{i=1}^l N_i,\mu\right]_{_{RN}}.
 \end{equation*}
 It follows from the conditions $\frac{n+3}{2}\leq k_1\leq \cdots\leq k_l \leq n+1$
that, for $1\leq i,j\leq l$, we have $k_i+k_j-1\geq n+2$. By degree reasons, any vector valued $k$-form of degree zero has to vanish for $k \geq n+1$. Hence, $\left[N_i,N_j\right]_{_{RN}}=0$ for all $1\leq i,j\leq l$, which implies that
 \begin{equation*}
 \left[S+\sum_{i=1}^l N_i,S+2\sum_{i=1}^l N_i\right]_{_{RN}}=0.
 \end{equation*}
\end{proof}
\begin{rem}
In Proposition \ref{corlast} one may replace each vector valued $k_i$-form $N_i$ by a family of symmetric vector valued $k_i$-forms. Also, we should stress that Proposition \ref{corlast} is \emph{not} a generalization of Proposition~\ref{deformationofDGLA}. Looking at $\pi \in E_0$ in Proposition \ref{deformationofDGLA} as a vector valued $0$-form of degree zero, the assumptions of Proposition \ref{corlast} do not apply to $\pi$ because they are not satisfied for $n=1$, $l=1$ and $k_1=0$. Besides, notice that in Proposition~\ref{deformationofDGLA} $E=\oplus_{i\in \mathbb{Z}}E_i$, while in Proposition \ref{corlast} $E=E_{-n}\oplus \cdots\oplus E_{-1}$, i.e., the degrees are concentrated in $-n, \cdots, -1$.
\end{rem}

Next, we consider a particular class of Lie $n$-algebras, those associated to $n$-plectic manifolds. Let us recall some definitions from \cite{CRoger}.
\begin{defn}\label{def:nplecticmanifold}
An {\em $n$-plectic manifold} is a manifold $M$ equipped with a non-degenerate and closed $(n+1)$-form $\omega$. It is denoted by $(M,\omega)$.
\end{defn}
An $(n-1)$-form $\alpha$ on an $n$-plectic manifold $(M,\omega)$ is said to be  a \emph{Hamiltonian form} if there exists a smooth vector field $\chi_{\alpha}$ on $M$ such that $\diff \alpha=-\iota_{\chi_{\alpha}}\omega$. The vector field $\chi_{\alpha}$ is called the \emph{Hamiltonian vector field} associated to $\alpha$. The space of all Hamiltonian forms on an $n$-plectic manifold $(M,\omega)$ is denoted by $\Omega_{Ham}^{n-1}(M)$.

Given two Hamiltonian forms $\alpha, \beta$ on an $n$-plectic manifold $(M,\omega)$, with Hamiltonian vector fields $\chi_{\alpha}$ and $\chi_{\beta}$, respectively, one may define a bracket $\{.,.\}$ by setting
\begin{equation*}\label{bracketonhamiltonians}
\{\alpha,\beta\}:=\iota_{\chi_{\alpha}}\iota_{\chi_{\beta}}\omega.
\end{equation*}
It turns out that $\{\alpha,\beta\}$ is a Hamiltonian form with associated Hamiltonian vector field $[\chi_{\alpha},\chi_{\beta}]$, see \cite{CRoger}.

Following \cite{CRoger}, we may associate to an $n$-plectic manifold $(M,\omega)$ a symmetric Lie $n$-algebra.
\begin{thm}\label{Crogernalgebrafromnplectic}
Let $(M,\omega)$ be an $n$-plectic manifold. Set
$$E_i = \begin{cases}
                      \Omega^{n-1}_{Ham}(M), & \mbox{if } \,\,\,\,\,i=-1,\\
                       \Omega^{n+i}(M), & \mbox{if }  -n\leq i\leq -2
        \end{cases}$$
and $E=\oplus_{i=-n}^{-1}E_i$. Let the collection $l_k :E\times {\stackrel{k}{\dots}} \times E\to E$, $k \geq 1$, of symmetric multi-linear maps be defined as
$$l_1(\alpha) = \begin{cases}
                      (-1)^{|\alpha|}\diff \alpha, & \mbox{if } \,\,\,\,\,\alpha \not\in E_{-1},\\
                       0, & \mbox{if } \,\,\,\,\, \alpha \in E_{-1},
        \end{cases}$$
$$l_{k}(\alpha_1,\cdots,\alpha_k) = \begin{cases}
                      0, & \mbox{if } \alpha_i \not\in E_{-1} \,\,\,\mbox{for some}\,\,\,\, 1\leq i \leq k, \\
                       (-1)^{\frac{k}{2}+1}\iota_{\chi_{\alpha_{1}}}\cdots \iota_{\chi_{\alpha_{k}}}\omega, & \mbox{if }  \alpha_i \in E_{-1} \,\,\,\mbox{for all}\,\,\,\,\,\ \,\,1\leq i \leq k \,\,\,\mbox{and}\,\,\,\, k \,\,\mbox{is even},\\
                       (-1)^{\frac{k-1}{2}}\iota_{\chi_{\alpha_{1}}}\cdots \iota_{\chi_{\alpha_{k}}}\omega, & \mbox{if }  \alpha_i \in E_{-1} \,\,\,\mbox{for all}\,\,\,\,\,\ \,\,1\leq i \leq k \,\,\,\mbox{and}\,\,\,\, k \,\,\mbox{is odd},\\
        \end{cases}$$
        for $k \geq 2$, where $\chi_{\alpha_{i}}$ is the Hamiltonian vector field associated to $\alpha_i$.
        Then, $(E,(l_k)_{k \geq 1})$ is a symmetric Lie $n$-algebra.
\end{thm}

In the next proposition we give an example of a Nijenhuis vector valued form, with respect to the $L_\infty$-algebra (Lie $n$-algebra) structure associated to a given $n$-plectic manifold, which is the sum of a symmetric vector valued $1$-form with a symmetric vector valued $i$-form, with $i=2,\cdots,n$.

\begin{prop}\label{thm:NijenN}
Let $(M,\omega)$ be an $n$-plectic manifold with the associated symmetric Lie $n$-algebra structure $\mu=l_1+\cdots+l_{n+1}$. For any $n$-form $\eta$ on the manifold $M$,  and any $i=2,\dots,n$, define $\widetilde{\eta}_i$ to be the symmetric vector valued $i$-form of degree zero given by
\begin{equation}\label{tilde}
\widetilde{\eta_i}(\beta_1,\cdots,\beta_i)=\begin{cases}
                                            \iota_{\chi_{\beta_1}}\cdots\iota_{\chi_{\beta_i}} \eta, & \mbox{if}\,\,\,\,\beta_i \in E_{-1},\\
                                           0, & \mbox{otherwise},
                                           \end{cases}
\end{equation}
where $\chi_{\beta_1},\cdots,\chi_{\beta_n}$ are the Hamiltonian vector fields of $\beta_1,\cdots,\beta_n,$ respectively. Then,
 $S+\widetilde{\eta}_i$ is a Nijenhuis vector valued form with respect to $\mu$, with square $S+2\widetilde{\eta}_i$. The deformed structure is
     \begin{equation*}
     [S+\widetilde{\eta_i},\mu]_{_{RN}}=\mu+[\widetilde{\eta_i},l_1]_{_{RN}}+[\widetilde{\eta_i},l_2]_{_{RN}}.
     \end{equation*}
\end{prop}

The proof of Proposition \ref{thm:NijenN} is based on the following lemma.
\begin{lem}\label{lem5ghesmati} For all $2\leq i\leq n$, and all homogeneous elements $\alpha_1,\cdots,\alpha_{i}\in E,$ we have:\\
\begin{enumerate}
\item[(1)] $\widetilde{\eta_i}(l_1(\alpha_1),\alpha_2,\cdots,\alpha_{i})=0,$
\\
\item[(2)] $[\widetilde{\eta_i},l_m]_{_{RN}}=\begin{cases} 0, & m \geq 3 \\ -\iota_{l_2}\widetilde{\eta_i}, & m=2 \\
\diff \circ \widetilde{\eta_i}, & m=1 \end{cases}$\\
\item[(3)] $[\widetilde{\eta_i}, [\widetilde{\eta_i},l_m]_{_{RN}}]_{_{RN}}=0, \,\, m \geq 1$.
\end{enumerate}
\end{lem}

\begin{proof}
We start by noticing that from its definition,  $\widetilde{\eta_i}$ vanishes on $\oplus_{i=-n}^{-2} E_i$ and ${\rm Im} \, \widetilde{\eta_i} \subset E_{-i}$, $i \geq 2$.
So, to prove item (1), the only case we have to investigate is when $\alpha_1\in E_{-2}$ and $l_1(\alpha_1), \alpha_2,\cdots,\alpha_{i}$ are all Hamiltonian forms. Let $\chi_{l_1(\alpha_1)}$ be the  Hamiltonian vector field associated to $l_1(\alpha_1)$. Then, we have
\begin{equation*}
\iota_{\chi_{l_1(\alpha_1)}}\omega=-\diff (l_1(\alpha_1))=-\diff^2 \alpha_1=0,
 \end{equation*}
thus $\chi_{l_1(\alpha_1)}=0$, by the non-degeneracy of $\omega$. This proves item (1).

Let us now compute $[\widetilde{\eta_i},l_m]_{_{RN}}$. When $m \geq 3$, from the definitions of $l_m$ and $\widetilde{\eta_i}$, we get
\begin{equation*}\label{item3-1}
l_m(\widetilde{\eta_i}(\alpha_1,\cdots,\alpha_{i}),\cdots,\alpha_{m+i-1})=0
 \end{equation*}
 and
 \begin{equation*}\label{item3-2}
\widetilde{\eta_i}( l_m(\alpha_1,\cdots,\alpha_{m}),\cdots,\alpha_{m+i-1})=0,
 \end{equation*}
 for all  $\alpha_1,\cdots,\alpha_{i+m-1}\in E$, $i \geq 2$, so that $[\widetilde{\eta_i},l_m]_{_{RN}}=0$. Since $\widetilde{\eta_i}$ takes value in $E_{-i}$, we have $\iota_{\widetilde{\eta_i}}l_2=0$, hence $[\widetilde{\eta_i},l_2]_{_{RN}}=-\iota_{l_2}\widetilde{\eta_i}$.
 From item (1) and definition of $\widetilde{\eta_i}$ we get $[\widetilde{\eta_i},l_1]_{_{RN}}=\diff \circ \widetilde{\eta_i}$.

 Last, we prove item (3). For $m \geq 3$, $[\widetilde{\eta_i}, [\widetilde{\eta_i},l_m]_{_{RN}}]_{_{RN}}=0$ is a direct consequence of item (2).
The case $m=2$ follows from the fact that
 $\widetilde{\eta_i}$ does not take value in $E_{-1}$, so $l_2(\widetilde{\eta_i}(\alpha_1,\cdots,\alpha_{i}),\alpha_{i+1})=0$, for all $\alpha_1,\cdots,\alpha_{i+1}\in E$. Hence, using item (2) we get
 \begin{equation*}\label{item7-1}
 \iota_{\widetilde{\eta_i}}[\widetilde{\eta_i},l_2]_{_{RN}}=0 \,\,\,\,\, {\mbox{\rm and}}\,\,\,\,\, \iota_{[\widetilde{\eta_i},l_2]_{_{RN}}}\widetilde{\eta_i}=0,
 \end{equation*}
 which gives $[\widetilde{\eta_i}, [\widetilde{\eta_i},l_2]_{_{RN}}]_{_{RN}}=0$.
Similar arguments as those used above prove that $[\widetilde{\eta_i}, [\widetilde{\eta_i},l_1]_{_{RN}}]_{_{RN}}=0$.
\end{proof}

\begin{proof}(of Proposition \ref{thm:NijenN})
From Lemma~\ref{lem5ghesmati} we have
\begin{equation}\label{themnplectic}
[S+\widetilde{\eta_i},\mu]_{_{RN}}=\mu+[\widetilde{\eta_i},l_1]_{_{RN}}+[\widetilde{\eta_i},l_2]_{_{RN}}
\end{equation}
and applying $[S+\widetilde{\eta_i},.]_{_{RN}}$ to both sides of Equation (\ref{themnplectic}), we get
\begin{eqnarray*}
[S+\widetilde{\eta_i},[S+\widetilde{\eta_i},\mu]_{_{RN}}]_{_{RN}}&=&\mu+2[\widetilde{\eta_i},l_1]_{_{RN}}+2[\widetilde{\eta_i},l_2]_{_{RN}}\\
&=&[S+2\widetilde{\eta_i},\mu]_{_{RN}}.
\end{eqnarray*}
Now, the equation
\begin{equation*}\label{haminzeperty}
[S+\widetilde{\eta_i},S+\widetilde{\eta_i}]_{_{RN}}=0,
\end{equation*}
holds, for all $i \geq 2$, as a consequence of $\iota_{\widetilde{\eta_i}}\widetilde{\eta_i}=0$.
\end{proof}

From Proposition \ref{thm:NijenN} we get the following result.
\begin{thm}\label{nplecticlasttheorem}
Let $(\eta^j)_{j\geq 1}$ be a family of $n$-forms on an $n$-plectic manifold $(M,\omega)$. Let $(E=E_{-n}\oplus \cdots\oplus E_{-1}, \mu=l_1+\cdots+l_{n+1})$ be the Lie $n$-algebra associated to $(M,\omega)$. For each $2\leq i\leq n$, define the vector valued $i$-forms $\widetilde{(\eta^{j})_i}$
 as
\begin{equation*}
\widetilde{(\eta^{j})_i}(\beta_1,\cdots,\beta_i)=\begin{cases}
                                             \iota_{\chi_{\beta_1}}\cdots\iota_{\chi_{\beta_i}}\eta^j, & \mbox{if} \,\,\,\ \beta_k \in E_{-1}\,\,\, \mbox{for all} \,\,\,1\leq k\leq i,\\
                                             0, & \mbox{otherwise}
                                             \end{cases}
\end{equation*}where $ \chi_{\beta_1},\cdots,\chi_{\beta_i}$ are the Hamiltonian vector fields associated to the Hamiltonian forms $\beta_1,\cdots,\beta_i$, respectively. Then, $\mathcal{N}:=S+\sum_{j\geq 1}\sum_{i=2}^{n}\widetilde{(\eta^{j})_i}$ is a Nijenhuis vector valued form with respect to the Lie $n$-algebra structure $\mu$.
\end{thm}

%%%%%%%%%%%%%%%%%%%%%%%%%%%%%%%%%%%%%%%%%%%%%%%%%%%%%%%%%%%%%%%%%%%%%%%%%%%%%%%%%%%%%%%%%%%%%%%%%%%%%%%%%%%%%%%%%%%%%%%%%%%%%%%%%%%%%%%
%%%%%%%%%%%%%%%%%%%%%%%%%%%%%%%%%%%%%%%%%%%%%%%%%%%%%%%%%%%%%%%%%%%%%%%%%%%%%%%%%%%%%%%%%%%%%%%%%%%%%%%%%%%%%%%%%%%%%%%%%%%%%%%%%%%%%%%%%%%%%%%%%%
\section{The case of Lie $2$-algebras}

In this section we treat the case of Lie $2$-algebras. We show how to construct Nijenhuis forms with respect to Lie $2$-algebras, which are the sum of a vector valued $1$-form with a vector valued $2$-form.

 We start by recalling that a Lie $2$-algebra is a pair $(E, \mu)$, where $E$ is a graded vector space with degrees concentrated in $-2$ and $-1$, that is $E=E_{-2} \oplus E_{-1}$, and $\mu=l_1 +l_2 +l_3$ with $l_1$, $l_2$ and $l_3$ being symmetric vector valued $1$-form, $2$-form and $3$-form, respectively, all of them of degree $1$. For degree reasons, the brackets $l_1$ and $l_3$ are not identically zero in the following cases:
$$l_1: E_{-2} \to E_{-1}, \,\,\,\,\,l_3: E_{-1} \times  E_{-1} \times  E_{-1} \to E_{-2},$$
while the binary bracket $l_2$ has two parts
 $$l_2|_{E_{-1} \times  E_{-2}}: E_{-1} \times  E_{-2} \to E_{-2}, \,\,\,\,\, l_2|_{E_{-1} \times  E_{-1}}: E_{-1} \times  E_{-1} \to E_{-1}.$$ The equation $[\mu, \mu]_{_{RN}}=0$ gives the following relations (by degree reasons, all the missing cases are identically zero):
 \begin{equation} \label{Lie2relations1A}
[l_1,l_2]_{_{RN}}(f,g)=0,
\end{equation}
\begin{equation} \label{Lie2relations1}
[l_1,l_2]_{_{RN}}(X,f)=0,
\end{equation}
\begin{equation} \label{Lie2relations2}
\left(2[l_1,l_3]_{_{RN}}+[l_2,l_2]_{_{RN}}\right)(X,Y,f)=0,
\end{equation}
\begin{equation} \label{Lie2relations3}
(2[l_1,l_3]_{_{RN}}+[l_2,l_2]_{_{RN}})(X,Y,Z)=0,
\end{equation}
\begin{equation} \label{Lie2relations4}
[l_2,l_3]_{_{RN}}(X,Y,Z,W)=0,
\end{equation}
with $X,Y,Z,W\in E_{-1}$ and $f, g \in E_{-2}$.

Let us set
\begin{equation} \label{omega}
l_1= \partial, \,\,\,\,\,\,\, l_3=\omega
\end{equation}
and,
 for all $X,Y\in E_{-1}$ and $f\in E_{-2}$,
\begin{equation} \label{chi}
 l_2|_{E_{-1} \times  E_{-1}}(X,Y)=[X,Y]_2 \,\,\,\,\,\,\, \mbox{\rm and}\,\,\,\,\,\,\, l_2|_{E_{-1} \times  E_{-2}}(X,f)= \chi(X)f,
\end{equation}
with $\chi:E_{-1} \to End(E_{-2})$. Then, we have:

\begin{lem} \label{prop:quadrubleLie2algebra}
A vector valued form $\mu=l_1+l_2+l_3$, with associated quadruple $(\partial, \chi, [.,.]_2, \omega)$ given by (\ref{omega}) and (\ref{chi}), is a Lie $2$-algebra structure on $E=E_{-2} \oplus E_{-1}$ if and only if
 \begin{equation} \label{Lie2relations5A}
 \chi(\partial f)g=- \chi(\partial g)f,
 \end{equation}
 \begin{equation} \label{Lie2relations5}
 [X,\partial f]_2=\partial(\chi(X)f),
 \end{equation}
\begin{equation} \label{Lie2relations6}
\chi([X,Y]_2)f+\chi(Y)\chi(X)f-\chi(X)\chi(Y)f+\omega(X,Y,\partial f)=0,
\end{equation}
\begin{equation} \label{Lie2relations7}
[[X,Y]_2,Z]_2+c.p.= \partial(\omega (X,Y,Z)),
\end{equation}
\begin{eqnarray} \label{Lie2relations8}
\lefteqn{ \chi(W)\omega(X,Y,Z)-\chi(Z)\omega(X,Y,W)+\chi(Y)\omega(X,Z,W)
 -\chi(X)\omega(Y,Z,W)= }\nonumber \\
 &-\omega([X,Y]_2,Z,W)+\omega([X,Z]_2,Y,W)
 -\omega([X,W]_2,Y,Z)\nonumber  \\
 &-\omega([Y,Z]_2,X,W)+\omega([Y,W]_2,X,Z)
 -\omega([Z,W]_2,X,Y),
\end{eqnarray}
for all $X,Y,Z,W \in E_{-1}$ and $f \in E_{-2}$.
\end{lem}
\begin{proof}
We have the following equivalences, by applying the definition of Richardson-Nijenhuis bracket:
$(\ref{Lie2relations1A})\Leftrightarrow (\ref{Lie2relations5A})$,
$(\ref{Lie2relations1})\Leftrightarrow (\ref{Lie2relations5})$, $(\ref{Lie2relations2})\Leftrightarrow (\ref{Lie2relations6})$, $(\ref{Lie2relations3})\Leftrightarrow (\ref{Lie2relations7})$ and $(\ref{Lie2relations4})\Leftrightarrow (\ref{Lie2relations8})$.
\end{proof}
The quadruple $(\partial, \chi, [.,.]_2, \omega)$ of Lemma \ref{prop:quadrubleLie2algebra} is the {\em quadruple associated to the Lie $2$-algebra structure} $\mu=l_1+l_2+l_3$.

\

There is an associated Chevalley-Eilenberg differential to each Lie $2$-algebra. Before giving its definition, we need the next lemma.
\begin{lem}\label{Koskhlane}
  Let $(E=E_{-2}\oplus E_{-1},\,\mu=l_1+l_2+l_3)$ be a Lie $2$-algebra
  with corresponding quadruple $(\partial,\chi, [.,.]_2,\omega)$
  and $\eta \in S^k(E^*)\otimes E$ be a vector valued $k$-form of degree $k-2$. Then,
\begin{eqnarray}\label{eq:CEilengerg}
[\eta,l_2]_{_{RN}} (X_0 , \dots, X_k) &=& \sum_{i=0}^k (-1)^i \chi(X_i) \eta(X_0, \cdots, \widehat{X_i}, \cdots, X_{k}) \nonumber\\
& + & \sum_{0\leq i<j\leq k}(-1)^{i+j}
 \eta([X_i,X_j]_2, X_0, \cdots, \widehat{X_i}, \cdots, \widehat{X_{j}}, \cdots, X_{k}),
 \end{eqnarray}
   for all $X_0,\dots,X_k \in E_{-1}$ , where $\widehat{X_{i}}$ means the absence of $X_i$.
    \end{lem}
    \begin{proof}
   By degree reasons, $\eta$ has to be of the form
$\eta:E_{-1}\times {\stackrel{k}{\dots}} \times E_{-1}\to E_{-2}$. Using the Richardson-Nijenhuis bracket definition one gets Equation (\ref{eq:CEilengerg}).
    \end{proof}

    \begin{defn}
    Let $E=E_{-2}\oplus E_{-1}$ be a graded vector space concentrated on degrees $-2$ and $-1$, $S_{k}(E)\subset S^{k}(E^*)\otimes E$ be the subspace of all symmetric vector valued $k$-forms of degree $k-2$ and $S^{\bullet}(E):=\oplus_{k\geq 1}S_{k}(E)$. Let $\chi:E_{-1} \to End(E_{-2})$ be a representation of vector spaces and $[.,.]:E_{-1} \times E_{-1}\to E_{-1}$ a graded symmetric bilinear map. Then, the {\em Chevalley-Eilenberg differential} $\diff^{CE}$ is the map $$\diff^{CE}:S^{\bullet}(E) \to S^{\bullet}(E)$$ such that, if $\eta \in S_{k}(E)$, then   $\diff^{CE}\eta \in S_{k+1}(E)$ is defined by
\begin{eqnarray*}\label{eq:CEilenberggeneral}
 \diff^{CE}\eta(X_0 , \dots, X_k) &=&  \sum_{i=0}^k (-1)^i \chi(X_i) \eta( X_0,\cdots, \widehat{X_i}, \cdots, X_{k})  \\
 & & +\sum_{0\leq i<j\leq k}(-1)^{i+j}
 \eta([X_i,X_j], X_0,\cdots,\widehat{X_i}, \cdots, \widehat{X_j}, \cdots, X_{k}),
 \end{eqnarray*}
   for all $X_0,\dots,X_k \in E_{-1}$, where $\widehat{X_i}$ means for the absence of $X_i$.
    \end{defn}
    In general, the operator $\diff^{CE}$ does not square to zero. However, according to Lemma \ref{Koskhlane} it can be written as
  \begin{equation} \label{diff_RN}
    {\diff}^{CE}  = [., l_2]_{_{RN}},
    \end{equation}
  and we get, from the graded Jacobi identity of the Richardson-Nijenhuis bracket,  that $\diff^{CE}$ squares to zero if and only if $[l_2,l_2]_{_{RN}}=0.$

   Next, we explain how  a crossed module of Lie algebras can be seen as a Lie $2$-algebra. Let us first recall the definition of a crossed module of Lie algebras \cite{FWagemann}:
\begin{defn} \label{crossedmodule_def}
A {\em crossed module} of Lie algebras $(\mathfrak g,\, [.,.]^{\mathfrak g})$ and $(\mathfrak h,\, [.,.]^{\mathfrak h})$ is a homomorphism  $\partial:\mathfrak g \to \mathfrak h$ together with an action by derivation of $\mathfrak h$ on $\mathfrak g$, that is, a linear map $\chi: \mathfrak h \to Hom(\mathfrak g,\mathfrak g)$ such that
\begin{equation}\label{crossedmodule1}
\partial(\chi(h)g)=[h,\partial(g)]^{\mathfrak h},\,\,\,\,
\mbox{\rm for all} \,\,\,g \in\mathfrak g, \, h \in \mathfrak h
\end{equation}
and
\begin{equation}\label{crossedmodule2}
\chi(\partial(g_1))g_2= [g_1,g_2]^{\mathfrak g},\,\,\,\, \mbox{\rm for all} \,\,\,g_1 ,g_2\in \mathfrak g.
\end{equation}
\end{defn}
\noindent Such a crossed module will be denoted by $(\mathfrak g,\,\mathfrak h,\,\partial,\,\chi)$.

\

From a Lie $2$-algebra with vanishing vector valued $3$-form, we may get a crossed module of Lie algebras.
\begin{prop} \cite{BaezCrans}
Let $(E=E_{-2}\oplus E_{-1},\,\mu=l_1+l_2+l_3)$ be a Lie $2$-algebra, with corresponding quadruple $(\partial,\chi, [.,.]_2,\omega)$ given by (\ref{omega}) and (\ref{chi}).
If $\omega=0$, then $(E_{-2}, E_{-1}, \partial, \chi)$ is a crossed module of Lie algebras.
\end{prop}

Proposition \ref{corlast} provides the construction of Nijenhuis forms on Lie $n$-algebras. However, for the case $n=2$, that proposition does not give the possibility of having a Nijenhuis vector valued $2$-form. We intend to give an example of Nijenhuis vector valued form with respect to a Lie $2$-algebra structure $\mu$ on a graded vector space $E_{-2}\oplus E_{-1}$ which is not purely a $1$-form, i.e. not just a collection of maps from $E_i$ to $E_i$, $i=1,2$. As we have mentioned before, elements of degree zero in $\tilde{S}(E^*) \otimes E $ are necessarily of the form $N + \alpha$ with $N: E \to E$ a linear endomorphism preserving the degree and $\alpha : E \times E \to E$
a symmetric vector valued $2$-form of degree zero.
\begin{thm}\label{thm:alphalpha}
Let $\mu=l_1+l_2+l_3$ be a Lie $2$-algebra structure on a graded vector space $E=E_{-2}\oplus E_{-1}$ and $\alpha$ a symmetric vector valued $2$-form of degree zero. Then,  $S+\alpha$ is a Nijenhuis vector valued form with respect to $\mu$, with square of $S+2\alpha $, if and only if
 \begin{equation}  \label{Lie2Nijenhuis}
  \alpha(l_1 (\alpha (X,Y)),Z ) +c.p.=0,
  \end{equation}
 for all $ X,Y,Z \in E_{-1}$.
\end{thm}
\begin{proof}
By degree reasons, the only case where the vector valued $3$-form $[\alpha,[\alpha, l_1]_{_{RN}}]_{_{RN}}$ is not identically zero is when it is evaluated on elements of $E_{-1}$.  In this case, we get
\begin{equation}\label{eq:Salpha0}
   \begin{array}{rcl}
     [\alpha,[\alpha, l_1]_{_{RN}}]_{_{RN}}(X,Y,Z)&=& [\alpha, l_1]_{_{RN}}(\alpha(X,Y),Z)+c.p. \\ & & -\alpha([\alpha, l_1]_{_{RN}}(X,Y),Z)+c.p.\\
                                    &=& -2\alpha(l_1(\alpha(X,Y)),Z)+c.p.,\\
 \end{array}
\end{equation} for all $X,Y,Z \in E_{-1}$.
Again by degree reasons, $[\alpha,[\alpha, l_2]_{_{RN}}]_{_{RN}}$ and $[\alpha, l_3]_{_{RN}}$ are identically zero. So, we have
\begin{eqnarray}\label{eq:Salpha}
    \lefteqn{[S+\alpha,[S+\alpha,l_1+l_2+l_3]_{_{RN}}]_{_{RN}}=}\nonumber \\
                                            &=& [S+\alpha,l_1+l_2+l_3+[\alpha, l_1]_{_{RN}}+[\alpha, l_2]_{_{RN}}]_{_{RN}} \nonumber \\
                                            &=&l_1+l_2+l_3+2[\alpha, l_1]_{_{RN}}+2[\alpha, l_2]_{_{RN}}+[\alpha,[\alpha, l_1]_{_{RN}}]_{_{RN}} \nonumber \\
                                            &=&[S+2\alpha,l_1+l_2+l_3]_{_{RN}}+[\alpha,[\alpha, l_1]_{_{RN}}]_{_{RN}}.
\end{eqnarray}
On the other hand, Lemma \ref{Euler}  and Equation (\ref{RNbracket}) imply that
\begin{equation}\label{Salpha2}
[S+\alpha,S+2\alpha]_{_{RN}}=0.
\end{equation}
Equations (\ref{eq:Salpha0}), (\ref{eq:Salpha}) and (\ref{Salpha2}) show that $S+\alpha$ is a Nijenhuis vector valued form with respect to $\mu$, with square $S+2\alpha$, if and only if $\alpha(l_1(\alpha(X,Y)),Z)+c.p.=0$, for all $X,Y,Z \in E_{-1}$.
\end{proof}

\begin{cor}
Let $\mu=l_1+l_2+l_3$ be a Lie $2$-algebra structure on a graded vector space $E=E_{-2}\oplus E_{-1}$, with $l_1=0$. Then, for every vector valued $2$-form $\alpha$ of degree zero, $S+\alpha$ is a Nijenhuis vector valued form with respect to $\mu$, with square $S+2\alpha$.
\end{cor}

Combining Theorems \ref{thm:alphalpha} and \ref{theo:Hierarchy} we get the following proposition.
\begin{prop}\label{prop:inverseNijenhuis}
Let $\mu=l_1+l_2+l_3$ be a Lie $2$-algebra structure on a graded vector space $E=E_{-2}\oplus E_{-1}$.
Let $\alpha$ be a vector valued $2$-form of degree zero such that $ \alpha(l_1 (\alpha (X,Y)),Z ) +c.p.=0,$
 for all $ X,Y,Z \in E_{-1}$.
Let $\mu_k$ stand for the vector valued form defined by  $\mu_k=[S+\alpha,[S+\alpha,\cdots ,[S+\alpha,\mu]_{_{RN}} \cdots]_{_{RN}}]_{_{RN}}$, with $k$ copies of $S+\alpha$.
Then, $S+\alpha$ is a Nijenhuis vector valued form with respect to all the terms of the hierarchy of successive deformations $\mu_k$, with square $S+2\alpha $.
\end{prop}

If $\mu=l_1+l_2+l_3$ is a Lie $2$-algebra on $E=E_{-2}\oplus E_{-1}$ with $l_1=0$, then $[.,.]_2$, given by (\ref{chi}), is a Lie bracket on $E_{-1} $. Also, the condition $[l_2,l_3]_{_{RN}}=0$ means that $l_3$  is a Chevalley-Eilenberg-closed $3$-form
      of this Lie algebra $E_{-1}$ valued in $E_{-2}$.
      This kind of Lie $2$-algebras are usually called \emph{string Lie algebras}.
A Lie $2$-algebra $(E_{-2}\oplus E_{-1}, l_1+l_2+l_3)$ with $l_2=l_3=0$ and $l_1$ invertible, is called a {\em trivial} Lie $2$-algebra.
The next example is an application of Theorem \ref{thm:alphalpha} to a trivial Lie $2$-algebra.

\begin{ex}
Let $\mathfrak{g}$ be a vector space and $[.,.]_{\mathfrak{g}}$ be a skew-symmetric bilinear map on $\mathfrak{g}$. Let $E_{-1}:=\{-1\}\times\mathfrak{g}$, $E_{-2}:=\{-2\}\times\mathfrak{g}$ and let $\partial:E_{-2}\to E_{-1}$ be given by $(-2,x)\mapsto (-1,x)$. Define $\alpha:E_{-1}\times E_{-1}\to E_{-2}$ to be vector valued $2$-form on the graded vector space $E=E_{-2}\oplus E_{-1}$ as $((-1,x),(-1,y))\mapsto (-2,[x,y]_{\mathfrak{g}})$. Then,  as a direct consequence of Theorem \ref{thm:alphalpha}, we have that $S+\alpha$ is Nijenhuis with respect to $\partial$ if and only if $[.,.]_{\mathfrak{g}}$ is a Lie bracket.
\end{ex}

Let us  now look at the deformation of a Lie $2$-algebra structure.

\begin{prop}\label{prop:deformed}
Let $\mu=l_1+l_2+l_3$ be a Lie $2$-algebra structure on a graded vector space $E=E_{-2}\oplus E_{-1}$, with associated quadruple $(\partial, [.,.]_2,\chi,\omega)$. Let $\alpha$ be a symmetric vector valued $2$-form of degree zero on $E$ and set ${\mathcal N}= S+\alpha$. The deformed structure $\mu^{\mathcal N}$ is associated to the quadruple $(\partial', [.,.]'_2,\chi',\omega')$:
\begin{equation}\label{deformedLie2algebra}
\begin{array}{rcl}
\partial' f&=& \partial f,\\
\left[X,Y\right]'_{2} &=&[X,Y]_2+\partial (\alpha(X,Y)), \\
\chi'(X)f   &=& \chi(X)f -\alpha(\partial f,X),\\
\omega'(X,Y,Z)&=& \omega(X,Y,Z)+{\diff}^{CE}\alpha(X,Y,Z),
\end{array}
\end{equation}
for all $X,Y,Z\in E_{-1}$ and $f\in E_{-2}$. If $\alpha$ satisfies (\ref{Lie2Nijenhuis}), $\mu^{\mathcal N}$ is a Lie $2$-algebra structure on $E$.
\end{prop}
\begin{proof}
The first part of the statement follows from the following easy relations:
\begin{itemize}
\item[] $[S+\alpha,\mu]_{_{RN}}=l_1+(l_2+[\alpha,l_1]_{_{RN}})+(l_3+[\alpha,l_2]_{_{RN}})$;
\item[] $[\alpha,l_1]_{_{RN}}(X,Y)=l_1(\alpha(X,Y)),\,\,\, \text{for all}\,\, X,Y \in E_{-1}$;
\item[] $[\alpha,l_1]_{_{RN}}(X,f)=-\alpha(l_1(f),X),\,\,\, \text{for all}\,\, X \in E_{-1},f \in E_{-2}$;
\item[] $[\alpha,l_2]_{_{RN}}={\diff}^{CE}\alpha.$
\end{itemize}
The second part is a direct consequence of Theorem \ref{thm:alphalpha}.
\end{proof}

Notice that, in the case of  Proposition \ref{prop:deformed}, the vector valued form $S-\alpha$ has the inverse effect of $S +\alpha$, that is, $[S-\alpha,[S+\alpha,\mu]_{_{RN}}]_{_{RN}} =\mu $.

As we have seen previously,  string Lie algebras on $E_{-2}\oplus E_{-1}$ are in one to one correspondence with Lie algebra structures on ${\mathfrak g}:=E_{-1}$
together with a representation of the Lie algebra ${\mathfrak g}$ on the vector space $V:=E_{-2}$ and a Chevalley-Eilenberg $3$-cocycle $\omega$ for this representation. Hence, we denote
 string Lie algebras  as triples $({\mathfrak g},V,\omega)$.
 According to Proposition \ref{prop:deformed}, the deformation of a string Lie algebra $({\mathfrak g},V,\omega)$ by $S+\alpha $,
just amounts to change the $3$-cocycle $\omega$ into $\omega + \diff^{CE} \alpha$.
So that, for string Lie algebras, adding up a coboundary, i.e., changing $({\mathfrak g},V,\omega)$
into  $({\mathfrak g},V,\omega+ \diff^{CE}\alpha)$  can be seen as a Nijenhuis transformation by
$S+\alpha$.

\

A \emph{Lie $2$-subalgebra} of a Lie $2$-algebra $(E=E_{-2}\oplus E_{-1}, \mu=l_1+l_2+l_3)$ is a Lie $2$-algebra $(E^\prime=E^\prime_{-2}\oplus E^\prime_{-1}, \mu'=l'_1+l'_2+l'_3)$ with $E^\prime_{-2}\subset E_{-2}$ and $E^\prime_{-1}\subset E_{-1}$ vector subspaces,
\begin{equation*}
l_1^{\prime}=l_1|_{E^{\prime}},\,\,l_2^{\prime}=l_2|_{E^{\prime}\times E^{\prime}}\,\,\,\mbox{and}\,\,\,l_3^{\prime}=l_3|_{E^{\prime}\times E^{\prime}\times E^{\prime}}.
\end{equation*}

\

Let us now investigate Lie $2$-algebras structures for which $\chi=0$.
There may be quite a few such Lie $2$-algebras but we are going to show that, after a Nijenhuis transformation of the form
$S+\alpha$, such Lie $2$-algebras will be decomposed as a direct sum of
a string Lie algebra with a trivial Lie $2$-algebra.

\begin{prop}\label{prop:strict+trivial}
 Given a Lie $2$-algebra structure $l_1+l_2+l_3$ on a graded vector space $E=E_{-2}\oplus E_{-1}$ and corresponding quadruple $(\partial, [.,.]_2, \chi,\omega )$, with $\chi=0$, there exists a Nijenhuis
 form $S+\alpha$, with $\alpha$ a vector valued $2$-form of degree zero, such that the deformed bracket $[S+\alpha,l_1+l_2+l_3]_{_{RN}} $
is the direct sum of a string Lie $2$-algebra with a trivial $L_\infty$-algebra.
\end{prop}
\begin{proof}
We set $E_{-1}^{t}:={\rm Im} (\partial)$, $E_{-2}^s := {\rm Ker}(\partial) $ and we choose two subspaces $E_{-2}^t \subset E_{-2}$
and $ E_{-1}^{s} \subset E_{-1}$ such that the following are direct sums: $E_{-2}^t\oplus E_{-2}^s=E_{-2}$ and $E_{-1}^t\oplus E_{-1}^s=E_{-1}$. Since $\chi=0$, by (\ref{Lie2relations5}),
the bracket $[ .,. ]_2$ vanishes on $E_{-1} \times E_{-1}^{t}$; so that, there exists a unique
skew-symmetric bilinear map $\alpha: E_{-1} \times E_{-1} \to E^{t}_{-2}$ such that
\begin{equation}  \label{defn_alpha}
 \partial \alpha (X,Y) = -pr_{E_{-1}^{t}} ([X,Y]_2),  \hbox{ for all}\,\,\, X,Y \in E_{-1},
 \end{equation}
where $pr_{E_{-1}^{t}} $ stands for the projection on ${E_{-1}^{t}}$ with respect to $E_{-1}^{s}$.
Notice that $\alpha$ is unique. In fact, if $\alpha': E_{-1} \times E_{-1} \to E^{t}_{-2}$ is another skew-symmetric bilinear map satisfying (\ref{defn_alpha}), then $(\alpha - \alpha')(X,Y) \in {\rm Ker}(\partial)= E_{-2}^s$, for all $X,Y \in E_{-1}$. Since $(\alpha - \alpha')(X,Y)$ is also an element of $E_{-2}^t$ and  $E_{-2}=E_{-2}^t\oplus E_{-2}^s$ is a direct sum, we have that $(\alpha - \alpha')(X,Y)=0$, for all $X,Y \in E_{-1}$, and so $\alpha=\alpha'$.
Note that we also have
\begin{equation} \label{condition_alpha_zero}
\alpha(X,Y)=0 \,\,\,{\textrm{if}} \,\,\, X \,\,\,{\textrm{or}} \,\,\,  Y \,\,\,{\textrm{belong to}} \,\,\,  E_{-1}^t,
\end{equation}
so that $\alpha(\partial \alpha (X,Y),Z)=0$, for all $X,Y,Z \in E_{-1}$. Hence, by Theorem \ref{thm:alphalpha},
$S+\alpha$
is a Nijenhuis form with square $S+2\alpha $ and, by Proposition~\ref{prop:deformed}, the deformed bracket $l_1'+l_2'+l_3':=[S+\alpha, l_1+l_2+l_3]_{_{RN}}$ is a Lie $2$-algebra structure on $E=E_{-2}\oplus E_{-1}$.

We claim that
$(E_{-1}^s \oplus E_{-2}^s,l_1'^s+l_2'^s+l_3'^s)$ and $(E_{-1}^t \oplus E_{-2}^t ,l_1'^t+l_2'^t+l_3'^t)$ are  Lie $2$-subalgebras of $(E_{-2}\oplus E_{-1}, l_1'+l_2'+l_3')$, where $l_i'^s$  and $l_i'^t$ stand for the restrictions of $l_i'$ to $E_{-1}^s \oplus E_{-2}^s$ and $E_{-1}^t \oplus E_{-2}^t$, respectively. We also claim that $(E_{-1}^s \oplus E_{-2}^s,l_1'^s+l_2'^s+l_3'^s)$ is a string Lie $2$-algebra
while $(E_{-1}^t \oplus E_{-2}^t ,l_1'^t+l_2'^t+l_3'^t)$ is a trivial Lie $2$-algebra, and that their direct sum is isomorphic to $(E_{-2}\oplus E_{-1},l_1'+l_2'+l_3')$.

Let $(\partial^\prime,[.,.]^\prime,\chi^\prime,\omega^\prime)$ stand for the quadruple associated to the deformed structure $l_1'+l_2'+l_3'$. First, we prove that $(E_{-1}^t \oplus E_{-2}^t ,l_1'^t+l_2'^t+l_3'^t)$ is a Lie $2$-subalgebra of $(E_{-2}\oplus E_{-1}, l_1'+l_2'+l_3')$ with $l_2'^t=l_3'^t=0$, hence it is a trivial Lie $2$-subalgebra. We use the explicit expressions given in Proposition \ref{prop:deformed} in the case  $\chi=0$. Let $f \in E_{-2}^t$ and  $X, Y, Z \in E_{-1}^t$. Then, $l_1'^t(f)=\partial^\prime (f)=\partial(f)$; thus $l_1'^t(E_{-2}^t) \subset E_{-1}^t$. Moreover, using (\ref{condition_alpha_zero}) and the fact that $[ .,. ]_2$ vanishes on $E_{-1} \times E_{-1}^{t}$, we get
\begin{center}
$l_2'^t(X,f)=\chi^\prime(X)f=-\alpha(\partial f, X)=0$,
\end{center}
\begin{center} $l_2'^t(X,Y)= [X,Y]_2^\prime= [X,Y]_2+ \partial(\alpha(X,Y))=0$,
\end{center}
so that $l_2'^t=0$. As for $l_3'^t$, we have
$$l_3'^t(X,Y,Z)= \omega^\prime (X,Y,Z)= \omega(X,Y,Z) + {\diff}^{CE}\alpha(X,Y,Z).$$
Now, $\omega(X,Y,Z)=0$ by (\ref{Lie2relations6}) and
${\diff}^{CE}\alpha(X,Y,Z) \stackrel{(\ref{diff_RN})}{=}[\alpha,l_2]_{_{RN}}(X,Y,Z)=0,$
by (\ref{condition_alpha_zero}) and because $l_2$ vanishes on $E_{-1}^{t} \times E_{-1}^{t}$.
 We therefore obtain $l_3'^t=0$, which completes the proof of the fact that $(E_{-1}^t \oplus E_{-2}^t ,l_1'^t+l_2'^t+l_3'^t)$ is a Lie $2$-subalgebra of $(E_{-2}\oplus E_{-1}, l_1'+l_2'+l_3')$ with $l_2'^t=l_3'^t=0$.

Next we prove that $(E_{-2}^s\oplus E_{-1}^s, l_1'^s+l_2'^s+l_3'^s)$ is a Lie $2$-subalgebra of $(E_{-2}\oplus E_{-1}, l_1'+l_2'+l_3')$ with $l_1'^s(E_{-2}^s)=0$ and hence it is a string Lie subalgebra. We use again Proposition~\ref{prop:deformed} with $\chi=0$. By definition of $E_{-2}^s$,   $$l_1'^s(E_{-2}^s)=\partial^\prime (E_{-2}^s) =\partial(E_{-2}^s)=0$$ holds.  Let $f \in E_{-2}^s={\rm Ker}(\partial)$ and $X, Y, Z \in E_{-1}^s$. Then,  we have
$
l_2'^s(X,f)=\chi^\prime(X)f=-\alpha(\partial f, X)=0
$
and \begin{equation}  \label{strictpartofl-0}
l_2'^s(X,f) \in E_{-1}^s.
\end{equation}
 Also,
\begin{equation*}
l_2'^s(X,Y)=[X,Y]_2^\prime=[X,Y]_2+\partial\alpha(X,Y)=[X,Y]_2-pr_{E_{-1}^{t}} ([X,Y]_2),
\end{equation*}
which implies that
\begin{equation}\label{strictpartofl-2}
l_2'^s(X,Y)\in E_{-1}^s.
\end{equation}
From (\ref{Lie2relations7}), and taking into account that $\partial=\partial^\prime$,
we get
\begin{equation}\label{strictpartofl-3}
\partial (\omega^\prime (X,Y,Z))=l_2'(l_2'(X,Y),Z) + c.p..
\end{equation}
Using Relation (\ref{strictpartofl-2}), the right hand side of Equation (\ref{strictpartofl-3}) belongs to $E_{-1}^s$, while according to the definition of $E_{-1}^t$, the left hand side of Equation (\ref{strictpartofl-3}) belongs to $E_{-1}^t$ and since $E_{-1}=E_{-1}^t\oplus E_{-1}^s$ is a direct sum, both sides of Equation (\ref{strictpartofl-3}) should be zero. This implies that
\begin{equation}\label{strictpartlast}
l_3'(X,Y,Z)\in E_{-2}^s.
\end{equation}
 Relations (\ref{strictpartofl-0}) and (\ref{strictpartofl-2}) together with Equation (\ref{strictpartlast}) show that $(E_{-2}^s\oplus E_{-1}^s, l'_1+l'_2+l'_3)$ is a Lie $2$-subalgebra. This completes the proof.
\end{proof}

Next, it is interesting to see that Lie algebras themselves can be seen as Nijenhuis forms.
 We start by noticing that any vector valued $2$-form of degree zero on a graded vector space $E_{-2}\oplus E_{-1}$ is of the form
\begin{equation}\label{khouk}
\alpha(X,Y)=\begin{cases}-\alpha(Y,X),& \,\,{\mbox{if}} \, \,X,Y \in E_{-1},\\
                          0,         & \,\,{\mbox{otherwise}}.
            \end{cases}
\end{equation}
This, together with the fact that $\alpha$ always takes value in $E_{-2}$, imply that
\begin{equation}\label{sag}
\alpha(\alpha(X,Y),Z)+c.p.=0,
\end{equation}
for all $X,Y,Z \in E_{-1}$.
Equations (\ref{khouk}) and (\ref{sag}) mean that any symmetric vector valued $2$-form $\alpha$ on an arbitrary graded vector space $E_{-2}\oplus E_{-1}$ is a Lie algebra (not a graded Lie algebra).
 In the next proposition, we show that there is also a way to get a Lie bracket on a graded vector space $E=E_{-2}\oplus E_{-1}$ from a Nijenhuis form with respect to a Lie $2$-algebra structure $\mu=l_1+l_2+l_3$ on the vector space $E$.

 \begin{prop}
 Let $(E=E_{-2}\oplus E_{-1},\, \mu=l_1+l_2+l_3)$ be a Lie $2$-algebra, with corresponding quadruple $(\partial,[.,.]_2,\chi,\omega)$. Let $\alpha$ be a vector valued $2$-form of degree zero and define a bilinear map $\tilde{\alpha}$ by setting
 \begin{equation*}
 \tilde{\alpha}(X,Y)=\begin{cases}\alpha(X,Y), & \,{\mbox{for}} \,\, X,Y \in E_{-1},\\
                                  \alpha(\partial X,Y), & \,{\mbox{for}} \,\, X\in E_{-2},Y \in E_{-1},\\
                                  \alpha(X,\partial Y), & \,{\mbox{for}} \,\,  X \in E_{-1},Y\in E_{-2},\\
                                  \alpha(\partial X,\partial Y), & \,{\mbox{for}} \,\, X,Y \in E_{-2}.
                     \end{cases}
  \end{equation*}
  Then, $S+\alpha$ is Nijenhuis vector valued $2$-form with respect to $\mu$, with square $S+2\alpha$, if and only if $(E,\tilde{\alpha})$ is a Lie algebra.
  \end{prop}
  \begin{proof}
  By definition, $\tilde{\alpha}$ is a skew-symmetric bilinear map on the vector space $E$ and we have
  \begin{equation}\label{jacobialpha}\begin{array}{rcl}
  \tilde{\alpha}(\tilde{\alpha}(X,Y),Z)+c.p.&=&\alpha(\partial\alpha(X,Y),Z)+c.p.,\\
  \tilde{\alpha}(\tilde{\alpha}(f,Y),Z)+c.p.&=&\alpha(\partial\alpha(\partial f,Y),Z)+c.p.,\\
  \end{array}
  \end{equation}
  for all $X,Y,Z \in E_{-1}$ and $f\in E_{-2}$. Hence, Theorem \ref{thm:alphalpha} together with (\ref{jacobialpha}) imply that $\tilde{\alpha}$ is a Lie bracket on the vector space $E$ if, and only if, $S+\alpha$ is a Nijenhuis form with respect to $\mu$, with square $S+2\alpha$.
 \end{proof}
 Last, we give a result involving weak Nijenhuis forms on a Lie $2$-algebra.
 \begin{prop}
  Let $\partial: E_{-2}\to E_{-1}$ be a Lie $2$-algebra structure on a graded vector space $E=E_{-2}\oplus E_{-1}$, that is, a Lie $2$-algebra structure $\mu=l_1+l_2+l_3$ on $E$, with $l_1=\partial$ and $l_2=l_3=0$. Let $\alpha$ be a symmetric vector valued $2$-form of degree zero on the graded vector space $E$. If $S+\alpha$ is a weak Nijenhuis vector valued form with respect to $\partial$, then $E_{-1}$ is a Lie algebra with a representation on $E_{-2}$.
  \end{prop}

\begin{proof}
  According to Proposition~\ref{prop:deformedisLinfty}, $S+\alpha$ is a weak Nijenhuis vector valued form with respect to $\partial$ if and only if $[S+\alpha,\partial]_{_{RN}}$ is an $L_{\infty}$-structure on the graded vector space $E$ which, in turn, is equivalent to
 \begin{equation*}
 [[S+\alpha,\partial]_{_{RN}},[S+\alpha,\partial]_{_{RN}}]_{_{RN}}=0
 \end{equation*}
 or to
 \begin{equation*}\label{partial}
 [[\alpha,\partial]_{_{RN}},[\alpha,\partial]_{_{RN}}]_{_{RN}}=0.
 \end{equation*}
 Therefore, $S+\alpha$ is a weak Nijenhuis vector valued form with respect to $\partial$ if and only if
 \begin{equation}\label{eq:weakexample1}
 \partial\alpha(\partial\alpha(X,Y),Z)+c.p.(X,Y,Z)=0
 \end{equation}
 and
 \begin{equation}\label{eq:weakexample2}
 \alpha(\partial\alpha(X,Y),\partial f)+c.p.(X,Y,\partial f)=0,
 \end{equation}
 for all $X,Y,Z \in E_{-1}$ and $f\in E_{-2}$. Equation (\ref{eq:weakexample1}) means that $[X,Y]:=\partial\alpha(X,Y)$ defines a Lie bracket on  $E_{-1}$ since clearly it is skew-symmetric. If we define a map $\cdot:E_{-1}\times E_{-2}\to E_{-2}$ by setting  $ X\cdot f:=\alpha(X,\partial f)$, then (\ref{eq:weakexample2}) can be written as
  \begin{equation*}
  [X,Y]\cdot f= X\cdot(Y\cdot f)-Y\cdot(X\cdot f),
  \end{equation*}
  which means that $\cdot$  is a representation of $E_{-1}$ on $E_{-2}$.
 \end{proof}

 \begin{rem} \label{sheng}
 A notion of Nijenhuis operator on a Lie $2$-algebra independently appeared
in \cite{LiuSheng}, while the present paper was about to be completed. This notion is a particular case of ours, by the following reasons. First, in \cite{LiuSheng}, a Nijenhuis operator is necessarily a vector valued $1$-form. Second,
if $\mathcal N = (N_0,N_1)$ is a Nijenhuis operator in the sense of Definition 3.2. in \cite{LiuSheng}, with respect to a Lie $2$-algebra $l_1+l_2+l_3$, then
\begin{equation*}
[ \mathcal N, [ \mathcal N,l_i]_{_{RN}}]_{_{RN}}=[ \mathcal N^2, l_i ]_{_{RN}}
\end{equation*}
holds for $i= 1,2$ and $3$,
which means that $\mathcal N$ is a Nijenhuis vector valued form, in our sense, with square $\mathcal N^2$.
 \end{rem}

 %%%%%%%%%%%%%%%%%%%%%%%%%%%%%%%%%%%%%%%%%%%%%%%%%%%%%%%%%%%%%%%%%%%%%%%%%%%%%%%%%%%%%%%%%%%%%%%%%%%%%%%%%%%%%%%%%%%%%%%%%%%%%%%%%%%%%%%%%%%%%%%
 %%%%%%%%%%%%%%%%%%%%%%%%%%%%%%%%%%%%%%%%%%%%%%%%%%%%%%%%%%%%%%%%%%%%%%%%%%%%%%%%%%%%%%%%%%%%%%%%%%%%%%%%%%%%%%%%%%%%%%%%%%%%%%%%%%%%%%%%%%%%%%%%%%
 %%%%%%%%%%%%%%%%%%%%%%%%%%%%%%%%%%%%%%%%%%%%%%%%%%%%%%%%%%%%%%%%%%%%%%%%%%%%%%%%%%%%%%%%%%%%%%%%%%%%%%%%%%%%%%%%%%%%%%%%%%%%%%%%%%%

 \section{Nijenhuis forms on Courant algebroids}

We recall that one can associate a Lie $2$-algebra to a Courant algebroid \cite{RoytenbergWeinstein}. We use this construction to see how $(1,1)$-tensors on a Courant algebroid, with vanishing Nijenhuis torsion, are related with Nijenhuis forms with respect to the associated Lie $2$-algebra.

 \begin{defn}\label{def:Courantdorfmanshort}
  A {\em Courant algebroid} is a vector bundle $E\to M$ together with a non-degenerate inner product  $\langle .,. \rangle$, a morphism of vector bundles $\rho: E \to TM$ and a bilinear operator $\circ: \Gamma(E) \times \Gamma(E) \to \Gamma(E)$, such that the following axioms hold:
  \begin{enumerate}
    \item[(i)]  $(\Gamma(E),\circ)$ is a Leibniz algebra, i.e., $X\circ (Y\circ Z)=(X\circ Y)\circ Z+Y\circ (X\circ Z) $,
    \item[(ii)] $\rho(X)\langle Y,Z\rangle = \langle X \circ Y, Z\rangle +\langle Y,X\circ Z\rangle $,
    \item[(iii)] $\rho(X)\langle Y,Z\rangle = \langle X,Y \circ Z\rangle +\langle X,Z\circ Y\rangle $,
  \end{enumerate}
  for all $ X,Y,Z \in \Gamma(E)$.
  \end{defn}

  When item $(i)$ in Definition~\ref{def:Courantdorfmanshort} does not hold, the quadruple $(E, \circ, \rho, \langle . , . \rangle )$ is called a \emph{pre-Courant algebroid} \cite{CJP}.

The next proposition is stated in \cite{YKSquasi}, for Courant algebroids. Since the proof does not use the fact of $\circ$ being a Leibniz bracket, the result also holds for pre-Courant algebroids.
  \begin{prop}\label{prop:fmiadbiroonupto}
  For every pre-Courant algebroid $(E, \circ, \rho, \langle .,. \rangle)$ we have
  \begin{equation*}\label{LeibnizforpreCourantholds}
  X\circ (f Y)=f (X\circ Y) + (\rho(X)f)Y,
  \end{equation*}
  for all $X,Y \in \Gamma(E)$ and $f\in \mathcal{C}^{\infty}(M).$
  \end{prop}

  \begin{cor}\label{cor:Anchoruniqe}
  Let $(E, \circ, \rho, \langle .,. \rangle)$ and $(E, \circ', \rho', \langle .,. \rangle)$ be two pre-Courant algebroids. If $\circ=\circ'$, then $\rho=\rho'$.
  \end{cor}
  \begin{proof}
  Assume that $(E, \circ, \rho, \langle .,. \rangle)$ and $(E, \circ, \rho', \langle .,. \rangle)$ are both pre-Courant algebroids. By Proposition \ref{prop:fmiadbiroonupto} we have
  \begin{equation*}
   (\rho(X)f)Y=(\rho'(X)f)Y,
  \end{equation*}
  for all $X,Y\in \Gamma(E)$ and $f\in \mathcal{C}^{\infty}(M)$, which implies that $\rho=\rho'$.
  \end{proof}

  We intend to define Nijenhuis deformation of Courant structures. Let $(E,\circ, \rho,\langle .,.\rangle)$ be a Courant algebroid. For a given endomorphism $N:E \to E$, the deformed bracket by $N$ is a bilinear operation $\circ^{N}$, defined as:
  $$X\circ^{N}Y:=NX\circ Y+X\circ NY -N(X\circ Y),$$
  for all $X,Y \in \Gamma(E)$. The deformation of $\rho$ by $N$ is the map $\rho^N$ given by $\rho^N(X)=\rho(NX)$, $X\in \Gamma(E)$. The Nijenhuis torsion of $N$, with respect to the bracket $\circ$, is defined as:
  \begin{equation*}
  T_{\circ}N(X,Y):=NX\circ NY-N(X\circ^{N}Y),
  \end{equation*}
  for all $X,Y \in \Gamma(E)$.
  A direct computation shows that
  $$T_{\circ}N=\frac{1}{2}(\circ^{N,N}-\circ^{N^{2}}).$$
   All maps $N:\Gamma(E)\to \Gamma(E)$ that will be considered here are $\mathcal{C}^{\infty}(M)$-linear, that is to say they are $(1,1)$-tensors, that is, smooth sections of endomorphisms of $E$. We denote an endomorphism (vector bundle morphism) of $E$ and the induced map on $\Gamma(E)$ by the same letter.

 According to \cite{CGM}, for every vector bundle $E\to M$, if $(\Gamma(E),\circ)$ is a Leibniz algebra and $N: E \to E$ is any endomorphism  whose Nijenhuis torsion vanishes, then the pair $(\Gamma(E),\circ^N)$ is a Leibniz algebra.
 However, given a Courant algebroid $(E, \circ, \rho, \langle .,.\rangle )$ and a
 $(1,1)$-tensor $N$, $(E, \circ^N, \rho^N, \langle .,.\rangle )$ may fail to be a pre-Courant algebroid, even if the Nijenhuis torsion of $N$ vanishes.
 Indeed,  from \cite{CGM} we have the following:

 \begin{thm}\label{them:Grobowski}
 If $N$ is an endomorphism on a pre-Courant algebroid $(E, \circ, \rho, \langle .,.\rangle )$, then the quadruple $(E , \circ^N, \rho^N, \langle .,.\rangle )$ is a pre-Courant algebroid if and only if
 \begin{equation*}
 X \circ (N+N^*)Y=(N+N^*)(X \circ Y) \hbox{ and } (N+N^*)(Y \circ Y)=((N+N^*) Y) \circ Y
 \end{equation*}
 for all $X,Y \in \Gamma (E)$, where $N^*$ stands for the transpose of $N$, with respect to $\langle.,.\rangle.$
 \end{thm}

 \begin{rem}
 In fact, Theorem \ref{them:Grobowski} is slightly different from Theorem $4$ in \cite{CGM}, because there, the authors start from a Courant
 algebroid. However, the same proof is valid for the pre-Courant algebroid case.
 \end{rem}

  A \emph{Casimir function} or simply a Casimir on a Courant algebroid $(E, \circ, \rho, \langle . , .\rangle )$ is a function $f\in \mathcal{C}^{\infty}(M)$ such that $\rho(X)f=0 $, for all $X\in \Gamma(E)$. It is easy to check that $f$ is a Casimir if and only if $\mathcal{D}f=0$, where $\mathcal{D}:\mathcal{C}^{\infty}(M)\to \Gamma(E) $ is given by
  \begin{equation}\label{matcallD}
\langle \mathcal{D}f,X\rangle =\rho (X)f.
\end{equation} Also, if $f$ is a Casimir, then
  \begin{equation}\label{eq:CasimirGetsOut}(f X) \circ Y = f (X \circ Y) = X \circ (fY) \end{equation}
holds for all sections $X,Y \in \Gamma(E)$.

 The next lemma is a slight generalization of a result in \cite{CGM}.\footnote{In \cite{CGM},  $\lambda$ is a real number.}
  \begin{lem}\label{pre}
  Given a pre-Courant algebroid $(E, \circ, \rho, \langle .,.\rangle )$ and a map $N:\Gamma(E) \to \Gamma(E)$,
  if \, $N+N^*=\lambda Id_{\Gamma(E)}$, for some Casimir function $\lambda \in \mathcal{C}^{\infty}(M)$, then $(E, \circ^N, \rho^N, \langle .,.\rangle )$ is a pre-Courant algebroid.
  \end{lem}
  \begin{proof}
  This lemma is a direct consequence of Theorem \ref{them:Grobowski} together with (\ref{eq:CasimirGetsOut}).
  \end{proof}
  \begin{thm}\label{Theml^N}
Let $ (E , \circ, \rho, \langle .,.\rangle ) $ be a Courant algebroid and $N$ a $(1,1)$-tensor on $E$ whose Nijenhuis torsion vanishes and such that
 \begin{equation*}
    N+N^*=\lambda Id_{\Gamma(E)},
 \end{equation*}
  with $\lambda$ being a Casimir function. Then, $ (E, \circ^N, \rho^N, \langle . ,. \rangle ) $ is a Courant algebroid.
  \end{thm}
\begin{proof}
 Note that $(E,\circ)$ is a Leibniz algebra, so that $(E,\circ^N)$ is also a Leibniz algebra since the Nijenhuis torsion of $N$ vanishes. This, together with Lemma \ref{pre}, prove the theorem.
 \end{proof}

  \begin{rem}
  For a (pre-)Courant algebroid $(E,\circ,\rho,\langle .,.\rangle)$, and a $(1,1)$-tensor $N$ on $E$ with $N+N^*=\lambda Id_{\Gamma(E)}$ and $\lambda$ a Casimir function, we have
  \begin{equation*}
  \rho^{N}(X)f=\rho(NX)f=\langle NX,\mathcal D f\rangle=\langle X,N^*\mathcal D f\rangle=\langle X,(-N+\lambda Id_{\Gamma(E)})\mathcal D f\rangle,
  \end{equation*}
for all $X\in \Gamma(E), f\in \mathcal{C}^{\infty}(M)$.  This means that the
  operator $ \mathcal{D}^N : \mathcal{C}^\infty(M) \to \Gamma(E)$ associated with the (pre-)Courant algebroid $(E,\circ^N,\rho^N,\langle .,.\rangle)$, is given by
  \begin{equation} \label{deformedD}
  \mathcal {D} ^N=(-N+\lambda Id_{\Gamma(E)}) \circ \mathcal D.
  \end{equation}
  \end{rem}

  If we consider the skew-symmetrization of $\circ$, we obtain the bracket $[.,.]$ used in the original definition of Courant algebroid \cite{LWX}:
\begin{equation} \label{skewbracket}
[X, Y]= \frac{1}{2} (X \circ Y - Y \circ X),
\end{equation}
with $X,Y \in \Gamma(E)$. The deformation of $[.,.]$ by a $(1,1)$-tensor $N$ on $E$ is the bracket $[.,.]_N$ on $\Gamma(E)$, given by
\begin{equation*}
[X, Y]_N:=[NX,Y]+[X,NY]-N[X,Y] =\frac{1}{2} (X \circ^N Y - Y \circ^N X).
\end{equation*}

The next lemma is an axiom included in the original definition of Courant algebroid  \cite{RoytenbergPh.D.}.
\begin{lem}   \label{skewbracketlemma}
Let $ (E , \circ, \rho, \langle .,.\rangle ) $ be a Courant algebroid and $\mathcal D$ its associated operator, given by (\ref{matcallD}). Then,
 \begin{equation*}
 [X,fY]=f[X,Y]+(\rho(X)f)Y-\frac{1}{2}\langle X,Y\rangle \mathcal{D}f,
 \end{equation*}
 for all $X,Y \in \Gamma(E)$ and $f\in \mathcal{C}^{\infty}(M)$, where $[.,.]$ is the bracket given by (\ref{skewbracket}).
\end{lem}

\begin{rem} \label{skewbracketrem}
From the proof of Proposition 2.6.5 in \cite{RoytenbergPh.D.}, we realize that Lemma \ref{skewbracketlemma} also holds in the case of a pre-Courant algebroid.
\end{rem}

  In \cite{RoytenbergWeinstein}, it was proved that to each Courant algebroid corresponds a Lie $2$-algebra. The result in \cite{RoytenbergWeinstein} is established using the graded skew-symmetric version of a Lie $2$-algebra and the definition of Courant algebroid with skew-symmetric bracket. With our  conventions it goes as follows.

  Let $(E,\circ,\rho,\langle .,.\rangle)$ be a Courant algebroid over $M$, with associated operator $ \mathcal{D}$, given by (\ref{matcallD}). Consider the graded vector space
 $V=\mathcal{C}^{\infty}(M)\oplus \Gamma(E)$, where the elements of $\mathcal{C}^{\infty}(M)$ have degree $-2$ and the elements of $\Gamma(E)$ have degree $-1$, and the following
symmetric vector valued forms  $l_1$, $l_2$ and $l_3$ on $V$, defined by:
  \begin{equation}\label{Lie2fromCourant}
  \begin{array}{rcl}
  l_1f &=& \mathcal{D} f\\
  l_2(X,Y)&=& \frac{1}{2} ( X \circ Y - Y \circ X)\\
  l_2(X,f)&=& \frac{1}{2}\langle X,\mathcal{D}f\rangle,\\
  l_3(X,Y,Z)&=& \frac{1}{12}\langle X\circ Y-Y\circ X,Z\rangle +c.p.,
  \end{array}
\end{equation}
for all $X,Y,Z\in \Gamma(E)$ and $f\in \mathcal{C}^{\infty}(M)$,
with $l_1$, $l_2$ and $l_3$ being identically zero in all the other cases. Notice that $l_2|_{\Gamma(E) \times \Gamma(E)}$ coincides with the bracket $[.,.]$ given by (\ref{skewbracket}).

\begin{prop}\label{Roytenbergassociation}
If $(E,\circ,\rho,\langle .,.\rangle )$ is Courant (respectively, pre-Courant) algebroid, then the pair $(V,l_1+l_2+l_3)$, constructed in above, is a symmetric Lie (respectively, pre-Lie\footnote{A pre-Lie $2$-algebra is a pair $(E=E_{-2}\oplus E_{-1}, l_1+l_2+l_3)$, where $E$ is a graded vector space concentrated in degrees $-2$ and $-1$, and $l_1$, $l_2$ and $l_3$ are symmetric graded vector valued $1$-form, $2$-form and $3$-form, respectively, of degree $1$.}) $2$-algebra.
\end{prop}

We call this symmetric  Lie $2$-algebra the symmetric Lie $2$-algebra \emph{associated} to the Courant algebroid $(E,\circ,\rho,\langle .,.\rangle )$.

Starting with a $(1,1)$-tensor on a Courant algebroid with vanishing Nijenhuis torsion we construct, in the next proposition, a Nijenhuis form for the Lie $2$-algebra associated to that Courant structure. First, we need the following lemma.
\begin{lem}\label{leml^N}
Let $ (E , \circ, \rho, \langle .,.\rangle ) $ be a pre-Courant algebroid with the associated symmetric pre-Lie $2$-algebra structure $\mu=l_1+l_2+l_3$, on the graded vector space $V=\mathcal{C}^{\infty}(M)\oplus \Gamma(E)$. Let $N$  be a $(1,1)$-tensor on $E$ such that
 \begin{equation*}
    N+N^*=\lambda \, {\rm Id}_{\Gamma(E)},
\end{equation*}
  with $\lambda$ a Casimir function. Then, the pre-Lie $2$-algebra structure associated to the pre-Courant algebroid $ (E, \circ^N, \rho^N, \langle . ,. \rangle ) $ is $[{\mathcal N},l_1+l_2+l_3]_{_{RN}}$, with ${\mathcal N}$  defined as follows:
  \begin{equation}\label{eq:defmathcalN}
{\mathcal N}|_{\Gamma(E)}= N \,\,\,\,\, \mbox{and} \,\,\,\,\,{\mathcal N}|_{\mathcal{C}^{\infty}(M)}=\lambda \, {\rm Id}_{\mathcal{C}^{\infty}(M)}.
\end{equation}
\end{lem}
\begin{proof}
 Let us denote the pre-Lie $2$-algebra associated to the pre-Courant algebroid $ (E , \circ^N, \rho^N, \langle . ,. \rangle ) $ by $l_1^{N}+l_2^{N}+l_3^{N}$. Using (\ref{deformedD}) and (\ref{Lie2fromCourant}) and taking into account the fact that $\mathcal{D}$ is a derivation, we have, for all $f\in\mathcal{C}^{\infty}(M) $ and for all $X,Y,Z \in \Gamma(E)$,
\begin{equation}\label{l1N}
l_1^{N}f=\mathcal{D}^Nf=\lambda \mathcal{D}f-N\mathcal{D}f=l_1(\mathcal {N}f)-\mathcal{N}l_1(f)= [\mathcal{N},l_1]_{_{RN}}(f),
\end{equation}
\begin{eqnarray}\label{l2N}
l_2^{N}(X,Y)&=&\frac{1}{2}(X\circ^N Y-Y\circ^N X)
            =l_2(NX,Y)+l_2(X,NY)-Nl_2(X,Y) \nonumber \\
            &=&[\mathcal{N},l_2]_{_{RN}}(X,Y),
\end{eqnarray}
\begin{eqnarray}\label{l22N}
l_2^{N}(X,f)&=&\frac{1}{2}\langle X, \mathcal{D}^N f\rangle=\frac{1}{2}\langle X,(- N+\lambda  \, {\rm Id}_{\Gamma(E)})\mathcal{D} f\rangle
            =\frac{1}{2}\langle X,N^* \mathcal{D} f\rangle \nonumber \\ &=&\frac{1}{2}\langle NX, \mathcal{D} f\rangle
            =l_2(NX,f)+\lambda l_2(X,f)-\lambda l_2(X,f) \nonumber \\
           & =&l_2(\mathcal{N}X,f)+l_2(X,\mathcal{N}f)-\mathcal{N}l_2(X,f)\nonumber \\
           & =& [\mathcal{N},l_2]_{_{RN}}(X,f)
\end{eqnarray}
and
\begin{eqnarray} \label{l3N}
& &l_3^{N}(X,Y,Z)=\frac{1}{12} \langle X\circ^N Y-Y\circ^N X, Z \rangle +c.p.(X,Y,Z)  \nonumber \\
&&\hspace{.2cm}=\frac{1}{6}\langle l_2^N(X,Y), Z\rangle+c.p.(X,Y,Z)  \nonumber \\
            &&\hspace{.2cm}=\frac{1}{6}(\langle l_2(NX,Y)+l_2(X,NY)+(N^*-\lambda Id_{\Gamma(E)})l_2(X,Y), Z\rangle)+c.p.(X,Y,Z)   \nonumber \\
            &&\hspace{.2cm}=\frac{1}{6}(\langle l_2(NX,Y),Z\rangle+\langle l_2(X,NY), Z \rangle+\langle l_2(X,Y), NZ\rangle-\lambda\langle l_2(X,Y), Z \rangle)  \nonumber \\
            & &\hspace{.4cm}+c.p.(X,Y,Z)  \nonumber \\
            &&\hspace{.2cm}=\frac{1}{6}(\langle l_2(NX,Y),Z\rangle +c.p.(NX,Y,Z) +\langle l_2(X,NY), Z\rangle+c.p.(X,NY,Z)  \nonumber \\
            & & \hspace{.4cm}+\langle l_2(X,Y),N Z\rangle+c.p.(X,Y,NZ)  - \lambda \langle l_2(X,Y), Z\rangle+c.p.(X,Y,Z))  \nonumber \\
            &&\hspace{.2cm}=l_3(\mathcal N X,Y,Z)+l_3(X,\mathcal N Y,Z)+l_3(X,Y,\mathcal N Z)-\mathcal N l_3(X,Y,Z)  \nonumber \\
            &&\hspace{.2cm}=[\mathcal{N},l_3]_{_{RN}}(X,Y,Z).
\end{eqnarray}
Equations (\ref{l1N}), (\ref{l2N}), (\ref{l22N}) and (\ref{l3N}) complete the proof.
\end{proof}

For the case of a Courant algebroid, we have the following result.

\begin{cor}\label{cornew}
Let $ (E , \circ, \rho, \langle .,.\rangle ) $ be a Courant algebroid with the associated symmetric Lie-$2$  algebra structure $\mu=l_1+l_2+l_3$, on the graded vector space $V=\mathcal{C}^{\infty}(M)\oplus \Gamma(E)$. Let $N$ be a  $(1,1)$-tensor on $E$ such that
 \begin{equation*}\label{con:CourantNijenhuisconditions}
  \begin{cases}
  N+N^*=\lambda \, {\rm Id}_{\Gamma(E)},\\
  (\Gamma(E),\circ^N) \, \, \mbox{\textrm is a Leibniz algebra,}
  \end{cases}
  \end{equation*}
  with $\lambda$ a Casimir function. Then, the Lie $2$-algebra structure associated to the Courant algebroid $ (E, \circ^N, \rho^N, \langle . ,. \rangle ) $ is $[{\mathcal N},l_1+l_2+l_3]$, with ${\mathcal N}$  given by
  (\ref{eq:defmathcalN}).
\end{cor}

\begin{prop}\label{prop:NijenhuisonCourant}
Let $ (E , \circ, \rho, \langle .,.\rangle ) $ be a Courant algebroid with the associated symmetric Lie $2$-algebra structure $\mu=l_1+l_2+l_3$, on the graded vector space $V=\mathcal{C}^{\infty}(M)\oplus \Gamma(E)$. Let $N$  be a  $(1,1)$-tensor on $E$ whose Nijenhuis torsion with respect to the bracket $\circ$ vanishes and
satisfies the following conditions
 \begin{equation*}\label{con:CourantNijenhuisconditions}
  \begin{cases}
  N+N^*=\lambda  \, {\rm Id}_{\Gamma(E)}\\
  N^2+(N^2)^*=\gamma  \, {\rm Id}_{\Gamma(E)},
  \end{cases}
  \end{equation*}
  with $\lambda$ and $\gamma$  Casimir functions. Define   ${\mathcal N}$ and ${\mathcal K}$ as
\begin{equation*}
{\mathcal N}|_{\Gamma(E)}= N \,\,\,\mbox{and}\,\,\,{\mathcal N}|_{\mathcal{C}^{\infty}(M)}=\lambda  \, {\rm Id}_{\mathcal{C}^{\infty}(M)},
\end{equation*}
\begin{equation*}\label{eq:defmathcalK}
{\mathcal K}|_{\Gamma(E)}= N^2 = \lambda N + \frac{\gamma - \lambda^2}{2}  \, {\rm Id}_{\Gamma(E)} \,\,\,\,\mbox{and}\,\,\,\,{\mathcal K}|_{\mathcal{C}^{\infty}(M)}=\gamma \, {\rm Id}_{\mathcal{C}^{\infty}(M)}.
\end{equation*}
Then, $\mathcal N$ is a Nijenhuis vector valued $1$-form with respect to $\mu$, with square $\mathcal K$.
\end{prop}
\begin{proof}
Since the Nijenhuis torsion of $N$ vanishes, $(E, \circ^N)$ and $(E, \circ^{N^2})$ are Leibniz algebras \cite{CGM}, \cite{CJP}. Applying Corollary \ref{cornew} for the Courant algebroid $ (E , \circ, \rho, \langle .,.\rangle ) $, the $(1,1)$-tensor $N$ and the vector valued $1$-form $\mathcal{N}$, twice, we get
\begin{equation}\label{mah}
l_1^{N,N}+l_2^{N,N}+l_3^{N,N}=[\mathcal{N},[\mathcal{N},l_1+l_2+l_3]_{_{RN}}]_{_{RN}},
\end{equation}
where $l_1^{N,N}+l_2^{N,N}+l_3^{N,N}$ stands for the Lie $2$-algebra structure associated to the Courant algebroid $ (E , \circ^{N,N}, \rho^{N,N}, \langle .,.\rangle )$.
Applying again Corollary \ref{cornew} for the Courant algebroid $ (E , \circ, \rho, \langle .,.\rangle ) $, the $(1,1)$-tensor $N^2$ and the vector valued $1$-form $\mathcal{K}$, we get
\begin{equation}\label{khorshid}
l_1^{N^2}+l_2^{N^2}+l_3^{N^2}=[\mathcal{K},l_1+l_2+l_3]_{_{RN}},
\end{equation}
where $l_1^{N^2}+l_2^{N^2}+l_3^{N^2}$ stands for the Lie $2$-algebra structure associated to the Courant algebroid $ (E , \circ^{N^2}, \rho^{N^2}, \langle .,.\rangle )$.
On the other hand, since the Nijenhuis torsion of $N$ vanishes, the Courant algebroids $(E,\circ^{N,N},\rho^{N,N},\langle.,.\rangle)$ and $(E,\circ^{N^2},\rho^{N^2},\langle.,.\rangle)$ coincide. Therefore,  (\ref{mah}) and (\ref{khorshid}) imply that
\begin{equation*}
[\mathcal{N},[\mathcal{N},l_1+l_2+l_3]_{_{RN}}]_{_{RN}}=[\mathcal{K},l_1+l_2+l_3]_{_{RN}}.
\end{equation*}
Finally, an easy computation shows that $[\mathcal{N},\mathcal{K}]_{_{RN}}$ vanishes both on functions and on sections of $E$.
\end{proof}

Since there exists a Lie $2$-algebra associated to each Courant algebroid,
there was a hope that we could, given a Courant structure, find a Nijenhuis deformation
by a Nijenhuis tensor, which is the sum of a vector valued $1$-form and a vector valued $2$-form, of the corresponding Lie $2$-algebra structure,
and prove, eventually, that the Lie $2$-algebra structure obtained by this procedure comes from a Courant structure.
But this fails, at least when the anchor is not identically zero, as it is shown in the next theorem.
First, notice that every  $C^\infty(M)$-linear vector valued form of degree $0$ on $E_{-2}\oplus E_{-1}$, where $E_{-2}:=\mathcal{C}^{\infty}(M)$ and $E_{-1}:=\Gamma(E)$,
is the sum of a $2$-form $\alpha$, a $(1,1)$-tensor $N$ and an endomorphism of $C^\infty(M)$
of the form $f \mapsto \lambda f $ for some smooth function $\lambda$.
Hence, we denote a  $C^\infty(M)$-linear vector valued form of degree zero on $E_{-2}\oplus E_{-1}$ as a sum, $\lambda+ N + \alpha $.

\begin{thm} \label{them:NijenhuisCourant}
Let $(\circ, \rho,\langle.,.\rangle)$  be a Courant structure on a vector bundle $E\to M$ with the associated Lie $2$-algebra structure $l_1+l_2+l_3$ on the graded vector space $V=E_{-2}\oplus E_{-1}$, where $E_{-2}:=\mathcal{C}^{\infty}(M)$ and $E_{-1}:=\Gamma(E)$.
Let ${\mathcal N} = \lambda + N + \alpha$ be a $C^\infty(M)$-linear vector valued form of degree zero on $V$.
Assume also that $\rho$ is not equal to zero on a dense subset of the base manifold.
If $ [{\mathcal N}, l_1+l_2+l_3]_{_{RN}}$ is the Lie $2$-algebra associated to a Courant structure with the same scalar product $\langle.,.\rangle$, then
\begin{enumerate}
\item $\lambda$ is a Casimir,
\item $\alpha=0,$
\item $ N+N^* = \lambda Id_{\Gamma(E)}.$
\end{enumerate}
In this case, the Courant structure that $ [{\mathcal N}, l_1+l_2+l_3]_{_{RN}}$ is associated to, is $(\circ^N, \rho^N,\langle . ,. \rangle).$
\end{thm}
\begin{proof}
Set $\mu=l_1+l_2+l_3$ and denote the $i$-form component of $[\mathcal{N},\mu]_{_{RN}}$ by $[\mathcal{N},\mu]^i_{RN}$, $i=1,2$. Then, for all $X,Y \in \Gamma(E)$ and $f\in \mathcal{C}^{\infty}(M)$, we have
\begin{equation*}\label{equation1intheorem}
\begin{array}{rcl}
[\mathcal{N},\mu]^1_{RN}(f)&=&([\lambda,l_1]_{_{RN}}+[N,l_1]_{_{RN}})(f)\\
                           &=&l_1(\lambda f)-Nl_1(f)\\
                           &=&\lambda l_1( f)+fl_1(\lambda)-Nl_1(f).
\end{array}
\end{equation*}
   The first equation in (\ref{Lie2fromCourant}) implies that, if $[\mathcal{N},\mu]_{_{RN}}$ is a Lie 2-algebra associated to a Courant algebroid, then $[\mathcal{N},\mu]^1_{RN}$ has to be a derivation, and this happens if and only if $l_1(\lambda)=0$. So, we get that $\lambda$ is a Casimir and
\begin{equation}\label{equation1intheorem-1}
[\mathcal{N},\mu]^1_{RN}(f)=(\lambda Id_{\Gamma(E)}-N) l_1( f).
\end{equation}
On the other hand,
\begin{eqnarray}\label{equation2intheorem}
[\mathcal{N},\mu]_{_{RN}}^2(X,f)&=&([\lambda,l_2]_{_{RN}}+[N,l_2]_{_{RN}}+[\alpha,l_1]_{_{RN}})(X,f) \nonumber\\
                           &=&l_2(X,\lambda f)-\lambda l_2(X,f)+l_2(NX,f)-\alpha(X,l_1(f)) \nonumber \\
                           &=&\frac{1}{2}\lambda\langle X,l_1(f)\rangle-\frac{1}{2}\lambda\langle X,l_1(f)\rangle+\frac{1}{2}\langle NX,l_1(f)\rangle-\alpha(X,l_1(f)) \nonumber \\
                           &=&\frac{1}{2}\langle NX,l_1(f)\rangle-\alpha(X,l_1(f)),
\end{eqnarray}
and the same computations for $(f,X)$ instead of $(X,f)$ gives
\begin{equation}\label{eq:174A}
[\mathcal{N},\mu]_{_{RN}}^2(f,X)=\frac{1}{2}\langle NX,l_1(f)\rangle-\alpha(l_1(f),X).
\end{equation}
Since $[\mathcal{N},\mu]_{_{RN}}^2(X,f)=[\mathcal{N},\mu]_{_{RN}}^2(f,X)$, from (\ref{equation2intheorem}) and (\ref{eq:174A}) we get $\alpha(X,l_1(f))=0$, for all $X\in \Gamma(E)$ and $f\in \mathcal{C}^{\infty}(M)$; so,
\begin{equation}\label{eq:deformedon(X,f)}
[\mathcal{N},\mu]_{_{RN}}^2(X,f)=\frac{1}{2}\langle NX,l_1(f)\rangle.
\end{equation}
For any $X,Y\in \Gamma(E)$, we have
\begin{equation}\label{equation2intheorem-2}
\begin{array}{rcl}
[\mathcal{N},\mu]_{_{RN}}^2(X,Y)&=&([\lambda,l_2]_{_{RN}}+[N,l_2]_{_{RN}}+[\alpha,l_1]_{_{RN}})(X,Y)\\
                           &=&l_2(NX,Y)+l_2(X,NY)-Nl_2(X,Y)+l_1\alpha(X,Y).\\

\end{array}
\end{equation}
According to Lemma \ref{skewbracketlemma}, if $[\mathcal{N},\mu]_{_{RN}}$ is a Lie $2$-algebra associated to a Courant structure, then we must have:
\begin{equation}\label{mainbody}
[\mathcal{N},\mu]_{_{RN}}^2(X,fY)=f[\mathcal{N},\mu]_{_{RN}}^2(X,Y)+2[\mathcal{N},\mu]_{_{RN}}^2(X,f).Y-\frac{1}{2}\langle X,Y\rangle[\mathcal{N},\mu]^1_{RN}(f).
\end{equation}
Using (\ref{equation1intheorem-1}), (\ref{eq:deformedon(X,f)}) and (\ref{equation2intheorem-2}), we get
\begin{eqnarray}\label{eq:LHS}
                              &&[\mathcal{N},\mu]_{_{RN}}^2(X,fY)=l_2(NX,fY)+l_2(X,NfY)-Nl_2(X,fY)+l_1\alpha(X,fY) \nonumber\\
                              &&\hspace{.2cm}=fl_2(NX,Y)+2l_2(NX,f)Y-\frac{1}{2}\langle NX,Y\rangle l_1(f) \nonumber \\
                              &&\hspace{.4cm}+fl_2(X,NY)+2l_2(X,f)NY-\frac{1}{2}\langle X,NY\rangle l_1(f) \nonumber \\
                              &&\hspace{.4cm}-fNl_2(X,NY)-2l_2(X,f)NY+\frac{1}{2}\langle X,Y\rangle Nl_1(f) \nonumber \\
                              &&\hspace{.4cm}+fl_1\alpha(X,Y)+\alpha(X,Y)l_1(f) \nonumber \\
                              &&\hspace{.2cm}=f(l_2(NX,Y)+l_2(X,NY)-Nl_2(X,Y)+l_1\alpha(X,Y))+2l_2(NX,f)Y \nonumber \\
                              &&\hspace{.4cm}-\frac{1}{2}\langle X,(N+N^*)Y\rangle l_1(f)+\frac{1}{2}\langle X,Y\rangle Nl_1(f)+\alpha(X,Y)l_1(f)
\end{eqnarray}
and
\begin{eqnarray}\label{eq:RHS}
&&f[\mathcal{N},\mu]_{_{RN}}^2(X,Y)+2[\mathcal{N},\mu]_{_{RN}}^2(X,f).Y-\frac{1}{2}\langle X,Y\rangle[\mathcal{N},\mu]^1_{RN}(f) \nonumber \\
&& \hspace{.2cm} =f(l_2(NX,Y)+l_2(X,NY)-Nl_2(X,NY)+l_1\alpha(X,Y))+2l_2(NX,f).Y \nonumber\\
&& \hspace{.2cm}-\frac{1}{2}\langle X,Y\rangle(\lambda  \, {\rm Id}_{\Gamma(E)}-N)l_1(f).
\end{eqnarray}
Now, Equations (\ref{mainbody}), (\ref{eq:LHS}) and (\ref{eq:RHS}) show that
\begin{equation*}\label{akhari}
\frac{1}{2}\langle X,(N+N^*-\lambda  \, {\rm Id}_{\Gamma(E)})Y\rangle l_1(f)=\alpha(X,Y)l_1(f),
\end{equation*}
for all $X,Y\in \Gamma(E)$ and $f\in \mathcal{C}^{\infty}(M)$. Since $\alpha$ is skew-symmetric, $\langle . ,(N+N^*-\lambda \, {\rm Id}).\rangle$ is symmetric on $\Gamma(E) \times \Gamma(E)$ and the anchor is not zero everywhere, which implies that $l_1(f)$ is not always zero, we have $\alpha=0$ and $N+N^*-\lambda \, {\rm Id}_{\Gamma(E)}=0$.
\end{proof}

\begin{cor}\label{thewrongcorollary}
Let $(E,\circ, \rho,\langle.,.\rangle)$ be a Courant algebroid  with anchor $\rho$ being different from zero on a dense subset of $E$ and let $\mu$ be the associated Lie $2$-algebra structure on the graded vector space $\mathcal{C}^{\infty}(M)\oplus \Gamma(E)$. Then, there is a one to one correspondence between:
 \begin{enumerate}
 \item[(i)] quadruples $(N,K,\lambda,\gamma)$ with $N,K$ being $(1,1)$-tensors on $E$ and $\lambda,\gamma$ being Casimir functions satisfying the following conditions:
    $$ \begin{cases}
     \circ^{N,N}=\circ^K,\\
     N K - K N =0,\\
     N+N^*=\lambda   \, {\rm Id}_{\Gamma(E)} ,\\
     K+K^* = \gamma   \, {\rm Id}_{\Gamma(E)},\\
     (\Gamma(E), \circ^N)\,\,\, \mbox{and}\,\,\, (\Gamma(E), \circ^K) \,\,\,\mbox{are Leibniz algebras}.
     \end{cases}$$
 \item[(ii)] Nijenhuis vector valued forms ${\mathcal N} $ with respect to $\mu$, with square ${\mathcal K}$, such that the deformed brackets $[ {\mathcal N}, \mu]_{_{RN}} $ and $[ {\mathcal K}, \mu]_{_{RN}} $ are Lie $2$-algebras associated to  Courant structures with the same scalar product.
 \end{enumerate}
\end{cor}
\begin{proof}
Given a quadruple $(N,K,\lambda,\gamma)$ satisfying  conditions in item $(i)$,  we define vector valued $1$-forms $\mathcal{N}$ and $\mathcal{K}$ on the graded vector space $\mathcal{C}^{\infty}(M)\oplus \Gamma(E)$ as $\mathcal{N}(f)=\lambda f$, $\mathcal{K}(f)=\gamma f$,  $\mathcal{N}(X)=NX$ and $\mathcal{K}(X)=KX$, for all $X\in \Gamma(E)$ and $f\in \mathcal{C}^{\infty}(M)$. We prove that $\mathcal{N}$ is a Nijenhuis vector valued form with respect to $\mu$, with square $\mathcal{K}$. First, notice that using Corollary \ref{cor:Anchoruniqe}, the assumption $\circ^{N,N}=\circ^{K}$ implies that $(E,\circ^{N,N}, \rho^{N,N}, \langle .,.\rangle)$ and $(E,\circ^{K}, \rho^{K}, \langle .,.\rangle)$ are the same pre-Courant algebroid, hence, they have the same associated pre-Lie $2$-algebras. On the other hand, using Lemma \ref{leml^N}, the pre-Lie $2$-algebra associated to the pre-Courant algebroid $(E,\circ^{N,N}, \rho^{N,N}, \langle .,.\rangle)$ is
$[\mathcal{N},[\mathcal{N},\mu]_{_{RN}}]_{_{RN}}$ and the pre-Lie $2$-algebra associated to the pre-Courant algebroid $(E,\circ^{K}, \rho^{K}, \langle .,.\rangle)$ is $[\mathcal{K},\mu]_{_{RN}}$. Hence,
\begin{equation}\label{injast}
[\mathcal{N},[\mathcal{N},\mu]_{_{RN}}]_{_{RN}}=[\mathcal{K},\mu]_{_{RN}}.
\end{equation}
Also, using the assumption $NK-KN=0$, we get
\begin{equation}\label{keiman}
[\mathcal{N},\mathcal{K}]_{_{RN}}=0.
\end{equation}
Equations (\ref{injast}) and (\ref{keiman}) show that $\mathcal{N}$ is a Nijenhuis vector valued $1$-form with respect to $\mu$, with square $\mathcal{K}$. By Corollary \ref{cornew}, $[\mathcal{N},\mu]_{_{RN}}$ is a Lie $2$-algebra associated to the Courant algebroid $(E,\circ^N, \rho, \langle .,.\rangle)$ and $[\mathcal{K},\mu]_{_{RN}}$ is a Lie $2$-algebra associated to the Courant algebroid $(E,\circ^K, \rho, \langle .,.\rangle)$.

Conversely, assume that $\mathcal{N}$ is a Nijenhuis vector valued form with respect to $\mu$, with square $\mathcal{K}$, such that $[\mathcal{N},\mu]_{_{RN}}$ and $[\mathcal{K},\mu]_{_{RN}}$ are Lie $2$-algebras associated to  Courant algebroids. Then, by Theorem \ref{them:NijenhuisCourant}, $\mathcal{N}$ is of the form $\lambda+N$ with $N+N^*=\lambda   \, {\rm Id}_{\Gamma(E)}$ and $\mathcal{K}$ is of the form $\gamma+K$, with $K+K^*=\gamma   \, {\rm Id}_{\Gamma(E)}$.
 Moreover, the Courant algebroid which is associated to the Lie $2$-algebra $[\mathcal{N}, \mu]_{_{RN}}$ (respectively, $[\mathcal{K}, \mu]_{_{RN}}$) is $(E,\circ^N,\rho^N,\langle .,.\rangle)$ (respectively, $(E,\circ^K,\rho^K,\langle .,.\rangle)$ ). From this, we get that  $(\Gamma(E), \circ^N)$ and $(\Gamma(E), \circ^K)$ are Leibniz algebras. Since $\mathcal{N}$ is a Nijenhuis vector valued form with respect to $\mu$, with square $\mathcal{K}$, we have
\begin{equation}\label{famus}
[\mathcal{N},[\mathcal{N},\mu]_{_{RN}}]_{_{RN}}=[\mathcal{K},\mu]_{_{RN}}
\end{equation}
and
\begin{equation}\label{famusmus}
[\mathcal{N},\mathcal{K}]_{_{RN}}=0.
\end{equation}
Applying both sides of Equation (\ref{famus}) on a pair of sections $X,Y\in \Gamma(E)$ we get $X\circ^{N,N}Y=X\circ^K Y$, which implies $\circ^{N,N}=\circ^K$. Lastly, Equation (\ref{famusmus}) implies $KN-NK=0$.
\end{proof}

Using Lemma \ref{skewbracketlemma} and Remark \ref{skewbracketrem}, and also taking into account the fact that the operator $\mathcal{D}$ associated to a pre-Courant algebroid $(E,\circ,\rho,\langle.,.\rangle)$, given by (\ref{matcallD}),  is a derivation, we may
 restate Theorem \ref{them:NijenhuisCourant}.

\begin{thm} \label{them:NijenhuisCourant2}
Let $ (\circ,\rho, \langle . ,. \rangle ) $ be a Courant structure on a vector bundle $E\to M$, with the associated symmetric Lie $2$-algebra structure $l_1+l_2+l_3$ on the graded vector space $V=E_{-2}\oplus E_{-1}$, where $E_{-2}:=\mathcal{C}^{\infty}(M)$ and $E_{-1}:=\Gamma(E)$.
Let ${\mathcal N} = \lambda + N + \alpha$ be a $C^\infty(M)$-linear vector valued form of degree zero on $V$.
Assume also that $\rho$ is not equal to zero on a dense subset of the base manifold.
If $ [{\mathcal N}, l_1+l_2+l_3]_{_{RN}}=l'_1+l'_2+l'_3 $, where the vector valued forms $l'_1,l'_2,l'_3$ are obtained from a pre-Courant algebroid, by the construction given in (\ref{Lie2fromCourant}), with the same scalar product, then
\begin{enumerate}
\item $\lambda$ is a Casimir,
\item $\alpha=0,$
\item $ N+N^* = \lambda   \, {\rm Id}_{\Gamma(E)}.$
\end{enumerate}
In this case, the Courant structure that $ [{\mathcal N}, l_1+l_2+l_3]_{_{RN}}$ is associated to, is $(\circ^N, \rho^N,\langle . ,. \rangle).$
\end{thm}
And this leads to the next result:
\begin{cor}\label{thewrongcorollary2}
Let $(E,\circ, \rho,\langle.,.\rangle)$ be a Courant algebroid  with anchor $\rho$ being different from zero on a dense subset of $E$, with the associated Lie $2$-algebra structure $\mu=l_1+l_2+l_3$ on the graded vector space $\mathcal{C}^{\infty}(M)\oplus \Gamma(E)$. Then, there is a one to one correspondence between:
 \begin{enumerate}
 \item[(i)] quadruples $(N,K,\lambda,\gamma)$ with $N,K$ being $(1,1)$-tensors and $\lambda,\gamma$ being Casimir functions satisfying the following conditions:
    \begin{equation}\label{yekikamtar}
     \begin{cases}
     \circ^{N,N}=\circ^K,\\
     N K - K N =0,\\
     N+N^*=\lambda   \, {\rm Id}_{\Gamma(E)} ,\\
     K+K^* = \gamma   \, {\rm Id}_{\Gamma(E)}.\\
     \end{cases}
     \end{equation}
     \item[(ii)] $C^\infty(M)$-linear Nijenhuis vector valued forms ${\mathcal N} $ with respect to $\mu$, with square ${\mathcal K} $, such that the deformed bracket is of the form $[ {\mathcal N}, \mu]_{_{RN}}=l'_1+l'_2+l'_3 $ and $l'_1,l'_2,l'_3$ are constructed by the procedure in (\ref{Lie2fromCourant}) obtained from a pre-Courant algebroid, with the same scalar product.
 \end{enumerate}
 \end{cor}
\begin{proof}
Let $\mathcal{N}$ be a $C^\infty(M)$-linear Nijenhuis vector valued form with respect to the Lie $2$-algebra structure $\mu=l_1+l_2+l_3$, with square $\mathcal{K}$, and assume that $\left[\mathcal{N},\mu\right]_{_{RN}}$ is obtained from a pre-Courant algebroid. Let
\begin{equation*}
\mathcal{N}|_{\Gamma(E)}=N,\quad \mathcal{N}|_{\mathcal{C}^{\infty}(M)}=\lambda   \, {\rm Id}_{\mathcal{C}^{\infty}(M)},\quad \mathcal{K}|_{\Gamma(E)}=K \quad\mbox{and}\quad\mathcal{K}|_{\mathcal{C}^{\infty}(M)}=\gamma   \, {\rm Id}_{\mathcal{C}^{\infty}(M)}.
\end{equation*}
By Theorem \ref{them:NijenhuisCourant2}, $N+N^*=\lambda   \, {\rm Id}_{\Gamma(E)}$ and $(E,\circ^N, \rho^N,\langle.,.\rangle)$ is a pre-Courant algebroid (it is, in fact, the pre-Courant algebroid which $\left[\mathcal{N},\mu\right]_{_{RN}}$ is obtained from). Hence, by Lemma \ref{pre}, $(E,\circ^{N,N}, \rho^{N,N},\langle.,.\rangle)$ is a pre-Courant algebroid. Now, Lemma \ref{leml^N} implies that $\left[\mathcal{K},\mu\right]_{_{RN}}=\left[\mathcal{N},\left[\mathcal{N},\mu\right]_{_{RN}}\right]_{_{RN}}$ is obtained from the pre-Courant algebroid $(E,\circ^{N,N}, \rho^{N,N},\langle.,.\rangle)$, by the construction given in (\ref{Lie2fromCourant}). Therefore, by Theorem \ref{them:NijenhuisCourant2}, $K+K^*=\gamma   \, {\rm Id}_{\Gamma(E)}$. The assumption $\left[\mathcal{N},\mathcal{K}\right]_{_{RN}}=0$ implies that $NK-KN=0$, while $\left[\mathcal{N},\left[\mathcal{N},\mu\right]_{_{RN}}\right]_{_{RN}}=\left[\mathcal{K},\mu\right]_{_{RN}}$ implies that $\circ^{N,N}=\circ^{K}$.

Conversely, assume that we are given a quadruple $(N,K,\lambda,\gamma)$ satisfying the properties in (\ref{yekikamtar}). By Lemma \ref{pre}, $(E,\circ^{N},\rho^{N},\langle .,.\rangle)$ is a pre-Courant and by Lemma \ref{leml^N}, the pre-Lie $2$-algebra structure associated to the pre-Courant algebroid $(E,\circ^{N},\rho^{N},\langle .,.\rangle)$ is $\left[\mathcal{N},\mu\right]_{_{RN}}$. Similar arguments prove that the pre-Lie $2$-algebra structure associated to the pre-Courant algebroid $(E,\circ^{N,N},\rho^{N,N},\langle .,.\rangle)$ is $\left[\mathcal{N},\left[\mathcal{N},\mu\right]_{_{RN}}\right]_{_{RN}}$ and the pre-Lie $2$-algebra structure associated to the pre-Courant algebroid $(E,\circ^{K},\rho^{K},\langle .,.\rangle)$ is $\left[\mathcal{K},\mu\right]_{_{RN}}$. Now, the assumption $\circ^{N,N}=\circ^{K}$ and Lemma \ref{cor:Anchoruniqe} imply that $(E,\circ^{N,N},\rho^{N,N},\langle .,.\rangle)$ and  $(E,\circ^{K},\rho^{K},\langle .,.\rangle)$ are the same pre-Courant algebroid; therefore, we have $\left[\mathcal{N},\left[\mathcal{N},\mu\right]_{_{RN}}\right]_{_{RN}}=\left[\mathcal{K},\mu\right]_{_{RN}}$. It follows from the assumption $NK-KN=0$ that $\left[\mathcal{N},\mathcal{K}\right]=0$. Hence, $\mathcal{N}$ is a $C^\infty(M)$-linear Nijenhuis vector valued form with respect to the Lie $2$-algebra structure $\mu$, with square $\mathcal{K}$.
\end{proof}

\begin{rem}
The notion of weak Nijenhuis tensor on a Courant algebroid was introduced in \cite{CGM} (see also \cite{YKSBrazil}). A $(1,1)$-tensor $N$ on a Courant algebroid $(E,\circ, \rho,\langle.,.\rangle)$ is a \emph{weak Nijenhuis tensor} if its  Nijenhuis torsion is a Leibniz $2$-cocycle.
We may ask how weak Nijenhuis tensors on a Courant algebroid are related to weak Nijenhuis vector valued forms, with respect to the Lie $2$-algebra associate to the Courant algebroid.

Let $(E,\circ, \rho,\langle.,.\rangle)$ be a Courant algebroid, with associated Lie $2$-algebra $\mu$,
 and $N$ a $(1,1)$-tensor on $E$ which is weak Nijenhuis (in the sense of \cite{CGM}) and such that $N+N^*=\lambda   \, {\rm Id}_{\Gamma(E)}$, with $\lambda$ a Casimir function.
 Then, $(E,\circ^N, \rho^N,\langle.,.\rangle)$ is a Courant algebroid \cite{CGM}; let us denote by  $\mu^N$ its associated Lie $2$-algebra. By Corollary \ref{cornew}, $\mu^N= \left[\mathcal{N},\mu\right]_{_{RN}}$, with $\mathcal{N}$ given by (\ref{eq:defmathcalN}). But $\left[\mathcal{N},\mu\right]_{_{RN}}$ being a Lie $2$-algebra is equivalent to $\mathcal{N}$ being weak Nijenhuis with respect to $\mu$ (see Proposition \ref{prop:deformedisLinfty}).

 Summarizing, if $N$ is a weak Nijenhuis tensor on a Courant algebroid $E$ and $N+N^*=\lambda   \, {\rm Id}_{\Gamma(E)}$, with $\lambda$ a Casimir function, then, $\mathcal{N}$ given by
 \begin{equation*}
{\mathcal N}|_{\Gamma(E)}= N \,\,\,\,\, \mbox{and} \,\,\,\,\,{\mathcal N}|_{\mathcal{C}^{\infty}(M)}=\lambda \, {\rm Id}_{\mathcal{C}^{\infty}(M)}
\end{equation*}
 is a weak Nijenhuis vector valued form with respect to the Lie $2$-algebra associated to the Courant algebroid.
\end{rem}
\

%%%%%%%%%%%%%%%%%%%%%%%%%%%%%%%%%%%%%%%%%%%%%%%%%%%%%%%%%%%%%%%%%%%%%%%%%%%%%%%%%%%%%%%%%%%%%%%%%%%%%%%%%%%%%%%%%%%%%%%%%%%
%%%%%%%%%%%%%%%%%%%%%%%%%%%%%%%%%%%%%%%%%%%%%%%%%%%%%%%%%%%%%%%%%%%%%%%%%%%%%%%%%%%%%%%%%%%%%%%%%%%%%%%%%%%%%%%%%%%%%%%%%%%%%%%%%%%%%%%%%%

\section{Multiplicative $L_{\infty}$-structures}
Adapting the notion of $P_\infty$-structure on a graded vector space \cite{Cattaneo-Felder} to the symmetric graded case, we define, in this section,  multiplicative $L_{\infty}$-structures. We classify all multiplicative $L_{\infty}$-structures on $\Gamma(\wedge A)[2]$,
for $A \to M $ an arbitrary vector bundle over a manifold $M$. When $A \to M $ is equipped with a Lie algebroid structure, given a $(1,1)$-tensor $N$ on $A$, we define the extension of $N$ by derivation, which is a symmetric vector valued $1$-form on $\Gamma(\wedge A)[2]$, of degree zero. For a $k$-form on the Lie algebroid, we also define its extension by derivation, yielding a symmetric vector valued form $k$-form of degree $k-2$. These multi-derivations will be used in the next section to construct examples of Nijenhuis forms.

There is an important graded Lie subalgebra of $(\tilde{S}(E^*) \otimes E,\,\left[.,.\right]_{{RN}})$, when there exists a graded
commutative associative algebra structure on $E[2]$, denoted by $\wedge$, that is, a bilinear operation such that for all $X \in E_i, Y \in E_j, Z\in E_k$
\begin{itemize}
 \item $  X\wedge Y \in E_{i+j+2}$,
 \item $(X\wedge Y) \wedge Z = X\wedge (Y \wedge Z)$,
 \item $ X \wedge Y = (-1)^{|X||Y|} Y\wedge X$,
\end{itemize}
 where $|X|=i+2$ and $|Y|=j+2$.

\begin{defn}  \label{defn:multiderivation}
Let $E$ be a graded vector space equipped with an associative graded commutative algebra structure, that is a graded symmetric bilinear map $\wedge$ of degree zero which is associative. An element $D\in S^{d}(E^*)\otimes E$ is called a {\em multi-derivation vector valued $d$-form}, if
\begin{eqnarray}\label{eqdef:multi-derivation}
&&D(X_1, \cdots,X_{i-1}, Y\wedge Z, X_{i+1},\cdots,X_d)\\
&& \hspace{.2cm}=(-1)^{|Z|(|X_{i+1}|+\cdots+|X_{d}|)}D(X_1, \cdots,X_{i-1}, Y, X_{i+1},\cdots,X_d)\wedge Z \nonumber \\
&&\hspace{.2cm} +(-1)^{|Y|(|X_{1}|+\cdots+|X_{i-1}|+\bar{D})}Y\wedge D(X_1, \cdots,X_{i-1}, Z, X_{i+1},\cdots,X_d), \nonumber
\end{eqnarray}
for all $X_1, \cdots,X_d,Y,Z \in E$, where $\bar{D}$ is the degree of $D$ as a graded map.
\end{defn}
\begin{rem}  \label{remark_multider}
The graded commutativity of the product $\wedge$ implies that Equation (\ref{eqdef:multi-derivation}) is equivalent to
\begin{eqnarray*}
&&D(X_1,\cdots,X_{d-1},Y\wedge Z)\\
&& \hspace{.2cm}=D(X_1, \cdots,X_{d-1},Y)\wedge Z +(-1)^{|Y||Z|} D(X_1,\cdots,X_{d-1},Z)\wedge Y.
\end{eqnarray*}
\end{rem}

 We denote the space of all multi-derivation vector valued forms by $MultiDer(E)$. Elements of $S^{1}(E^*)\otimes E$ are simply called \emph{derivations}. By definition, $E \subset MultiDer(E)$ and we have the following:
\begin{prop}\label{prop:Multider}
    $MultiDer(E)$ is a graded Lie subalgebra of  $(\tilde{S}(E^*) \otimes E,\, \left[.,.\right]_{_{RN}})$.
 \end{prop}
 We shall use the following lemmas in the proof of Proposition \ref{prop:Multider}.
 \begin{lem}\label{lem:multiDer1}
 Let $D_1$ and $D_2$ be two derivations. Then, $[D_1,D_2]_{_{RN}}$ is also a derivation.
 \end{lem}
 \begin{proof}
 We have
 \begin{equation*}
 \begin{array}{rcl}
 [D_1,D_2]_{_{RN}}&=& D_2\circ D_1-(-1)^{\bar{D_1}\bar{D_2}}D_1\circ D_2\\
                  &=&-(-1)^{\bar{D_1}\bar{D_2}}[D_1,D_2],
 \end{array}
 \end{equation*}
 where $\left[.,.\right]$ is the graded commutator on the space of derivations of the graded associative commutative algebra $(E,\, \wedge)$.  Since $[D_1,D_2]$ is a derivation, so is $[D_1,D_2]_{_{RN}}$.
\end{proof}

 \begin{lem}\label{lem:multiDer2}
 For $d \geq 2$, an element $D\in S^{d}(E^*)\otimes E $ is a multi-derivation vector valued $d$-form if and only if $[X,D]_{_{RN}}$ is a multi-derivation vector valued $(d-1)$-form, for all $X\in E$.
 \end{lem}
 \begin{proof}
 It is a direct consequence of
 \begin{equation*}
 [X,D]_{_{RN}}(X_1,\cdots,X_{d-2}, Y\wedge Z)=D(X,X_1,\cdots,X_{d-2}, Y\wedge Z),
 \end{equation*}
 which holds for all elements $Y,Z,X_1,\cdots,X_{d-2} \in E$.
 \end{proof}
 \begin{proof}(of Proposition \ref{prop:Multider})
 Let $D, D^{\prime}$ be two multi-derivation vector valued $d$-form and $d'$-form, respectively. We show that $[D,D^{\prime}]_{_{RN}}$ is a multi-derivation vector valued $(d+d'-1)$-form, using induction on the number $n=d+d'-1$. Lemmas \ref{lem:multiDer1} and \ref{lem:multiDer2} prove the case $n=1$. Assume, by induction, that $[D,D^{\prime}]_{_{RN}}$ is a multi-derivation vector valued $(d+d'-1)$-form and let $D_1$ and $D_2$ be two multi-derivation vector valued $d_1$- and $d_2$-forms respectively, such that $d_1+d_2-1=n+1$. From (\ref{kform}) we have
 \begin{equation*}\label{eq:firstinprop}
 [D_1,D_2]_{_{RN}}(X_1,\cdots,X_{d_1+d_2-2},Y\wedge Z)=[Y\wedge Z,[X_{d_1+d_2-2},\cdots,[X_1,[D_1,D_2]_{_{RN}}]_{_{RN}}\cdots]_{_{RN}}]_{_{RN}},
 \end{equation*}
 or, using the graded Jacobi identity of $[.,.]_{_{RN}}$,
 \begin{eqnarray*}
 &&[D_1,D_2]_{_{RN}}(X_1,\cdots,X_{d_1+d_2-2},Y\wedge Z)\\
 &&\hspace{.2cm} =[Y\wedge Z,[X_{d_1+d_2-2},\cdots,[[X_1,D_1]_{_{RN}},D_2]_{_{RN}}\cdots]_{_{RN}}]_{_{RN}}\\
 &&\hspace{.3cm} +(-1)^{\bar{D_1}\bar{X_1}}[Y\wedge Z,[X_{d_1+d_2-2},\cdots,[D_1,[X_1,D_2]_{_{RN}}]_{_{RN}}\cdots]_{_{RN}}]_{_{RN}},
 \end{eqnarray*}
 for all $X_1,\cdots,X_{d_1+d_2-2},Y, Z\in E$.
 By Lemma \ref{lem:multiDer2}, $[X_1,D_1]_{_{RN}}$ and $[X_1,D_2]_{_{RN}}$ are multi-derivation vector valued $(d_1-1)$- and $(d_2-1)$-forms respectively, and hence using the assumption of induction, $[[X_1,D_1]_{_{RN}},D_2]_{_{RN}}$ and $[D_1,[X_1,D_2]_{_{RN}}]_{_{RN}}$ are multi-derivation vector valued $n$-forms. Therefore,
 \begin{eqnarray*}
 &&[D_1,D_2]_{_{RN}}(X_1,\cdots,X_{d_1+d_2-2},Y\wedge Z)\\
 &&\hspace{.2cm} =[[X_1,D_1]_{_{RN}},D_2]_{_{RN}}(X_2,\cdots,X_{d_1+d_2-2},Y\wedge Z)\\
 &&\hspace{.3cm} +(-1)^{\bar{D_1}\bar{X_1}}[D_1,[X_1,D_2]_{_{RN}}]_{_{RN}}(X_2,\cdots,X_{d_1+d_2-2},Y\wedge Z)\\
 &&\hspace{.2cm} =[[X_1,D_1]_{_{RN}},D_2]_{_{RN}}(X_2,\cdots,X_{d_1+d_2-2},Y)\wedge Z\\
 &&\hspace{.3cm} +(-1)^{|Y||Z|}[[X_1,D_1]_{_{RN}},D_2]_{_{RN}}(X_2,\cdots,X_{d_1+d_2-2},Z)\wedge Y\\
 &&\hspace{.3cm} +(-1)^{\bar{D_1}\bar{X_1}}[D_1,[X_1,D_2]_{_{RN}}]_{_{RN}}(X_2,\cdots,X_{d_1+d_2-2},Y)\wedge Z\\
 &&\hspace{.3cm} +(-1)^{\bar{D_1}\bar{X_1}}(-1)^{|Y||Z|}[D_1,[X_1,D_2]_{_{RN}}]_{_{RN}}(X_2,\cdots,X_{d_1+d_2-2},Z)\wedge Y\\
 &&\hspace{.2cm} =[D_1,D_2]_{_{RN}}(X_1,\cdots,X_{d_1+d_2-2},Y)\wedge Z\\
 &&\hspace{.3cm} +(-1)^{|Y||Z|}[D_1,D_2]_{_{RN}}(X_1,\cdots,X_{d_1+d_2-2},Z)\wedge Y,
 \end{eqnarray*}
which completes the induction and also the proof (see Remark \ref{remark_multider}).
 \end{proof}

\begin{rem}
The graded symmetric bilinear map $\wedge$ of degree zero on $E[2]$, considered in Definition~\ref{defn:multiderivation}, can be viewed as a vector valued $2$-form of degree $2$ on $E$, and so we may compute $[\wedge, D]_{_{RN}}$ for any vector valued $d$-form $D$. An easy computation shows that
\begin{equation} \label{1_derivation}
\textrm{a vector valued} \,\,1\textrm{-form} \,\, D \,\, \textrm{is a derivation if and only if} \,\, [\wedge, D]_{_{RN}}=0.
\end{equation}
Now, if $D$ (resp. $D'$) is a multi-derivation vector valued $d$-form (resp. $d'$-form) and $X_1,\cdots X_{d+d'-1}$ are elements of $E$ then, by Lemma~\ref{lem:multiDer2},
${\mathcal D}:= [X_{d-1}, \cdots , [X_2, [X_1, D]_{_{RN}} ]_{_{RN}} \cdots ]_{_{RN}}$ (resp. ${\mathcal D'}:= [X_{d'-1}, \cdots , [X_2, [X_1, D']_{_{RN}} ]_{_{RN}} \cdots ]_{_{RN}}$) is a derivation vector-valued $1$-form. Hence, by (\ref{1_derivation}), we get
$
[\wedge, {\mathcal D}]_{_{RN}}=[\wedge, {\mathcal D'}]_{_{RN}}=0
$
which implies $[\wedge, {\mathcal D}_1]_{_{RN}}=0$, with ${\mathcal D}_1:= [X_{d+d'-2}, \cdots , [X_2, [X_1, [D,D']_{_{RN}}]_{_{RN}} ]_{_{RN}} \cdots ]_{_{RN}}$
(notice that we made use here of the Jacobi identity for the Richardson-Nijenhuis bracket). So, ${\mathcal D}_1$ is a derivation vector-valued $1$-form. Thus, by Lemma~\ref{lem:multiDer2}, $[D,D']_{_{RN}}$ is multi-derivation vector valued $(d+d'-1)$-form.

This observation gives an alternative proof of Proposition~\ref{prop:Multider}.
\end{rem}
%%%%%%%%%%%%%%%%%%%%%%%%%%%%%%%%%%%%%%%%%%%   Examples around Lie algebroids           %%%%%%%%%%%%%%%%%%%%%%%%%%%%%%%%%

Let us now define multiplicative $L_{\infty}$-algebra.
\begin{defn}
An $L_\infty$-structure $\mu=\sum_{i=1}^\infty l_i$ on a graded vector space $E$ equipped with
a graded commutative product $\wedge : E_i \times E_j \to E_{i+j}$
is called {\em multiplicative} if all the multi-linear brackets $l_i$
are multi-derivations.
\end{defn}

Next, we discuss the relation between multiplicative $L_\infty$-structures and Lie algebroids.

A \emph{pre-Lie algebroid} structure on a vector bundle $A\to M$ over a manifold $M$ is a pair $(\rho, \left[.,.\right])$, with $\rho: A \to TM $
a vector bundle morphism over the identity of $M$, called \emph{anchor map}, and $[ .,.] $
a skew-symmetric bilinear endomorphism of $\Gamma(A) $
subject to the so-called Leibniz identity:
   $$ [X,fY] = f[X,Y] + (\rho(X) f) \, Y,$$
   for all $X,Y \in \Gamma(A)$ and all $f \in C^\infty (M) $.
   When, moreover, $ \left[.,.\right]$ is a Lie algebra bracket, the pair $(\left[.,.\right],\, \rho)$ is called a
   \emph{Lie algebroid structure on $A\to M$}. We denote by $ \left[ ., . \right]_{_{SN}}$ the \emph{Schouten-Nijenhuis} bracket on the space of multivectors of the (pre-)Lie algebroid $A$ and by $\diff^A$ the (pre-)differential of $A$.

Let $(\left[.,.\right],\, \rho)$ be a pre-Lie algebroid structure on a vector bundle $A \to M$. Set $E_i:=\Gamma(\wedge^{i+1} A)$ and $ E=\oplus_{i\geq -1}E_i$ , with $E_{-1}=\Gamma(\wedge^0A)=\mathcal{C}^{\infty}(M)$.
The Schouten-Nijenhuis bracket is a graded skew-symmetric bracket of degree zero on $E=\oplus_{i\geq -1}E_i$ and it is known that a pre-Lie algebroid structure $(\rho,\left[.,.\right]) $ is a Lie algebroid structure on the vector bundle $A \to M,$ if and only if $\left[.,.\right]_{_{SN}}$ is a graded Lie algebra bracket on $E=\Gamma(\wedge A)[1]$. It is also well known that the pre-differential $\diff^A$ is a derivation of $\Gamma(\wedge A^*)$ and that $\diff^{A}$ squares to zero if and only if $(A,\,\left[.,.\right],\,\rho)$ is Lie algebroid.

    The discussion above leads to the conclusion that there are two ways to see Lie algebroids as $L_\infty$-structures: the first one will make it an $L_{\infty}$-structure on $\Gamma(\wedge A) $, and the second one will make it an $L_\infty$-structure on $\Gamma(\wedge A^*) $ \cite{YKS}. More precisely:

\begin{prop}\label{prop:algebroid}
   Let  $A \to M $ be a vector bundle and $A^* \to M $ its dual. There is a one to one correspondence between:
   \begin{enumerate}
   \item[(i)] pre-Lie algebroid structures $(\rho, \left[.,.\right]) $ on $A \to M$,
   \item[(ii)] binary multi-derivations of $\Gamma(\wedge A)[2] $ of degree $1$,
   \item[(iii)] unary multi-derivations of $\Gamma(\wedge A^*)[2] $ of degree $1$.
   \end{enumerate}
   The one to one correspondence above restricts to a one to one correspondence between:
   \begin{enumerate}
   \item[($i'$)] Lie algebroid structures $(\rho, \left[.,.\right]) $ on $A \to M$,
   \item[($ii'$)] multiplicative $L_{\infty}$-structures on $\Gamma(\wedge A)[2] $ given by a binary bracket,
   \item[($iii'$)] multiplicative $L_{\infty}$-structures on $\Gamma(\wedge A^*)[2] $ given by a unary bracket.
   \end{enumerate}
   \end{prop}

%%%%%%%%%%%%%%%%%%%%%%%%%%%%%%%%%%%%%%%%%%%%%%%%%%%%%%%%%%%%%%%%%%%%%%%%%%%%%%%%%%%%%%%%%%%%%%%%%%%%%%%%%%%%%%%%%%%%%%%%%%%%%%%%%%%%%%%%%%%%%%%%%%%%%%%%%%%%%%

  %%%%%%%%%%%%%%%%%%%%%%%%%%%%%%%%%%%%%%%%%%%%%%%%%%%%%%%%%%%%%%%%%%%%%%%%%%%%%%%%%%%%%%%%%%%%%%%%%%%%%%%%%%%%%%%%%%%%%%%%%%%%%%%%%%%%%%%%%%%

 Given a  $(1,1)$-tensor $N$ on a Lie algebroid $(A,\left[.,.\right],\rho)$, we define a linear map $\underline{N}$ on the graded vector space $\Gamma(\wedge A)[2]$, by setting
  \begin{equation*}
  \underline{N}(f):=0,
  \end{equation*}
  for all $f\in \mathcal{C}^{\infty}(M)$, and
\begin{equation*}
\underline{N}(P):=\sum_{i=1}^{p}P_1\wedge \cdots \wedge P_{i-1}\wedge N(P_i)\wedge P_{i+1}\wedge \cdots \wedge P_p,
\end{equation*}
for all monomial multi-sections $P=P_1\wedge \cdots \wedge P_{p}\in \Gamma(\wedge A)[2]$. The map $\underline{N}$ is called the {\em extension of N by derivation} on the graded vector space $\Gamma(\wedge A)[2]$.
It is a  multi-derivation on the graded vector space $\Gamma(\wedge A)[2]$, hence a symmetric vector valued $1$-form on $\Gamma(\wedge A)[2]$, and has degree zero.

 For  a $k$-form on a Lie algebroid, we also consider its extension by derivation. More precisely, if $\kappa \in\Gamma(\wedge^kA^*)$, the extension of $\kappa$ by derivation is a $k$-linear map, denoted by $\underline{\kappa}$,  given by
\begin{equation*}
\underline{\kappa}(P_1,\cdots,P_k):=\sum_{i_1,\cdots, i_k=1}^{p_1,\cdots, p_k}(-1)^{\spadesuit}\kappa(P_{1,i_1},\cdots,P_{k,i_k})\widehat{P_{1,i_1}}\wedge\cdots\wedge\widehat{P_{k,i_k}},
\end{equation*}
for all homogeneous multi-sections $P_i=P_{i,1}\wedge\cdots\wedge P_{i,p_i} \in \Gamma(\wedge^{p_i} A)$, with $i=1,\cdots,k,$ where $1\leq i_j\leq p_j$ for all $1\leq j\leq k$,
\begin{equation*}
\widehat{P_{j,i_j}}=P_{j,1}\wedge\cdots\wedge P_{j,i_j-1}\wedge P_{j,i_j+1}\wedge\cdots\wedge P_{j,p_j}\in \Gamma(\wedge^{p_j-1} A)
\end{equation*}
and
\begin{equation*}
\spadesuit=2p_1+3p_2+\cdots+(k+1)p_k+i_1+\cdots +i_k+1.
\end{equation*}

It follows from its definition that $\underline{\kappa}$  is a multi-derivation on the graded vector space $\Gamma(\wedge A)[2]$  and that it is a symmetric vector valued $k$-form  of degree $k-2$ on  $\Gamma(\wedge A)[2]$.
\begin{lem}\label{underlinescommut}
Let $(A, \left[.,.\right], \rho)$ be a Lie algebroid, $\alpha \in \Gamma(\wedge^k A^*)$ be a $k$-form and $\beta \in \Gamma(\wedge^l A^*)$ be an $l$-form. Then,
\begin{equation*}
\left[\underline{\alpha},\underline{\beta}\right]_{_{RN}}=0.
\end{equation*}
\end{lem}
\begin{proof}
The fact that $\underline{\alpha}$ (respectively $\underline{\beta}$) is a vector valued $k$-form (respectively $l$-form) of degree $k-2$ (respectively $l-2$), imply that $\left[\underline{\alpha},\underline{\beta}\right]_{_{RN}}$ is a vector valued $(k+l-1)$-form of degree $k+l-4$ on the graded vector space $\Gamma(\wedge A)=\oplus_{i\geq 0}\Gamma(\wedge^i A)$. Therefore, for all $l,k\geq 0$ the restriction of $\left[\underline{\alpha},\underline{\beta}\right]_{_{RN}}$ to the space of sections is zero and hence we have $\left[\underline{\alpha},\underline{\beta}\right]_{_{RN}}=0$, because $\left[\underline{\alpha},\underline{\beta}\right]_{_{RN}}$ is a multi-derivation and it is uniquely determined on the space of sections.
\end{proof}

According to Proposition \ref{prop:algebroid}, for a given Lie algebroid $(A,\,\left[.,.\right],\rho)$, the bracket $l_2^{\left[.,.\right]}$ given by
\begin{equation}\label{multiplicativestructure}
l_2^{\left[.,.\right]}(P,Q)=(-1)^{p-1}[P,Q]_{_{SN}}, \,\,\, P\in \Gamma(\wedge^{p}A) , Q\in \Gamma(\wedge^{q}A),
\end{equation}
 defines a multiplicative graded Lie algebra structure on $\Gamma(\wedge A)[2]$.
When we deform the bracket $\left[.,.\right]$ by $N$ as
\begin{equation*}
\left[X,Y\right]_N=\left[NX,Y\right]+\left[X,NY\right]-N\left[X,Y\right],
\end{equation*}
for all $X,Y\in \Gamma(A)$, of course we may consider $l_2^{\left[.,.\right]_N}$ using Equation (\ref{multiplicativestructure}) and we may take the Schouten-Nijenhuis bracket $\left[.,.\right]^{N}_{_{SN}}$ corresponding to the deformed bracket $\left[.,.\right]_N$. Note that the bracket $l_2^{\left[.,.\right]_N}$ is not necessarily a multiplicative graded Lie algebra structure.  On the other hand, since $l_2^{\left[.,.\right]}$ is a symmetric vector valued $2$-form of degree $1$ and $\underline{N}$ is a (symmetric) vector valued $1$-form of degree zero, we can consider the deformation of $l_2^{\left[.,.\right]}$ by $\underline{N}$. The following lemma shows the relation between
$\left[\underline{N},l_2^{\left[.,.\right]}\right]_{_{RN}}$ and $l_2^{\left[.,.\right]_N}$.
\begin{lem}\label{lem:deformingOidsN}
 Let $N$ be a $(1,1)$-tensor on a Lie algebroid $(A,\left[.,.\right],\rho)$.
 Then, we have
  \begin{equation*}\label{eq:underlineN}
   \left[\underline N , l_2^{\left[.,.\right]} \right]_{_{RN}} = l_2^{\left[.,.\right]_N}.
   \end{equation*}
  \end{lem}
\begin{proof}
The proof follows directly from the fact that the Schouten-Nijenhuis bracket on $\Gamma(\wedge A) $
    associated to the bracket $\left[.,.\right]_N$  is given by
 \begin{equation*}  \label{NSNbracket}
  \left[P,Q\right]^{N}_{_{SN}}=\left[\underline N P,Q\right]_{_{SN}} + \left[P, \underline N Q\right]_{_{SN}} - \underline N \left[P,Q\right]_{_{SN}},\end{equation*}
 for all $P,Q \in \Gamma(\wedge A)$, see \cite{YKS}.
\end{proof}
We will need the following lemma for our next purpose.
\begin{lem}\label{lem:commutatorAlgebroidContraction}
Let $(A,\left[.,.\right], \rho)$ be a Lie algebroid, with  differential $\diff^A$ and associated
multiplicative graded Lie algebra structure $l_2^{\left[.,.\right]}$ on  $\Gamma(\wedge A)[2]$. Then,
  \begin{equation*}
  \left[\underline{\alpha}, l_2^{\left[.,.\right]} \right]_{_{RN}}= \underline{\diff^A \alpha},
  \end{equation*}
  for all $\alpha \in \Gamma(\wedge^n A^*)$.
\end{lem}
\begin{proof}
We shall prove the statement for $n=2$. A similar proof can be done for any $n\geq 1$. First note that $\left[\underline{\alpha}, l_2^{\left[.,.\right]} \right]_{_{RN}}$ is a vector valued $3$-form of degree $1$ on the graded vector space $\Gamma(\wedge A)[2]$. This implies that the restriction of $\left[\underline{\alpha}, l_2^{\left[.,.\right]} \right]_{_{RN}}$ to $\Gamma(A)$ is of the form:
\begin{equation*}
\left[\underline{\alpha}, l_2^{\left[.,.\right]} \right]_{_{RN}}|_{\Gamma(A) \times \Gamma(A) \times \Gamma(A)}:\Gamma(A)\times \Gamma(A)\times \Gamma(A)\to {\mathcal C}^{\infty}(M)
\end{equation*}
and, by degree reasons, any other restriction of $\left[\underline{\alpha}, l_2^{\left[.,.\right]} \right]_{_{RN}}$ is zero. On the other hand, by Proposition \ref{prop:Multider}, $\left[\underline{\alpha}, l_2^{\left[.,.\right]} \right]_{_{RN}}$ is a multi-derivation, so that its restriction to the sections $\Gamma(A)$ is a ${\mathcal C}^{\infty}(M)$-linear map. Therefore,  $\left[\underline{\alpha}, l_2^{\left[.,.\right]} \right]_{_{RN}}\in \Gamma(\wedge^3A^*)$. Next, we show that
\begin{equation*}\label{restriction to sections}
\left[\underline{\alpha}, l_2^{\left[.,.\right]} \right]_{_{RN}}|_{\Gamma(A) \times \Gamma(A) \times \Gamma(A)}=\diff^A\alpha
\end{equation*}
and this together with the fact that $\left[\underline{\alpha}, l_2^{\left[.,.\right]} \right]_{_{RN}}$ is a multi-derivation will imply that
$\left[\underline{\alpha}, l_2^{\left[.,.\right]} \right]_{_{RN}}=\underline{\diff^A\alpha}$,
by the uniqueness of extension by derivation of $\diff^A\alpha$ to the graded vector space $\Gamma(\wedge A)[2]$. A direct computation shows that
\begin{eqnarray*}
\left[\underline{\alpha}, l_2^{\left[.,.\right]} \right]_{_{RN}}(X,Y,Z)&
=&\left[\alpha(X,Y),Z\right]_{_{SN}}-\alpha(\left[X,Y\right],Z)+c.p.\\
&=&\rho(Z)\alpha(X,Y)-\alpha(\left[X,Y\right],Z)+c.p.\\
&=&\diff\alpha^A(X,Y,Z)
\end{eqnarray*}
for all $X,Y,Z\in \Gamma(A)$.
This completes the proof.
\end{proof}

%%%%%%%%%%%%%%%%%%%%%%%%%%%%%%%%%%%%%%%%%%%%%%%%%%%%%%%%%%%%%%%%%%%%%%%%%%%%%%%%%%%%%%%%%%%%%%%%%%%%%%%%%%%%%%%%%%%
%%%%%%%%%%%%%%%%%%%%%%%%%%%%%%%%%%%%%%%%%%%%%%%%%%%%%%%%%%%%%%%%%%%%%%%%%%%%%%%%%%%%%%%%%%%%%%%%%%%%%%%%%%%%%%%%%%%%%
\section{Nijenhuis forms on multiplicative $L_{\infty}$-structures associated to Lie algebroids}

In this section we consider several structures defined by tensors and pairs of tensors on a Lie algebroid and, by using their extensions by derivations, we construct Nijenhuis forms (weak Nijenhuis and co-boundary Nijenhuis, in some cases) with respect to the graded Lie algebra associated to the Lie algebroid structure.

Let $(A,\left[.,.\right],\rho)$ be a Lie algebroid and $N:A \to A$ an endomorphim. Then, as in the case of Lie algebras, the Nijenhuis torsion of $N$ with respect to the Lie bracket $\left[.,.\right]$, denoted by $T_{\left[.,.\right]}N$, is defined by Equation (\ref{Torsion}) and again a direct computation shows that
\begin{equation*}
T_{\left[.,.\right]}N(X,Y)=\frac{1}{2}\left(\left[X,Y\right]_{N,N}-\left[X,Y\right]_{N^2}\right),
\end{equation*}
for all $X,Y \in \Gamma(A)$.
A $(1,1)$-tensor $N$ on a Lie algebroid $(A,\left[.,.\right],\rho)$ is said to be \emph{Nijenhuis} if the Nijenhuis torsion of $N$, with respect to the Lie algebroid bracket $\left[.,.\right]$, vanishes. As a consequence of Lemma \ref{lem:deformingOidsN}, we have the following proposition:
\begin{prop}\label{prop:NijenhuisAsNijenhuis}
For every Nijenhuis tensor  $N$ on a Lie algebroid $(A,\left[.,.\right],\rho)$,
the extension $\underline N $ of $N$ by derivation is a Nijenhuis vector valued $1$-form with respect to the multiplicative graded Lie algebra structure $l_2^{\left[.,.\right]}$ on the graded vector space $\Gamma(\wedge A)[2]$, with square ${\underline{(N^2)}}$.
\end{prop}
\begin{proof}
 Applying Lemma \ref{lem:deformingOidsN} twice, for the tensor $N$ and the bracket $l_2^{\left[.,.\right]}$, we get $\left[\underline N ,\left[\underline N , l_2^{\left[.,.\right]}\right]_{_{RN}}\right]_{_{RN}} = l_2^{\left[.,.\right]_{N,N}} $. The same lemma gives $\left[\underline {N^2} , l_2^{\left[.,.\right]}\right]_{_{RN}} = l_2^{\left[.,.\right]_{N^2}}$. Since $N$ is a Nijenhuis $(1,1)$-tensor on $A$, we have $l_2^{\left[.,.\right]_{N,N}}=l_2^{\left[.,.\right]_{N^2}}$, which implies that $\left[\underline N ,\left[\underline N , l_2^{\left[.,.\right]}\right]_{_{RN}}\right]_{_{RN}} = \left[\underline {N^2} , l_2^{\left[.,.\right]}\right]_{_{RN}}$.
Also, ${\underline{(N^2)}} $ and ${\underline{N}} $ commute with respect to the Richardson-Nijenhuis bracket.
\end{proof}
In the next proposition we obtain a Nijenhuis vector valued form which is the sum of a vector valued $1$-form with a vector valued $2$-form.
\begin{prop}
Let $(A, \left[.,.\right], \rho)$ be a Lie algebroid, with differential $\diff^A$ and associated
multiplicative graded Lie algebra structure $l_2^{\left[.,.\right]}$ on $\Gamma(\wedge A)[2]$. Then, for every section $\alpha \in \Gamma(\wedge^2 A^*)$,
$S+\underline{\alpha} $ is a Nijenhuis vector valued form with respect to $l_2^{\left[.,.\right]}$, with square $ S+2\,\underline{\alpha}$.
The deformed structure is $l_2^{\left[.,.\right]}+ \underline{\diff^A\alpha} $.
\end{prop}
\begin{proof}
As a direct consequence of Lemma \ref{lem:commutatorAlgebroidContraction}, we have
\begin{equation*}\label{eq:deformedstructurealpha}
\left[S+\underline{\alpha},l_2^{\left[.,.\right]}\right]_{_{RN}} = l_2^{\left[.,.\right]}+ \underline{\diff^A\alpha}.
\end{equation*}
A simple computation gives
\begin{equation*}
\left[S+\underline{\alpha},\left[S+\underline{\alpha},l_2^{\left[.,.\right]}\right]_{_{RN}}\right]_{_{RN}}=l_2^{\left[.,.\right]}+ 2\,\underline{\diff^A\alpha}=\left[S+ 2 \,\underline{\alpha},l_2^{\left[.,.\right]}\right]_{_{RN}}
\end{equation*}
and the fact that $ \left[S+\underline{\alpha},S+2\underline{\alpha}\right]_{_{RN}}=0$ completes the proof.
\end{proof}

Our next purpose is to use well-known structures on a Lie algebroid defined by pairs of compatible tensors, such as $\Omega N$-, Poisson-Nijenhuis and $P \Omega$-structures \cite{KoRou,antunes, AntunesCosta}, to construct Nijenhuis forms on the multiplicative graded Lie algebra associated to the Lie algebroid. We start by recalling what an $\Omega N$-structure is.

\begin{defn} \cite{antunes, KoRou}
Let $(A,\left[.,.\right], \rho)$ be a Lie algebroid, with differential $\diff^A$, $N$ be  a $(1,1)$-tensor on $A$ and $\alpha \in \Gamma(\wedge^2 A^*)$ a $2$-form. Let $\alpha_{_N}:\Gamma(A)\times\Gamma(A)\to\Gamma(A)$ be a bilinear map, defined as
\begin{equation}  \label{omeganskew}
\alpha_{_N}(X,Y)=\alpha(NX,Y).
\end{equation}
 Then, the pair $(\alpha, N)$ is an $\Omega N$-{\em structure} on the Lie algebroid $A$ if $\alpha(NX,Y)=\alpha(X,NY)$ for all $X,Y \in \Gamma(A)$ (which amounts to  $\alpha_{_N}$ being skew-symmetric and therefore a $2$-form on $A$), and $\alpha$ and $\alpha_{_N}$ are $\diff^A$-closed.
\end{defn}

\begin{lem}\label{omega_N}
Let $(A,\left[.,.\right], \rho)$ be a Lie algebroid, with differential $\diff^A$ and with the associated multiplicative graded Lie algebra structure $l_2^{\left[.,.\right]}$ on the graded vector space $\Gamma(\wedge A)[2]$. Let $N$ be a  $(1,1)$-tensor on the Lie algebroid and $\alpha\in \Gamma(\wedge^2 A^*)$ be a $2$-form such that  $\alpha_N:\Gamma(A)\times\Gamma(A)\to\Gamma(A)$ given by (\ref{omeganskew})
is skew-symmetric and therefore a $2$-form on $A$. Then,
\begin{enumerate}
\item[i)] $\left[\underline{N},\underline{\alpha}\right]_{_{RN}}=2 \underline{\alpha_{_N}},$
\item[ii)] $\left[\underline{N}+\underline{\alpha},l_2^{\left[.,.\right]}\right]_{_{RN}}=l_2^{\left[.,.\right]_N}+\underline{\diff^A\alpha}$
\item[iii)] If $N$ is Nijenhuis, then $$\left[\underline{N}+\underline{\alpha},\left[\underline{N}+\underline{\alpha},
      l_2^{\left[.,.\right]}\right]_{_{RN}}\right]_{_{RN}}=
      \left[\underline{N^2},l_2^{\left[.,.\right]}\right]_{_{RN}}-2\,\underline{\diff^A\alpha_{_{N}}}+2\left[\underline{N},
      \underline{\diff^A\alpha}\right]_{_{RN}}.$$
\end{enumerate}
\end{lem}
\begin{proof}
i) First notice that for all $X,Y\in \Gamma(A)$ we have $$\left[\underline{N},\underline{\alpha}\right]_{_{RN}}(X,Y)=\alpha(NX,Y)-\alpha(NY,X)=2\alpha_{_N}(X,Y).$$
 Since $\underline{N}$ and $\underline{\alpha}$ are both derivations, by Lemma \ref{lem:multiDer1} $\left[\underline{N},\underline{\alpha}\right]_{_{RN}}$ is a derivation and hence it is the unique extension of $2\alpha_{_N}$ by derivation.

ii) It is a direct consequence of Lemma \ref{lem:deformingOidsN} together with Lemma \ref{lem:commutatorAlgebroidContraction}.

iii) Using item (ii) and Lemma \ref{lem:deformingOidsN}, we have
\begin{equation*}
      \begin{array}{rcl}
      \left[\underline{N}+\underline{\alpha},\left[\underline{N}+\underline{\alpha},
            l_2^{\left[.,.\right]}\right]_{_{RN}}\right]_{_{RN}}
      &=&\left[\underline{N}+\underline{\alpha},l_2^{\left[.,.\right]_N}+\underline{\diff^A\alpha} \right]_{_{RN}}\\
      &=&l_2^{\left[.,.\right]_{N,N}}+\left[\underline{N},\underline{\diff^A\alpha}\right]_{_{RN}}+\left[\underline{\alpha},
      l_2^{\left[.,.\right]_{N}}\right]_{_{RN}}
      +\left[\underline{\alpha},\underline{\diff^A\alpha}\right]_{_{RN}}.
      \end{array}
\end{equation*}
 Lemma \ref{lem:deformingOidsN} and the graded Jacobi identity give
\begin{equation*}
\begin{array}{rcl}
\left[\underline{\alpha},l_2^{\left[.,.\right]_{N}}\right]_{_{RN}}&=&
\left[\underline{\alpha},\left[\underline{N},l_2^{\left[.,.\right]}\right]_{_{RN}}\right]_{_{RN}}\\
&=&\left[\left[\underline{\alpha},\underline{N}\right]_{_{RN}},l_2^{\left[.,.\right]}\right]_{_{RN}}+
\left[\underline{N},\left[\underline{\alpha},l_2^{\left[.,.\right]}\right]_{_{RN}}\right]_{_{RN}}\\
&=&\left[-2\underline{\alpha_{_N}},l_2^{\left[.,.\right]}\right]_{_{RN}}+\left[\underline{N},\underline{\diff^A\alpha}\right]_{_{RN}}
\end{array}
\end{equation*}
and, by Lemma \ref{underlinescommut}, $\left[\underline{\alpha},\underline{\diff^A\alpha}\right]_{_{RN}}=0.$ Hence, since $N$ is Nijenhuis, we get
\begin{equation*}
\begin{array}{rcl}
[\underline{N}+\underline{\alpha},[\underline{N}+\underline{\alpha},l_2^{\left[.,.\right]}]_{_{RN}}]_{_{RN}}&=& \left[\underline{N^2}-2\underline{\alpha_{_N}},l_2^{\left[.,.\right]}\right]_{_{RN}}+2\left[\underline{N},\underline{\diff^A\alpha}\right]_{_{RN}}\\
&=&\left[\underline{N^2},l_2^{\left[.,.\right]}\right]_{_{RN}}-2\,\underline{\diff^A\alpha_{_{N}}}+2\left[\underline{N},\underline{\diff^A\alpha}\right]_{_{RN}}.
\end{array}
\end{equation*}
\end{proof}

The next proposition is now immediate.

\begin{prop}\label{NijenhuisonOmegaNstructures}
Let $(A,\left[.,.\right], \rho)$ be a Lie algebroid, with  differential $\diff^A$ and with associated multiplicative graded Lie algebra structure $l_2^{\left[.,.\right]}$ on the graded vector space $\Gamma(\wedge A)[2]$. If $(\alpha, N)$ is an $\Omega N$-structure on the Lie algebroid $A$, then $\underline{N}+\underline{\alpha}$ is a Nijenhuis vector valued form, with respect to $l_2^{\left[.,.\right]}$, with square $\underline{N^2}+\underline{\alpha_{_{N}}}$.
\end{prop}
\begin{proof}
Let $(\alpha, N)$ be an $\Omega N$-structure on the Lie algebroid $A$. Then, $\diff^A \alpha_N =0$ and, by Lemma~\ref{lem:commutatorAlgebroidContraction},  we have $\left[\underline{\alpha_{_{N}}},l_2^{\left[.,.\right]}\right]_{_{RN}}=0$. It follows from item (iii) in Lemma \ref{omega_N}, that
\begin{equation*}
\left[\underline{N}+\underline{\alpha},\left[\underline{N}+\underline{\alpha},
      l_2^{\left[.,.\right]}\right]_{_{RN}}\right]_{_{RN}}=
      \left[\underline{N^2}+\underline{\alpha_{_{N}}},l_2^{\left[.,.\right]}\right]_{_{RN}}.
\end{equation*}
Since
\begin{equation*}
\left[\underline{N}+\underline{\alpha},\underline{N^2}+\underline{\alpha_{_{N}}}\right]_{_{RN}}=\left[\underline{N},\underline{\alpha_{_N}}\right]_{_{RN}}+
\left[\underline{\alpha},\underline{N^2}\right]_{_{RN}}=2\underline{(\alpha_{_{N}})_{_N}}-2\underline{\alpha_{_{N^2}}}=0,
\end{equation*}
the proof is complete.
\end{proof}

We are now going to see how to include Poisson-Nijenhuis structures among our examples of Nijenhuis
structures on $L_\infty$-algebras.
Let us first fix and recall some notations and notions.

Let $(A,\mu=\left[.,.\right], \rho)$ be a Lie algebroid, $\pi\in \Gamma(\wedge^2A)$  a bivector and $N: A \to A$  a vector bundle morphism.
We denote by $N^*$ the morphism $N^*: A^*\to A^*$ given by $\langle N^*\alpha,X\rangle=\langle \alpha,NX\rangle$, for all $X,Y \in \Gamma(A)$. We consider the
morphism induced by $\pi$, $\pi^{\#}: A^* \to A$,  given by $\langle \beta,\pi^{\#}\alpha\rangle=\pi(\alpha, \beta)$,   and we denote by $\pi_{_{N}}$ the  bivector defined by \begin{equation}  \label{1formbracket}
\pi_{_{N}}(\alpha,\beta)=\langle \beta,N\pi^{\#}\alpha\rangle=\langle N^*\beta,\pi^{\#}\alpha\rangle,\end{equation}
 for all $\alpha,\beta \in  \Gamma(A^*)$.
  A bracket $\{\cdot,\cdot\}_{_{\pi}}^{\mu}$ can be defined on $\Gamma(A^*)$, the space of $1$-forms on the Lie algebroid $(A,\mu=\left[.,.\right], \rho)$, as follows:
\begin{equation*}
\{\alpha,\beta\}_{_{\pi}}^{\mu}=\mathcal{L}^A_{_{\pi^\#(\alpha)}}\beta -\mathcal{L}^A_{_{\pi^\#(\beta)}}\alpha-\diff^A(\pi(\alpha,\beta)),
\end{equation*}
for all $\alpha,\beta\in \Gamma(A^*)$. It is well known that if $\pi$ is a Poisson bivector on the Lie algebroid $(A,\mu=\left[.,.\right], \rho)$, that is $\left[\pi,\pi\right]_{_{SN}}=0$, then $(\Gamma(A^*), \{.,.\}_{_{\pi}}^{\mu})$ is a Lie algebra and if this is the case, then $\pi^{\#}$ is a Lie algebra morphism form the Lie algebra $(\Gamma(A^*), \{.,.\}_{_{\pi}}^{\mu})$ to the Lie algebra $(\Gamma(A),\mu)$.

For every Poisson structure $\pi$ on a Lie algebroid $A $,
the triple $(\Gamma(\wedge A)[1], \left[.,.\right]_{_{SN}},\left[\pi,.\right]_{_{SN}}) $ is a skew-symmetric differential graded Lie algebra,
so that the pair $(l_1^{\left[.,.\right],\pi},l_2^{\left[.,.\right]})$ given by
\begin{equation*}\label{l_1andl_2}
\begin{array}{rcl}
l_1^{\left[.,.\right],\pi}(P)=\left[\pi,P\right]_{_{SN}}&\mbox{and}&l_2^{\left[.,.\right]}(P,Q) := (-1)^{(p-1)} \left[P,Q\right]_{_{SN}},
\end{array}
 \end{equation*}
where $P\in \Gamma(\wedge^{p} A)$ and $Q\in \Gamma(\wedge^{q} A)$, is an $L_{\infty}$-structure on the graded vector space $\Gamma(\wedge A)[2]$, which is clearly multiplicative.
This $L_\infty$-structure is called the \emph{$L_\infty$-structure associated to the Poisson
structure $\pi$ and the Lie algebroid $A$}.

Now, we recall the notion of Poisson-Nijenhuis structure on a Lie algebroid.
\begin{defn} \cite{YKS}
Let $(A,\mu=\left[.,.\right], \rho)$ be a Lie algebroid, $\pi\in \Gamma(\wedge^2A)$  a bivector and $N$  a $(1,1)$-tensor on $A$. Then, the pair $(\pi, N)$ is a {\em Poisson-Nijenhuis structure} on the Lie algebroid $(A,\mu=\left[.,.\right], \rho)$ if
\begin{enumerate}
\item[i)] $N$ is a Nijenhuis $(1,1)$-tensor with respect to the Lie bracket $\mu$,
\item[ii)] $\pi$ is a Poisson bivector,
\item[iii)] $N \circ \pi^{\#}=\pi^{\#} \circ  N^*$,
\item[iv)] $(\{\alpha,\beta\}_{_{\pi}}^{\mu})_{_{N^*}}=\{\alpha, \beta\}_{_{\pi}}^{\mu^{N}}$,
\end{enumerate}
for all $\alpha,\beta\in \Gamma(A^*)$, where $(\{.,.\}_{_{\pi}}^{\mu})_{_{N^*}}$ is the deformation of the Lie bracket $\{\cdot,\cdot\}_{_{\pi}}^{\mu}$ by $N^*$ and $\{.,.\}_{_{\pi}}^{\mu^{N}}$ is the bracket determined by the pair $(\pi, \mu^{N}=\left[.,.\right]_{N})$ according to formula (\ref{1formbracket}).
\end{defn}
Notice  that
$
\pi_{_{N}}^{\#}=N\pi^{\#}=\pi^{\#}N^*
$
and hence,
\begin{equation} \label{npi}
\underline{N}(\pi)=2\pi_{_{N}}.
\end{equation}

 Recall from \cite{YKS} that if $(\pi, N)$ is a Poisson-Nijenhuis structure on a Lie algebroid $(A,\mu=\left[.,.\right], \rho)$, then $\left(A,\mu^{N}=\left[.,.\right]_{N}, \rho\circ N\right)$ and $\left(A^*, \{.,.\}_{\pi}^{\mu}, \rho\circ\pi^{\#}\right)$ are Lie algebroids. Also,
 $$\left(\left(\{.,.\}_{\pi}^{\mu}\right)_{N^*},\rho\circ\\\pi^{\#}\circ N^*\right), \, \left(\{.,.\}_{\pi}^{\mu^{N}}, \rho\circ N\circ\pi^{\#}\right)\,\, {\mbox{and}}\,\, \left(\{.,.\}_{\pi_{_{N}}}^{\mu}, \rho\circ\pi_{_{N}}^{\#}\right)$$
 define the same Lie algebroid structure on $A^*$. Moreover, identifying the graded vector spaces $\Gamma(\wedge A^{**})$ and $\Gamma(\wedge A)$, the  differential  $ \diff^{A^{*}}_{(\{.,.\}_{\pi}^{\mu})}$ coincide with the linear map $\left[\pi,.\right]_{_{SN}}$. Hence, we have
 $$\diff^{A^{*}}_{(\{.,.\}_{\pi}^{\mu^{N}})}=\diff^{A^{*}}_{(\{.,.\}_{\pi_{_{N}}}^{\mu})},$$ which  is equivalent to
 \begin{equation} \label{npi1}
 \left[\pi,.\right]^{N}_{_{SN}}=\left[\pi_{_{N}},.\right]_{_{SN}},
 \end{equation}  where $\left[.,.\right]^{N}_{_{SN}}$
 is the Schouten-Nijenhuis bracket with respect to the Lie bracket $\left[.,.\right]_N$.

\begin{lem}\label{lamejanjali}
Let $(\pi,N)$ be a Poisson-Nijenhuis structure on a Lie algebroid $(A,\left[.,.\right], \rho)$. Then,
\begin{equation*}
\left[\underline{N},l_1^{\left[.,.\right],\pi}\right]_{_{RN}}(P)= \left[\pi,{\underline{ N}} (P)\right]_{_{SN}}-\underline {N} \left[\pi, P \right]_{_{SN}}  = \left[-\pi_N , P\right]_{_{SN}},
\end{equation*}
for all $P \in \Gamma(\wedge A)$.
\end{lem}
\begin{proof}
The first equality follows directly from the definition of $l_1^{\left[.,.\right],\pi}$. For the second equality, observe that for all $P \in \Gamma(\wedge A)$ we have
\begin{equation*}
\left[\pi,P\right]^{N}_{_{SN}}=\left[\underline{N}(\pi),P\right]_{_{SN}}+\left[\pi,\underline{N}(P)\right]_{_{SN}}-\underline{N}\left[\pi,P\right]_{_{SN}},
\end{equation*}
 where $\left[.,.\right]^{N}_{_{SN}}$ stands for the Schouten-Nijenhuis bracket with respect to the Lie bracket $\left[.,.\right]_{N}$.
Hence, using (\ref{npi})  and (\ref{npi1}), we have
\begin{equation*}
\begin{array}{rcl}
\left[\pi,{\underline{ N}} (P)\right]_{_{SN}}-\underline {N} \left[\pi, P \right]_{_{SN}}&=&\left[\pi,P\right]^{N}_{_{SN}}-\left[\underline{N}(\pi),P\right]_{_{SN}}
               =\left[\pi,P\right]^{N}_{_{SN}}-2\left[\pi_{_{N}},P\right]_{_{SN}}\\
               &=&\left(\left[\pi,P\right]^{N}_{_{SN}}-\left[\pi_{_{N}},P\right]_{_{SN}}\right)-\left[\pi_{_{N}},P\right]_{_{SN}}\\
               &=&-\left[\pi_{_{N}},P\right]_{_{SN}}.
\end{array}
\end{equation*}
\end{proof}

\begin{prop}\label{NijenhuisonPoissonNijenhuisalgeroid}
Let $(\pi,N)$ be a Poisson-Nijenhuis structure on a Lie algebroid $(A,\left[.,.\right],\rho)$.
Then, the derivation $\underline N $ is a weak Nijenhuis tensor
for the $L_\infty$-structure associated to the Poisson
structure $\pi$ and the Lie algebroid $(A,\left[.,.\right],\rho)$.

In this case, the deformed structure $[ \underline N,l_1^{\left[.,.\right],\pi} + l_2^{\left[.,.\right]}]_{_{RN}}$ is the $L_\infty$-structure associated to
the Poisson structure $-\pi_N $ and the Lie algebroid $(A, \left[.,.\right]_N, \rho \circ N)$.
\end{prop}
\begin{proof}
Lemmas \ref{lamejanjali} and \ref{lem:deformingOidsN} imply that
\begin{equation*}
\left[\underline{N},l_1^{\left[.,.\right],\pi}+l_2^{\left[.,.\right]}\right]_{_{RN}}=-l_1^{\left[.,.\right],\pi_{_{N}}}+l_2^{\left[.,.\right]_{_{N}}}.
\end{equation*}
Hence,
\begin{equation}\label{avalinesh}
\begin{array}{rcl}
\left[\underline{N},\left[\underline{N},l_1^{\left[.,.\right],\pi}+l_2^{\left[.,.\right]}\right]_{_{RN}}\right]_{{RN}}
&=&l_1^{\left[.,.\right],\pi_{_{N,N}}}+l_2^{\left[.,.\right]_{_{N,N}}}
=l_1^{\left[.,.\right],\pi_{_{N^2}}}+l_2^{\left[.,.\right]_{_{N^2}}}\\
&=&\left[\underline{N^2}\,,-l_1^{\left[.,.\right],\pi}+l_2^{\left[.,.\right]}\right]_{_{RN}}\\
&=&\left[\underline{N^2}\,,l_1^{\left[.,.\right],\pi}+l_2^{\left[.,.\right]}\right]_{_{RN}}-
2\left[\underline{N^2}\,,l_1^{\left[.,.\right],\pi}\right]_{_{RN}}.
\end{array}
\end{equation}
Denoting $\mu=l_1^{\left[.,.\right],\pi}+l_2^{\left[.,.\right]}$ and using the fact that $\pi_{_{N^{2}}}$ is a Poisson bivector on the Lie algebroid $(A,\left[.,.\right],\rho)$ and hence $(\Gamma(\wedge A)[2],l_1^{\left[.,.\right],\pi_{_{N^2}}}+l_2^{\left[.,.\right]})$ is a symmetric differential graded Lie algebra, we have
\begin{equation}\label{dovominesh}
\begin{array}{rcl}
\left[\mu,\left[\underline{N},\left[\underline{N},\mu\right]_{_{RN}}\right]_{_{RN}}\right]_{_{RN}}
&=&\left[\mu,\left[\underline{N^2}\,,\mu\right]_{_{RN}}\right]_{_{RN}}
-2\left[\mu,\left[\underline{N^2}\,,l_1^{\left[.,.\right],\pi}\right]_{_{RN}}\right]_{_{RN}}\\
&=&-2\left[\mu,\left[\underline{N^2}\,,l_1^{\left[.,.\right],\pi}\right]_{_{RN}}\right]_{_{RN}}
=2\left[\mu,l_1^{\left[.,.\right],\pi_{_{N^2}}}\right]_{_{RN}}\\
&=&2\left[l_1^{\left[.,.\right],\pi},l_1^{\left[.,.\right],\pi_{_{N^2}}}\right]_{_{RN}}+
2\left[l_2^{\left[.,.\right]},l_1^{\left[.,.\right],\pi_{_{N^2}}}\right]_{_{RN}}\\
&=&2\left[l_1^{\left[.,.\right],\pi},l_1^{\left[.,.\right],\pi_{_{N^2}}}\right]_{_{RN}}
\end{array}
\end{equation}
and
\begin{equation}\label{sevominesh}
\begin{array}{rcl}
\left[l_1^{\left[.,.\right],\pi},l_1^{\left[.,.\right],\pi_{_{N^2}}}\right]_{_{RN}}(P)&=&
l_1^{\left[.,.\right],\pi_{_{N^2}}}(l_1^{\left[.,.\right],\pi}(P))+
l_1^{\left[.,.\right],\pi}(l_1^{\left[.,.\right],\pi_{_{N^2}}}(P))\\
&=&\left[\pi_{_{N^2}},\left[\pi,P\right]_{_{SN}}\right]_{_{SN}}+\left[\pi,\left[\pi_{_{N^2}},P\right]_{_{SN}}\right]_{_{SN}}\\
&=&\left[\left[\pi,\pi_{_{N^2}}\right]_{_{SN}},P\right]_{_{SN}}\\
&=&0,
\end{array}
\end{equation}
for all $P \in \Gamma(\wedge A)$.
Therefore $\left[\mu,\left[\underline{N},\left[\underline{N},\mu\right]_{_{RN}}\right]_{_{RN}}\right]_{_{RN}}=0$, which means that $\underline{N}$ is a weak Nijenhuis vector valued form with respect to the symmetric differential graded Lie algebra structure $\mu=l_1^{\left[.,.\right],\pi}+l_2^{\left[.,.\right]}$ on the graded vector space
$\Gamma(\wedge A)[2]$.
\end{proof}
There is a second manner to see Poisson-Nijenhuis structures on a Lie algebroid as a Nijenhuis
form.
\begin{prop}
Let $(\pi,N)$ be a Poisson-Nijenhuis structure on a Lie algebroid $(A,\left[.,.\right],\rho)$. Then
 $\underline N + \pi $ is a weak Nijenhuis vector valued form with curvature, with respect to the multiplicative differential graded Lie algebra structure $l_1^{\left[.,.\right],\pi}+l_2^{\left[.,.\right]}$ on the graded vector space $\Gamma(\wedge A)[2]$.
\end{prop}
\begin{proof}
It follows from Lemma \ref{lem:deformingOidsN} that
\begin{equation*} \label{ELONE}
\left[\underline{N}+\pi,l_1^{\left[.,.\right],\pi}\right]_{_{RN}}=
-l_1^{\left[.,.\right],\pi_{_{N}}}+\left[\pi,\pi\right]_{_{SN}}=-l_1^{\left[.,.\right],\pi_{_{N}}},
\end{equation*}
while Lemma \ref{lamejanjali} implies that
\begin{equation*} \label{ELTWO}
\left[\underline{N}+\pi,l_2^{\left[.,.\right]}\right]_{_{RN}}=
l_2^{\left[.,.\right]_N}+l_2^{\left[.,.\right]}(\pi,.)=l_2^{\left[.,.\right]_N}-l_1^{\left[.,.\right],\pi}.
\end{equation*}
Hence, we have
\begin{equation}\label{divane}
\begin{array}{rcl}
&&\left[\underline{N}+\pi,\left[\underline{N}+\pi,l_1^{\left[.,.\right],\pi}+l_2^{\left[.,.\right]}\right]_{_{RN}}\right]_{_{RN}}
=\left[\underline{N}+\pi,-l_1^{\left[.,.\right],\pi_{_{N}}}+l_2^{\left[.,.\right]_N}-l_1^{\left[.,.\right],\pi}\right]_{_{RN}}\\
& &=l_1^{\left[.,.\right],\pi_{_{N,N}}}+l_1^{\left[.,.\right],\pi_{_{N}}}+l_2^{\left[.,.\right]_{N,N}}-l_1^{\left[.,.\right],\pi_{_{N}}}(\pi)
-l_1^{\left[.,.\right],\pi}(\pi)+l_2^{\left[.,.\right]_N}(\pi,.).\\
\end{array}
\end{equation}
But $l_1^{\left[.,.\right],\pi}(\pi)=\left[\pi,\pi\right]_{_{SN}}=0$, $l_1^{\left[.,.\right],\pi_{_{N}}}(\pi)=\left[\pi_{_{N}},\pi\right]_{_{SN}}=0$ and
$
l_1^{\left[.,.\right],\pi_{_{N}}}(P)+l_2^{\left[.,.\right]_N}(\pi,P)=\left[\pi_{_{N}},P\right]_{_{SN}}-\left[\pi,P\right]^{N}_{_{SN}}=0,
$
for all $P\in \Gamma(\wedge A)$, where $\left[.,.\right]^{N}_{_{SN}}$ is the Schouten-Nijenhuis bracket associated to the Lie bracket $\left[.,.\right]_N$.
Hence, (\ref{divane}) can be written as
\begin{equation*}
\left[\underline{N}+\pi,\left[\underline{N}+\pi,l_1^{\left[.,.\right],\pi}+l_2^{\left[.,.\right]}\right]_{_{RN}}\right]_{_{RN}}
=l_1^{\left[.,.\right],\pi_{_{N,N}}}+l_2^{\left[.,.\right]_{N,N}}.
\end{equation*}
Similar computations as in (\ref{avalinesh}), (\ref{dovominesh}) and (\ref{sevominesh}) show that $\left[\mu,\left[\underline{N},\left[\underline{N},\mu\right]_{_{RN}}\right]_{_{RN}}\right]_{_{RN}}=0,$ which means that $\underline{N}$ is weak Nijenhuis vector valued form with respect to the symmetric differential graded Lie algebra structure $\mu=l_1^{\left[.,.\right],\pi}+l_2^{\left[.,.\right]}$ on the graded vector space
$\Gamma(\wedge A)[2]$.
\end{proof}

The next proposition establishes a relation between Poisson-Nijenhuis structures and co-boundary Nijenhuis tensors on a Lie algebroid.

\begin{prop}\label{Poissonifandonlyif}
Let $(A,\left[.,.\right],\rho)$ be a Lie algebroid , $\pi\in \Gamma(\wedge^2 A)$  a bivector and $N$  a $(1,1)$-tensor on $A$  such that
\begin{equation*}\label{costezan}
N \circ \pi^{\#}=\pi^{\#} \circ N^*.
\end{equation*}
Then,
 $\underline{N} + \pi $ is a co-boundary Nijenhuis vector valued form with curvature, with respect to the multiplicative graded Lie algebra structure $l_2^{\left[.,.\right]}$ on the graded vector space $\Gamma(\wedge A)[2]$,
with square ${\underline{N^2}} $, if and only if $(\pi, N)$ is a Poisson-Nijenhuis structure on the Lie algebroid $(A,\left[.,.\right],\rho)$.
The deformed structure $[ \underline N, l_2^{\left[.,.\right]}]_{_{RN}}$ is the $L_\infty$-structure (indeed a differential graded Lie algebra structure) associated to
the Poisson structure $\pi $ on the Lie algebroid $(A,\left[.,.\right]_N,\rho \circ N) $.
\end{prop}
\begin{proof}
Assume that $(\pi,N)$ is a Poisson-Nijenhuis structure on the Lie algebroid $(A,\left[.,.\right],\rho)$. Then,
\begin{equation*}
\left[\underline{N}+\pi,l_2^{\left[.,.\right]}\right]_{_{RN}}=l_2^{\left[.,.\right]_{N}}-l_1^{\left[.,.\right],\pi}
\end{equation*}
and, by  (\ref{npi1}), we get
\begin{equation*}
\begin{array}{rcl}
\left[\underline{N}+\pi\left[\underline{N}+\pi,l_2^{\left[.,.\right]}\right]_{_{RN}}\right]_{_{RN}}
&=&l_2^{\left[.,.\right]_{N,N}}+l_1^{\left[.,.\right],\pi_{_{N}}}-l_1^{\left[.,.\right]_N,\pi}\\
&=&l_2^{\left[.,.\right]_{N,N}}=\left[\underline{N^2},l_2^{\left[.,.\right]}\right]_{_{RN}},
\end{array}
\end{equation*}
which means that $\underline{N}+\pi$ is a co-boundary Nijenhuis with respect to the multiplicative graded Lie algebra structure $l_2^{\left[.,.\right]}$ on the graded vector space
$\Gamma(\wedge A)[2]$, with square ${\underline{N^2}} $.

Conversely, assume that $\underline{N}+\pi$ be a co-boundary Nijenhuis with respect to the multiplicative graded Lie algebra structure $l_2^{\left[.,.\right]}$ on the graded vector space $\Gamma(\wedge A)[2]$, with square ${\underline{N^2}} $, that is,
\begin{equation}  \label{NpiCobound}
\left[\underline{N}+\pi,\left[\underline{N}+\pi,l_2^{\left[.,.\right]}\right]_{_{RN}}\right]_{_{RN}}=\left[\underline{N^2},l_2^{\left[.,.\right]}\right]_{_{RN}}.
\end{equation}
Decomposing by homogeneous components, we get
\begin{equation}  \label{NpiCobound1}
\left[\underline{N}+\pi,\left[\underline{N}+\pi,l_2^{\left[.,.\right]}\right]_{_{RN}}\right]_{_{RN}}=
l_2^{\left[.,.\right]_{N,N}}+
\left(\left[\underline{N},l_2^{\left[.,.\right]}(\pi, .) \right]_{_{RN}}+l_2^{\left[.,.\right]_{N}}(\pi,.)\right)+l_2^{\left[.,.\right]}(\pi,\pi).
\end{equation}
From (\ref{NpiCobound}) and (\ref{NpiCobound1}), we get
\begin{equation} \label{NpiCobound2}
\left[\underline{N^2},l_2^{\left[.,.\right]}\right]_{_{RN}}=l_2^{\left[.,.\right]_{N,N}},
\end{equation}
\begin{equation} \label{NpiCobound3}
 \left(\left[\underline{N},l_2^{\left[.,.\right]}(\pi, .) \right]_{_{RN}}+l_2^{\left[.,.\right]_{N}}(\pi,.)\right)=0
\end{equation}
 and \begin{equation} \label{NpiCobound4}
 l_2^{\left[.,.\right]}(\pi,\pi)=0.
 \end{equation}
 Equation (\ref{NpiCobound2}) is equivalent to $l_2^{\left[.,.\right]_{N,N}}=l_2^{\left[.,.\right]_{N^2}}$, or to $\left[.,.\right]_{N,N}=\left[.,.\right]_{N^2}$, which means that $N$ is a Nijenhuis tensor on $A$. Equation (\ref{NpiCobound4}) means that $\pi$ is Poisson, while Equation (\ref{NpiCobound3}) gives
 \begin{equation*}\label{PoissonshodeNijenhuisham}
\left(\left[\underline{N},l_2^{\left[.,.\right]}(\pi, .)\right]_{_{RN}}+l_2^{\left[.,.\right]_{N}}(\pi,.)\right)(P)=0,
\end{equation*}
or
\begin{equation}\label{NpiCobound5}
\left[\underline{N},l_2^{\left[.,.\right]}(\pi, .)\right]_{_{RN}}(P)=\left[\pi,P\right]^{N}_{_{SN}},
\end{equation}
for all $P\in \Gamma(\wedge A)$.
The definition of  $\left[.,.\right]^{N}_{_{SN}}$ gives
\begin{eqnarray} \label{NpiCobound6}
\left[\pi,P\right]^{N}_{_{SN}}&=& \left[\underline{N}(\pi), P\right]_{_{SN}} + \left[\pi, \underline{N}(P)\right]_{_{SN}} - \underline{N}\left[\pi, P\right]_{_{SN}}\\
&=& 2 \left[\pi_N, P \right]_{_{SN}}+ \left[\underline{N},l_1^{\left[.,.\right],\pi} \right]_{_{RN}}(P)  \nonumber\\
&=& 2 \left[\pi_N, P \right]_{_{SN}} - \left[\underline{N},l_2^{\left[.,.\right]}(\pi, .)\right]_{_{RN}}(P), \nonumber
\end{eqnarray}
where in the second equality we used $\underline{N}(\pi)=2\pi_N$ and the definition of the Richardson-Nijenhuis bracket. From (\ref{NpiCobound5}) and (\ref{NpiCobound6}), we get $$\left[\pi,P\right]^{N}_{_{SN}}=\left[\pi_N, P \right]_{_{SN}}.$$
and this completes the proof that $(\pi, N)$ is a Poisson-Nijenhuis structure on the Lie algebroid $(A,\left[.,.\right],\rho)$.
\end{proof}
Last, we shall say a few words about the so-called $P\Omega$-structures \cite{antunes, KoRou}.
Recall that a \emph{$P \Omega$-structure} on a Lie algebroid $(A,\rho,\left[.,.\right])$
is a pair $(\pi,\omega)$ where $\pi \in \Gamma(\wedge^2 A)$ is a Poisson element
and $\omega \in \Gamma(\wedge^2 A^*)$ is a $2$-form, with $\diff^A \alpha =0$. The $2$-form $\omega \in \Gamma(\wedge^2 A^*)$ determines a morphism $\omega^\flat:A \to A^*$, given by $ \langle Y, \omega^\flat (X) \rangle= \omega (X,Y)$.
Defining a $(1,1)$ tensor $N := \pi^{\#} \circ \omega^\flat$, it is known that $(\pi,N) $ is a Poisson-Nijenhuis structure
while $(\omega, N) $ is an $\Omega N$-structure.

\begin{prop}\label{NijenhuisonPoissonOmega}
Let $(\pi,\omega)$ be a  $P \Omega$-structure on a Lie algebroid $(A,\left[.,.\right],\rho)$. Then,
 ${\mathcal N}=\underline{\omega} + \pi $ is a co-boundary Nijenhuis form, with curvature, with respect to the multiplicative graded Lie algebra structure $l_2^{\left[.,.\right]}$ on the graded vector space $\Gamma(\wedge A)[2]$,
with square $\underline{N} $, where $N = \pi^{\#} \circ \omega^\flat$.
The deformed structure is $-l_1^{\left[.,.\right],\pi}$.
\end{prop}
\begin{proof}
Observe that
\begin{equation*}
l_1^{\left[.,.\right],\pi}(P)=\left[\pi,P\right]_{_{SN}}=-l_2^{\left[.,.\right]}(\pi,P)=-\left[\pi,l_2^{\left[.,.\right]}\right]_{_{RN}}(P)
\end{equation*}
for all $P\in \Gamma(\wedge^2 A)$. This means that
\begin{equation}\label{chetori}
l_1^{\left[.,.\right],\pi}=-\left[\pi,l_2^{\left[.,.\right]}\right]_{_{RN}}.
\end{equation}
Hence,
\begin{equation} \label{samandoon}
\left[{\mathcal N},l_2^{\left[.,.\right]}\right]_{_{RN}} = -l_1^{\left[.,.\right],\pi}+ \underline{ \diff^A \omega} = -l_1^{\left[.,.\right],\pi},
\end{equation}
which proves the last claim (and proves that ${\mathcal N}$ is weak Nijenhuis vector valued form with respect to $l_2^{\left[.,.\right]}$, since
 $ l_1^{\left[.,.\right],\pi}$ is an $L_{\infty}$-structure on $\Gamma(\wedge A)[2]$).
Equations (\ref{samandoon}) and (\ref{chetori}) imply that
\begin{equation*}
\begin{array}{rcl}
\left[\mathcal{N},\left[\mathcal{N},l_2^{\left[.,.\right]}\right]_{_{RN}}\right]_{_{RN}}
&=&-\left[\mathcal{N},l_1^{\left[.,.\right],\pi}\right]_{_{RN}}
=-\left[\underline{\omega},l_1^{\left[.,.\right],\pi}\right]_{_{RN}}-\left[\pi,\pi\right]_{_{SN}}\\
&=&\left[\underline{\omega},\left[\pi,l_2^{\left[.,.\right]}\right]_{_{RN}}\right]_{_{RN}}
=\left[\left[\underline{\omega},\pi\right]_{_{RN}},l_2^{\left[.,.\right]}\right]_{_{RN}}.
\end{array}
\end{equation*}
This shows that $\mathcal{N}$ is a co-boundary Nijenhuis vector valued form with respect to the graded Lie algebra structure $l_2^{\left[.,.\right]}$,
 on the graded vector space $\Gamma(\wedge A)[2]$, with square $\left[\underline{\omega},\pi\right]_{_{RN}}.$
A direct computation shows that  $[\pi,\underline{\omega}]_{_{RN}}= \underline{N}$ and completes the proof.
\end{proof}

  %%%%%%%%%%%%%%%%%%%%%%%%%%%%%%%%%%%%%%%%%%%%%%%%%%%%%%%%%%%%%%%%%%%%%%%%%%%%%%%%%%%%%%%%%%%%%%%%%%%%%%%%%%%%%%%%%%%%%%%%%%%%%%%%%%%%%%%%%%%%%%%%%%

\noindent {\bf Acknowledgments.} M. J. Azimi and J. M. Nunes da Costa acknowledge the support of CMUC-FCT (Portugal) and PEst-C/MAT/UI0324/2011
 through European program COMPETE/FEDER. M. J. Azimi acknowledges the support of FCT grant BD33810/2009.
%%%%%%%%%%%%%%%%%%%%%%%%%%%%%%%%%%%%%%%%%%%%%%%%%%%%%%%%%%%%%%%%%%%%%%%%%%%%%%%%%%%%%%%%%%%%%%%%%%%%%%%%%%%%%%%%%%%%%%%%%%%%%%%%%%%%%%%%%%%%%%%%%%

 \end{document}